\documentclass{article}

\usepackage{enumerate}
\usepackage{amssymb}
\usepackage{amsmath}
\usepackage{amsthm}
\usepackage{amsfonts}
\usepackage{epsfig}
\usepackage{color}

\usepackage[english]{babel}

\newcommand{\hypcadre}{{\bf Drift}}
\newcommand{\hypV}{{\bf Pot}}
\newcommand{\hyppoinc}[1]{{\bf Poinc(#1)}}
\newcommand{\hypspec}{{\bf min Spec}}
\newcommand{\hypspecbeta}[1]{{\bf Spec(#1)}}
\newcommand{\hypconv}{{\bf Conv}}

\newcommand{\hypdelta}{{\bf V}}

\newtheorem{defi}{Definition}
\newtheorem{example}[defi]{Example}
\newtheorem{prop}[defi]{Proposition}
\newtheorem{cor}[defi]{Corollary}
\newtheorem{lem}[defi]{Lemma}
\newtheorem{theo}[defi]{Theorem}
\newtheorem{rem}[defi]{Remark}
\newtheorem{hyp}{Assumption}

\newcommand{\dlam}{\partial_\lambda^0}
\newcommand{\rmd}{\mathrm d}
\newcommand{\E}{\mathbb E}
\newcommand{\eV}{e^{-V(x)}\rmd x}
\newcommand{\HH}{\mathbb H}
\newcommand{\LL}{\mathbb L}
\newcommand{\PP}{\mathbb P}
\newcommand{\QQ}{\mathbb Q}
\newcommand{\R}{\mathbb R}
\newcommand{\Rd}{{\R^d}}
\newcommand{\Tl}{\dlam\mathcal T_\lambda}
\newcommand{\W}{\mathbb{W}}

\newcommand{\essinf}{{\rm ess\, inf}}
\newcommand{\esssup}{{\rm ess\, sup}}

\renewcommand{\epsilon}{\varepsilon}
\renewcommand{\phi}{\varphi}

\author{Roland Assaraf\footnote{Laboratoire de Chimie Th\'eorique,
    CNRS-UMR 7616 et Universit\'e Pierre et Marie Curie, 75252 Paris
    Cedex, France, \texttt{assaraf@lct.jussieu.fr}},
Benjamin Jourdain\footnote{Universit\'e Paris-Est, CERMICS (ENPC),
  INRIA, F-77455 Marne-la-Vall\'ee, France, \texttt{jourdain@cermics.enpc.fr}},
Tony Leli\`evre\footnote{Universit\'e Paris-Est, CERMICS (ENPC), INRIA,
  F-77455 Marne-la-Vall\'ee, France, \texttt{lelievre@cermics.enpc.fr}},
Rapha\"el Roux\footnote{Laboratoire de Probabilit\'es et Mod\`eles Al\'eatoires, UMR
7599, UPMC, Case 188, 4 pl. Jussieu, F-75252 Paris Cedex 5, France,
\texttt{raphael.roux@upmc.fr}}
}

\title{Computation of sensitivities for the invariant
  measure of a parameter dependent diffusion}

\begin{document}

\maketitle

\abstract{We consider the solution to a stochastic differential equation
  with a drift function which depends smoothly on some real parameter~$\lambda$,
  and admitting a unique
  invariant measure for any value of~$\lambda$ around $\lambda=0$. 
  Our aim is to compute the derivative with
  respect to~$\lambda$ of averages with respect to the invariant
  measure, at $\lambda=0$. We analyze a numerical method which consists in simulating
  the process at $\lambda=0$ together with its derivative
  with respect to $\lambda$ on long time horizon.
  We give sufficient conditions implying uniform-in-time square integrability
  of this derivative. This allows in particular to compute
  efficiently the derivative with respect to~$\lambda$ of the mean of an
  observable through Monte Carlo simulations.

}

\section{Introduction}

We are interested in methods to compute the response of a Brownian
dynamics to an infinitesimal change of a parameter $\lambda \in \R$. More precisely, we
consider the dynamics in $\R^d$:
\begin{equation}\label{eq:EDS_lam}
\left\{
\begin{aligned}
\rmd X_t^\lambda&=F_\lambda(X_t^\lambda)\rmd t+\sqrt2\rmd W_t,\\
X_0^\lambda&=X_0,
\end{aligned}
\right.
\end{equation}
for~$\lambda \in \R$ close to~$0$, where~$(W_t)_{t\geq0}$ is a standard
$d$-dimensional Brownian motion independent of $X_0 \in \R^d$. Note that neither the initial
condition~$X_0$ nor the Brownian motion depend on~$\lambda$. The
family of vector fields~$F_\lambda:\R^d\to\R^d$ is indexed by a real
parameter~$\lambda$. We assume that when~$\lambda=0$, the vector field
derives from some potential energy~$V:\R^d\to\R$, namely
$$F_0=-\nabla V,$$
where~$\nabla$ denotes the gradient operator with respect to the space
variables. For $\lambda$ close to zero, one can think of~$(X_t^\lambda)_{t\geq0}$ as a physical system undergoing
a potential energy~$V$ to which one
applies an external force~$\lambda\dlam F_\lambda$. Here and in all the
following, the notation~$\dlam$ denotes the derivative with
respect to~$\lambda$ computed at~$\lambda=0$.

Concerning the potential $V$, we assume that the following assumption holds.
\begin{hyp}[\hypV]The function $V$ satisfies the following assumptions:
\begin{enumerate}[(i)]
\item $V:\R^d\to\R$ is a $\mathcal C^2$ function such that $x \mapsto
  \nabla^2 V(x)$ is locally Lipschitz.
\item $\displaystyle{\int_{\R^d} \eV = 1}$ and $\displaystyle{\int_{\R^d} |\nabla V|^2(x) \eV < \infty}$.
\item Pathwise existence and uniqueness hold for the process~$(X_t^0)_{t\geq0}$.
\end{enumerate}
\end{hyp}
Since $\nabla V$ is assumed to be locally Lipschitz, pathwise uniqueness is
automatically ensured. Pathwise existence is ensured for instance as soon as
there exists a finite constant $C$ such that
for all $x \in \R^d,$ $\nabla V(x) \cdot x \le C (1 + |x|^2)$.

At~$\lambda=0$, the dynamics~\eqref{eq:EDS_lam} is of the following
gradient form
\begin{equation}\label{eq:EDS}
\rmd X_t^0=-\nabla V(X_t^0)\rmd t+\sqrt2\rmd W_t.
\end{equation}
Under the above assumptions, it can be checked that~$\eV$ is the
unique invariant probability measure (see Lemma~\ref{lem:erggen} below) denoted in the following:
$$\rmd \pi_0 = \eV.$$
{Then, from} classical results in ergodic theory, for any~$f$
in~$\LL^1(\pi_0)$, almost surely,
\begin{equation}\label{eq:erg1}
\lim_{t \to \infty} \frac1t\int_0^tf(X^0_s)\rmd s = \int_\Rd f \rmd\pi_0.
\end{equation}

Let us now introduce the assumptions we need on the drift~$(F_\lambda)_{\lambda\in\R}$.
\begin{hyp}[\hypcadre]
There exists $\lambda_0 >0$ such that,
for all $\lambda \in [0,\lambda_0]$,
\begin{enumerate}[(i)]
\item The function $F_\lambda-F_0$ is bounded by~$C\lambda$
  for some constant~$C$ not depending on~$x$. Moreover,  as $\lambda
  \to 0^+$, $\frac{F_\lambda(x)-F_0(x)}{\lambda}$ converges locally
  uniformly for $x\in\R^d$ to some limit $\dlam F_\lambda$. Note
  that~$\dlam F_\lambda$ is bounded by~$C$.
\item The function~$x\mapsto F_\lambda(x)$ is locally Lipschitz.
  The function~$x\mapsto\dlam F_\lambda(x)$ is differentiable on~$\Rd$
  and~$\nabla \cdot\dlam F_\lambda$ is
  in~$\LL^2(\pi_0)$.
\end{enumerate}
\end{hyp}
Under these assumptions, we will show (see Lemma~\ref{lem:erggen} below) that the
dynamics~\eqref{eq:EDS_lam} is ergodic with respect to a probability
measure $\pi_\lambda$: for any $f \in \LL^1(\pi_\lambda)$, almost surely,
\begin{equation}\label{eq:erg2}
\lim_{t \to \infty} \frac1t\int_0^tf(X^\lambda_s)\rmd s = \int_\Rd f \rmd\pi_\lambda.
\end{equation}

The aim of this paper is to study the quantity: for a given observable $f$,
\begin{equation}\label{eq:but}
\lim_{\lambda\rightarrow0}
\frac{\int_\Rd f\rmd\pi_\lambda-\int_\Rd f\rmd\pi_0}\lambda
=\dlam\left(\int_\Rd f\rmd\pi_\lambda\right).
\end{equation}
In particular, we will exhibit sufficient conditions for the existence
of this derivative, derive various explicit formulae for this
quantity, and discuss numerical techniques in order to
approximate it.

The estimation of derivatives of the form~\eqref{eq:but} is useful
in various applications, in particular in molecular simulations (see
for example the recent work~\cite{warren-allen-12}): optimization procedure to fit a force field to
some observations, study of phase transitions,
estimate of forces in Variational Monte Carlo methods (see~\cite{ACK-11}), or computation of
transport coefficients. Transport coefficients are
computed as the ratio of
the magnitude of the response of the system submitted to a
perturbation in its steady-state to the magnitude of the
perturbation. These coefficients are
related to macroscopic properties of the system through fluctuation
dissipation theorems~\cite{chandler-87,evans-morriss-08}. Examples include the mobility or the thermal conductivity.

It is well-known that it is possible to approximate
$\dlam\left(\int_\Rd f\rmd\pi_\lambda\right)$ by considering a
simulation at $\lambda=0$. For example, the
celebrated Green-Kubo formula~\cite{chandler-87,evans-morriss-08}
writes (see Theorem~\ref{theo:GK} in Section~\ref{sect:GK}  for a
proof in our specific context):
$$    \dlam\int_\Rd f\rmd\pi_\lambda
    =\int_0^\infty\E_{\pi_0}\left[
      f(X_0)\left(\nabla V\cdot\dlam F_\lambda-\nabla\cdot \dlam F_\lambda\right)(X_s^0)
    \right]\rmd s$$
where the subscript $\pi_0$ indicates that the initial condition $X_0$
is distributed according to $\pi_0$.
The derivative $\dlam\left(\int_\Rd f\rmd\pi_\lambda\right)$ can thus
be approximated by considering infinite-time integrals of
auto-correlation functions for the stationary process at $\lambda=0$. This formula can be used to approximate
$\dlam\left(\int_\Rd f\rmd\pi_\lambda\right)$ numerically, which
requires at least in some cases to be careful when choosing the
truncation time in the integral, see for example~\cite{chen-zhang-li-10}. Let us also
mention another technique discussed in ~\cite{warren-allen-12}, based
on the use of Malliavin weights and the Bismut Elworthy Li formula
(see~\cite{bally-bavouzet-messaoud-07,bismut-84,elworthy-li-94}).

In this work, we are interested in so-called non-equilibrium molecular
dynamics (NEMD) methods
which consists in simulating two trajectories
with $\lambda=0$ and $\lambda=\varepsilon$ small, and then considering
the finite difference when $\varepsilon \to 0$ (see for
example~\cite{ciccotti-jacucci-75,ciccotti-kapral-sergi-05}):
$$\dlam\left(\int_\Rd f\rmd\pi_\lambda\right) \simeq 
\frac{\frac{1}{t} \int_0^t f(X^\varepsilon_s) \, ds - \frac{1}{t}
  \int_0^t f(X^0_s) \, ds}{\varepsilon} \text{ when $\varepsilon \to
  0$ and $t \to \infty$.}
$$
Note that the consistency of this estimate is based on the ergodic
properties~\eqref{eq:erg1} and~\eqref{eq:erg2}.
To reduce the variance of the computation, it is natural to use the
same driving Brownian motion for the two processes $(X^\varepsilon_s)_{s \ge
  0}$ and $(X^0_s)_{s \ge 0}$ (see~\cite{ciccotti-jacucci-75}) and we
therefore end up with the natural following estimate:
\begin{equation}\label{eq:estim}
\dlam\left(\int_\Rd f\rmd\pi_\lambda\right) \simeq
\frac1t\int_0^t\dlam (f(X_s^\lambda))\rmd s
\text{ when $t \to \infty$.}
\end{equation}
As will be shown below (see Proposition~\ref{proptt}), it is easy to simulate $\dlam
(f(X_t^\lambda))$ by using the formula 
$$\dlam
(f(X_t^\lambda))=T_t \cdot \nabla f(X^0_t)$$
where the so-called tangent vector $T_t \in \Rd$ is defined by
$$T_t=\dlam X^\lambda_t.$$
Indeed, the couple $(X^0_t, T_t)$ is a Markov process which satisfies the
following extended version of the stochastic differential
equation~\eqref{eq:EDS}:
\begin{equation}\label{eq:EDS_extended}
\left\{
\begin{aligned}
\rmd X_t^0&=-\nabla V(X_t^0)\rmd t+\sqrt2\rmd W_t \,,\\
\rmd T_t &= \left( \dlam F_\lambda(X_t^0)-\nabla^2V(X_t^0) T_t\right)
\rmd t\, ,
\end{aligned}
\right.
\end{equation}
with initial conditions $X^0_0=X_0$ and $T_0=0$ (since, we recall, $X_0$ does
not depend on $\lambda$). The estimate~\eqref{eq:estim} thus leads to
a practical numerical method to evaluate the derivative $\dlam \left(\int_\Rd f\rmd\pi_\lambda\right) $.

The main theoretical result of this paper consists in exhibiting
sufficient conditions such that the following equalities hold true
(see Theorem~\ref{theo:interversion}):
\begin{equation}\label{eq:CV_estim_1}
\dlam\left(\int_\Rd f\rmd\pi_\lambda\right) 
=
\lim_{t\rightarrow\infty}\frac1t\int_0^t\dlam (f(X_s^\lambda))\rmd s
~~~\mbox{ a.s.}
\end{equation}
and
\begin{equation}\label{eq:CV_estim_2}
\dlam\left(\int_\Rd
  f\rmd\pi_\lambda\right)=
\lim_{t\rightarrow\infty}\E\left[\dlam (f(X_t^\lambda))\right].
\end{equation}
Therefore, two natural estimators of $\dlam\left(\int_\Rd
  f\rmd\pi_\lambda\right)$ are 
\begin{equation}\label{eq:estim_1}
\frac1t\int_0^t\dlam (f(X_s^\lambda))\rmd s
\end{equation}
and
\begin{equation}\label{eq:estim_2}
\E\left[\dlam (f(X_t^\lambda))\right].
\end{equation}
The second estimator is derived from the expected ergodic property on
the time marginals: $\lim_{t \to \infty} \E(f(X^\lambda_t))=\int_\Rd f
\rmd \pi_\lambda$. For both estimators, 
Theorem~\ref{theo:interversion} can be seen as a rigorous
justification of the
interversion of the derivative $\dlam$ with the limit $\lim_{t \to
  \infty}$ and an average in time for the first estimator and over the
underlying probability space for the second one, since $\dlam\left(\int_\Rd
  f\rmd\pi_\lambda\right)=\dlam\left(\lim_{t\rightarrow\infty}\frac1t\int_0^tf(X_s^\lambda)\rmd
  s\right)$ and $\dlam\left(\int_\Rd
  f\rmd\pi_\lambda\right)=\dlam\left(\lim_{t\to\infty}\E\left[f(X_t^\lambda)\right]\right)$. In
addition, we also study the variance of the random variable
$\dlam(f(X^0_t))=T_t\cdot \nabla f(X^0_t)$ which influences the
statistical errors associated with the two
estimators~\eqref{eq:estim_1} and~\eqref{eq:estim_2}.

The proof of~\eqref{eq:CV_estim_1} and~\eqref{eq:CV_estim_2} is based
on two main ideas. First, the long-time limit (in law) of the couple
$(X^0_t,T_t)$ is identified using a time-reversal argument (see
Lemma~\ref{lem:conv_loi}) in
the spirit of the argument used in~\cite{fontbona-jourdain-14} to study the long-time behavior of two
interacting stochastic vortices. We are then able to identify the
long-time limit of the estimators using the Green-Kubo formula which we prove in our setting in
Section~\ref{sect:GK}. Second, the justification of the interversion of the
derivative with the long-time limit and the integrals requires some integrability results, which are based on the study
of the long-time behaviour of
$\E\left[e^{-\int_0^t\phi(Y_s^x)\rmd s}\right]$ for $\phi=\min
{\rm Spec}(\nabla^2 V)$, where $(Y^x_s)_{ s \ge 0}$
satisfies~\eqref{eq:EDS} with $x$ as an initial condition:
\begin{equation}\label{eq:Ytx}
\left\{
\begin{aligned}
\rmd Y_t^x&= - \nabla V(Y_t^x) \, \rmd t  + \sqrt{2} \rmd W_t,\\
Y_0^x&=x.\\
\end{aligned}
\right.
\end{equation}
Let us emphasize that we prove all these results in a rather general
setting: the state space is non compact (namely $\Rd$), the
coefficients are only assumed to be locally
Lipschitz and the potential $V$ is not necessarily strictly convex.
The study of the long-time behaviour of the couple $(X^0_t,T_t)$ is
very much related to the study of the long-time behaviour of the
couple $(Y^x_t,DY^x_t)$, (see Lemma~\ref{lemderci}) which may be also
useful to analyze other related
numerical methods, see~\cite{tailleur-kurchan-07}.

The paper is organized as follows. In Section~\ref{sect:poincare}, we
give preliminary results on the stochastic differential
equations~\eqref{eq:EDS_lam} and~\eqref{eq:EDS}, in particular on their
ergodic properties and the long-time behaviour of the associated
Kolmogorov equations. In Section~\ref{sect:vecteur_tangent}, we then
introduce the tangent vector~$T_t$ and study its integrability. In
Section~\ref{sect:GK}, we derive and prove finite-time and
infinite-time Green-Kubo formulae. We are then in position to prove
the long-time convergence of the estimators~\eqref{eq:estim_1}
and~\eqref{eq:estim_2} in Section~\ref{sec:main_result}. Finally, the theoretical results are illustrated
through various numerical experiments in Section~\ref{sect:numeric}.

In all the following, we assume that Assumptions~{\bf (\hypV)}
and~{\bf (\hypcadre)} hold, and we do not mention them
explicitly in the statements of the mathematical results.

\section{Preliminary results on~\eqref{eq:EDS_lam} and~\eqref{eq:EDS} and the associated Kolmogorov equations}\label{sect:poincare}

In this section, we introduce partial differential equations related
to the stochastic differential equations~\eqref{eq:EDS_lam} and~\eqref{eq:EDS}, and study their
long-time behaviors. These preliminary results will be crucial to analyze
the numerical methods aimed at evaluating $\dlam\left(\int_\Rd
  f\rmd\pi_\lambda\right)$ that we study.

Let us introduce a few notations. For any positive Borel measure~$\mu$ on~$\R^d$, we denote by~$\LL^2(\mu)$ the space of
real valued measurable functions on~$\R^d$ which are square integrable with respect to~$\mu$.
We denote by~$\LL_0^2(\eV)$ the space of zero-mean square
integrable functions:
\begin{equation}\label{eq:L20}
\LL^2_0\left(\eV\right)
=\left\{f\in\LL^2\left(\eV\right),~\int_\Rd f(x)\eV=0\right\},
\end{equation}
and by~$\HH^1(\eV)$ the first order Sobolev space associated with the
measure~$\eV$:
\[
\HH^1\left(\eV\right)
=\left\{f\in\LL^2\left(\eV\right),~\nabla f\in\LL^2\left(\eV\right)\right\},
\]
where~$\nabla f$ is to be understood in the distributional sense.

\subsection{About the solution to~\eqref{eq:EDS_lam} and the regularity of the law of $X^\lambda_t$}
 
Let us first start by an existence and uniqueness result for the
process $(X^\lambda_t)_{t \ge 0}$ solution to~\eqref{eq:EDS_lam}.
\begin{lem}\label{lem:EDS_lam}
There exists a unique strong solution to~\eqref{eq:EDS_lam}.
\end{lem}
\begin{proof}
The proof is rather standard. The existence of a weak solution
to~\eqref{eq:EDS_lam} is obtained thanks to the Girsanov theorem and
the existence assumption for the process $(X^0_t)_{t \ge
  0}$ (see Assumption~{\bf (\hypV)}-(iii)). Indeed,
\begin{align*}
\rmd X^0_t
&=-\nabla V(X^0_t) \, \rmd t + \sqrt{2} \, \rmd W_t\\
&=F_\lambda(X^0_t) \, \rmd t + \sqrt{2} \,\rmd \left( \int_0^t \frac{1}{\sqrt{2}}(F_0-F_\lambda)(X^0_s) \,
  \rmd s + W_t \right).
\end{align*}
Indeed, under the probability $\QQ$ such that, for all $t \ge 0$,
($({\mathcal F}_t)_{t \ge 0}$ being the natural filtration for
$(W_t)_{t \ge 0}$),
$$\frac{\rmd \QQ}{\rmd \PP}\Big|_{{\mathcal F}_t}=\exp\left(-\int_0^t
  \frac{1}{\sqrt{2}}(F_0-F_\lambda)(X^0_s) \rmd W_s - \frac{1}{4} \int_0^t
|F_0-F_\lambda|^2(X^0_s) \, \rmd s\right)
$$
the process $\widetilde{W}_t=\int_0^t \frac{1}{\sqrt{2}}(F_0-F_\lambda)(X^0_s) \,
  \rmd s + W_t$ is a Brownian motion and therefore the triple
  $(X^0_t,\widetilde{W}_t,\QQ)$ is a weak solution
  to~\eqref{eq:EDS_lam}. Note that thanks to the Assumption~{\bf (\hypcadre)}-(i), the Novikov conditions are
  satisfied which justifies the use of the Girsanov
  theorem~\cite[Section 3.5-D]{karatzas-shreve-88}.

Moreover, it is standard to check that trajectorial uniqueness holds
for the stochastic differential equation~\eqref{eq:EDS_lam}, since
from Assumption~{\bf (\hypcadre)}-(ii),
$x \mapsto F_\lambda(x)$ is locally Lipschitz (see~\cite[Section 5.2-B, Theorem 2.5]{karatzas-shreve-88}).
 As a consequence, by the Yamada-Watanabe theorem (see for
example~\cite[Section 5.3-D]{karatzas-shreve-88}), the  stochastic differential
equation~\eqref{eq:EDS_lam}  admits a unique strong solution.
\end{proof}

In the sequel, we will need some results about the Radon-Nikodym
density of the distribution of the process with respect to the equilibrium
measure. These properties are given in the two following Lemmas.

\begin{lem}\label{lem:p}
Whatever the choice of $X_0$, for each $\lambda\in[0,\lambda_0]$ and
$t>0$, $X^\lambda_t$ admits a positive density with respect to the
Lebesgue measure on $\R^d$.

Let us now consider $\lambda=0$ and $(Y_t^{x})_{t \ge 0}$ solution to~\eqref{eq:Ytx}.
For all $t > 0$, the law of $Y_t^{x}$ admits a density $y \mapsto p(t,x,y)$
with respect to the Lebesgue measure which satisfies the
reversibility property:
\begin{equation}\label{eq:rev}
\forall t > 0, e^{-V(x)}
p(t,x,y)=e^{-V(y)} p(t,y,x), \text{$\rmd x \otimes \rmd y$-a.e.}.
\end{equation}
\end{lem}
\begin{rem}
   A well-known corollary of~\eqref{eq:rev} is that if $X_0$ is
distributed according to $\pi_0$, then the process $(X^0_t)_{t \ge 0}$
solution to~\eqref{eq:EDS} is reversible: for any $t>0$,
$$ (X^0_s)_{s \in [0,t]} \text{ has the same law as } (X^0_{t-s})_{s \in [0,t]}.$$
\end{rem}
\begin{proof}
Let $\psi:\R^d\to\R$ be a bounded measurable function.
  By the Girsanov theorem, $$\E[\psi(X_t^\lambda)]
  =\E\left[
    \psi(X_0+\sqrt2W_t)
    e^{
      \frac1{\sqrt2}\int_0^tF_\lambda(X_0+\sqrt2W_s)\rmd W_s
      -\frac14\int_0^t|F_\lambda(X_0+\sqrt2W_s)|^2\rmd s
    }
  \right].$$Here, the assumptions of the Girsanov theorem are satisfied. Indeed,
according to~\cite[Theorem~2.1]{rydberg-97} these assumptions are
satisfied if global-in-time existence and uniqueness in law hold for
both Equation~\eqref{eq:EDS_lam} (see Lemma~\ref{lem:EDS_lam}) and its driftless counterpart
\[
\rmd Y_t=\sqrt2\rmd W_t,
\]
which is a mere Brownian motion. The first assertion of Lemma~\ref{lem:p}
is thus proved. For $\lambda=0$, we follow~\cite[page 91]{gihman-skorohod-72} to deduce that
  \begin{align}
  \E[\psi(X_t^0)]
  &=\E\left[
    \psi(X_0+\sqrt2W_t)
    e^{
      -\frac12V(X_0+\sqrt2W_t)
    }
    e^{ \frac12V(X_0)}
    e^{
        \frac14\int_0^t\left(2\Delta V-|\nabla V|^2\right)(X_0+\sqrt2W_s)\rmd s
      }
  \right]. \label{eq:girsanov}
\end{align}
Now, if one considers the
Brownian bridge: $$\forall s \in [0,t], \, B_s^{x,y}=x+ \sqrt{2}W_s + \frac{s}{t} \left( y - x -
  \sqrt{2}W_t\right)$$ one obtains by conditioning with respect to
$W_t$:
$$\E(\psi(Y_t^{x}))= \E \left( \psi(x+\sqrt{2} W_t)
  e^{- \frac12 V(x+\sqrt{2}W_t)} e^{\frac12 V(x)}   g(x,x+\sqrt{2}
  W_t)  \right)$$
where $g(x,y)=\E \left( e^{
        \frac14\int_0^t\left(2\Delta V-|\nabla
          V|^2\right)(B^{x,y}_s)\rmd s
      }\right)$.
This shows that $Y_t^{x}$ admits a density with respect to the
Lebesgue measure:
\begin{equation}\label{eq:pn}
p(t,x,y)=e^{- \frac12 V(y)} e^{\frac12 V(x)}   g (x,y)
\frac{e^{-\frac{|x-y|^2}{4t}}}{(4 \pi t)^{d/2}}.
\end{equation}
{F}rom the formula~\eqref{eq:pn}, it is straightforward to check that
\begin{equation}\label{eq:rev_approx}
e^{-  V(x)} p(t,x,y) =  e^{-  V(y)} p(t,y,x)
\end{equation}
by using the fact that $g(x,y)=g(y,x)$ which is a direct consequence
of the fact that $(B^{x,y}_s)_{s \in [0,t]}$ has the same law as
$(B^{y,x}_{t-s})_{s \in [0,t]}$. This concludes the proof of Lemma~\ref{lem:p}.
\end{proof}

Let us now state a few additional results on
the dynamics when $\lambda=0$ and when $(X^0_t)_{t \geq 0}$ starts from a general random variable instead of a
deterministic point.
\begin{lem}\label{lem:densite_X_t}
Let~$X_0$ be distributed according to some probability measure~$\mu_0$, and
let~$(X_t^0)_{t\geq0}$ evolve according to
Equation~\eqref{eq:EDS}.  Denote by~$\mu_t$ the
distribution of the random variable~$X_t^0$.

For all $t>0$, $\mu_t$ has a density $r(t,\cdot )$  with
respect to $d\pi_0=e^{-V(x)} \, \rmd x$:
$$\mu_t(\rmd x) = r(t,x) e^{-V(x)} \, \rmd x.$$

Moreover, for $0 < s \le t$, for $\rmd x$-a.e. $x \in \R^d$,
\begin{equation}\label{eq:rtx}
r(t,x) = \E(r(s,Y_{t-s}^x))
\end{equation}
where $Y_t^x$ satisfies~\eqref{eq:Ytx}. Equation~\eqref{eq:rtx} holds
for $s=0$ if $\mu_0$ has a density $r(0,\cdot)$ with
respect to $d\pi_0=e^{-V(x)} \, \rmd x$.

If there exists $s \ge 0$ such that $\|r(s,\cdot)\|_{\LL^\infty} <
\infty$, then, for all $t \ge s$,
  \begin{equation}\label{eq:max_princ}
  \essinf_{x\in\Rd}r(s,x)\
  \leq\essinf_{x\in\Rd}r(t,x)
  \leq\esssup_{x\in\Rd}r(t,x)
  \leq\esssup_{x\in\Rd}r(s,x).
  \end{equation}
Finally, for any $q \in [1,\infty)$,  if there exists $s \ge 0$ such that $r(s,\cdot) \in \LL^q(\eV)$, then
for all $t \ge s$, $r(t,\cdot) \in \LL^q(\eV)$ and
\begin{equation}\label{eq:decroit_L2}
\forall t \ge s, \, \|r(t,\cdot)\|_{\LL^q(\eV)} \le \|r(s,\cdot)\|_{\LL^q(\eV)}.
\end{equation}
\end{lem}

\begin{proof}
Let $\psi:\R^d\to\R$ be a bounded measurable function. By
conditioning with respect to~$X_0$, and using the function $p(t,x,y)$
introduced in Lemma~\ref{lem:p}
\begin{align}
\E(\psi(X^0_t))
&= \int \int \psi(y) p(t,x,y) \, \rmd y \, {\rmd}\mu_0(x) \label{eq:densite_X_t}\\
&= \int \psi(y)  e^{-  V(y)} e^{  V(y)} \int p(t,x,y) {\rmd}\mu_0(x) \,
\rmd y. \nonumber
\end{align}
This shows that the law $\mu_t$ of $X^0_t$ is 
$r(t,y) e^{-V(y)} \, \rmd y$ with
$$r(t,y)= e^{  V(y)} \int p(t,x,y) {\rmd}\mu_0(x).$$

Likewise, for any $s \in [0,t]$, by
conditioning with respect to $X^0_s$, it is easy to check that
$\E(\psi(X^0_t))=\int \psi(y)  e^{-  V(y)} e^{  V(y)} \int p(t-s,x,y)
{\rmd}\mu_s(x) \, \rmd y$. Now by taking $0 < s \le t$ and using the
reversibility property~\eqref{eq:rev}, we get
\begin{align*}
\E(\psi(X^0_t))
&=\int \psi(y)  e^{-  V(y)} \int e^{  V(y)} p(t-s,x,y)
{\rmd}\mu_s(x) \, \rmd y\\
&=\int \psi(y)  e^{-  V(y)} \int e^{  V(x)}  p(t-s,y,x)
r(s,x) \eV \, \rmd y.
\end{align*}
Since the law of $X^0_t$ is 
$r(t,y) e^{-V(y)} \, \rmd y$, this shows that, 
\begin{equation}\label{eq:rty}
\rmd y\text{-a.e.}, \, r(t,y)= \int p(t-s,y,x)   r(s,x) \, \rmd x = \E(r(s,Y_{t-s}^y)).
\end{equation}
This integral is well defined since $x \mapsto p(t-s,y,x)$ and $x
\mapsto r(s,x)$ are non negative measurable functions. This shows formula~\eqref{eq:rtx}.

The maximum principle~\eqref{eq:max_princ} is then a direct consequence from this
representation formula~\eqref{eq:rty}.

Finally, if $r(s,\cdot) \in \LL^q(\eV)$, then, for $t > s$, $r(t,\cdot) \in \LL^q(\eV)$ since (using the fact
that $x \mapsto p(t-s,y,x)$ is a probability density function and the
reversibility property~\eqref{eq:rev})
\begin{align*}
\int |r(t,y)|^q e^{-V(y)} \, \rmd y
&=
\int \left|\int p(t-s,y,x)   r(s,x) \, \rmd x\right|^q e^{-V(y)} \, \rmd y\\
&\le
\int \int p(t-s,y,x)  | r(s,x) |^q \, \rmd x \, e^{-V(y)} \, \rmd y\\
&=\int \int p(t-s,x,y) e^{-V(x)} | r(s,x) |^q \, \rmd x \, \rmd y\\
&=\int e^{-V(x)} | r(s,x) |^q \, \rmd x  < \infty.
\end{align*}

\end{proof}

\begin{rem}
In Appendix~\ref{sec:hypdelta}, we discuss a stronger 
assumption on $V$ under which we are able to get more precise bounds
on $p(t,x,y)$.
\end{rem}

\subsection{A Feynman-Kac formula and the Fokker-Planck equation}

For two measurable functions~$f:\R^d\to\R$ and~$\phi:\R^d\to\R$,
with~$\phi$ locally integrable with respect to the Lebesgue measure, consider the
Kolmogorov equation associated with the infinitesimal generator of the
stochastic differential equation~\eqref{eq:EDS_lam}:
\begin{equation}\label{eq:EDP_phi}
\left\{
\begin{aligned}
  \partial_tu(t,x)
  &=\Delta u(t,x)
  +F_\lambda(x)\cdot\nabla u(t,x)
  -\phi(x)u(t,x),\quad
  t>0, \, x\in\Rd,\\
  u(0,x)&=f(x), \quad
 x\in\Rd.
\end{aligned}
\right.
\end{equation}
In all this section, $\lambda$ is a fixed parameter in the interval $[0,\lambda_0]$.

In the following, we will consider solutions to
Equation~\eqref{eq:EDP_phi} in the following weak sense:
\begin{defi}\label{defi:solution_EDP}
Let~$u$ be a function
in the space
\[
\LL^\infty\left([0,T],\LL^2(\eV)\right)
\cap\LL^2\left([0,T],\LL^2(|\phi(x)|\eV)\right)
\cap\LL^2\left([0,T],\HH^1(\eV)\right)
\]
for any~$T>0$.
For~$f\in\LL^2(\eV)$, we
say that~$u$ is a weak solution to~\eqref{eq:EDP_phi} 
if~$u(0,\cdot)=f$  and for any function~$v$
in~$\HH^1\left(\eV\right)\cap\LL^2\left(|\phi(x)|\eV\right)$,
\begin{equation}\label{eq:solution_EDP}
\begin{aligned}
\frac{\rmd}{\rmd t}\int_\Rd u(t,x)v(x)\eV
&=-\int_\Rd\nabla u(t,x)\cdot\nabla v(x)\eV
-\int_\Rd\phi(x)u(t,x)v(x)\eV\\
&\quad + \int_\Rd (F_\lambda(x) -F_0(x))\cdot \nabla u(t,x) v(x)\eV
\end{aligned}
\end{equation}
in distributional sense.
\end{defi}
Note that the last term in~\eqref{eq:solution_EDP} is well defined
since for $\lambda \in [0,\lambda_0]$, from
Assumption~{\bf(\hypcadre)}-$(i)$, $\|F_\lambda-F_0\|_{\LL^\infty(\R^d)}\le C\lambda$.
Moreover, note that the condition~$u(0,\cdot)=f$ makes sense, since a
function~$u$ satisfying
\[
u\in\LL^2([0,T],\HH^1(\eV))
\mbox{ and }
\partial_tu\in\LL^2([0,T],\HH^{-1}(\eV))
\]
actually lies
in~$\mathcal C([0,T],\LL^2(\eV))$ (see for
example~\cite[Lemma~$1.2$ p.~$261$]{temam-79}).

\begin{prop}\label{prop:EDP_bien_posee}
Assume~$f\in\LL^2(\eV)$ and that the
function~$\phi$ is locally integrable with respect to the Lebesgue
measure and bounded from below. Then,
Equation~\eqref{eq:EDP_phi} admits a unique solution in
the sense of Definition~\ref{defi:solution_EDP}. Moreover, this
solution is in ${\mathcal C}([0,+\infty),\LL^2(\eV))$.

In addition, if $ f \in
\HH^1\left(\eV\right)\cap\LL^2\left(|\phi(x)|\eV\right)$, the solution
$u$ is more regular: for any~$T>0$,
\[u \in
\LL^\infty\left([0,T],\HH^1(\eV)
\cap\LL^2(|\phi(x)|\eV)\right)
\cap\HH^1\left([0,T],\LL^2(\eV)\right).
\]
\end{prop}
\begin{proof}
By Assumption~{\bf(\hypcadre)}-$(i)$, there exists $C_0>0$ such that
$\|F_\lambda-F_0\|_{\LL^\infty(\R^d)}\le C_0$. Let~$C$ be a positive constant such that~$\phi+C-\frac{C_0^2}{2}$ is nonnegative.
From~\cite[Lemma~$1.2$ p.~$261$]{temam-79}, one can take~$e^{-Ct}u(t,x)$
as a test function in~\eqref{eq:solution_EDP} and obtain the
following estimate
\begin{equation*}
\begin{aligned}
&\frac{\rmd}{\rmd t}\left( \frac{e^{-Ct}}2\int_\Rd|u(t,x)|^2\eV\right)\\
&= -e^{-Ct} \int_\Rd |\nabla u(t,x)|^2\eV
- e^{-Ct}\int_\Rd
|u(t,x)|^2(\phi(x)+C)\eV\\
&\quad + e^{-Ct} \int_{\Rd}  (F_\lambda-F_0) \cdot \nabla u(t,x) u(t,x)
\eV\\
&\le  -e^{-Ct} \int_\Rd |\nabla u(t,x)|^2\eV
- e^{-Ct}\int_\Rd
|u(t,x)|^2(\phi(x)+C)\eV\\
&\quad + \frac{e^{-Ct}}{2} \int_{\Rd} |\nabla u(t,x)|^2
\eV + \frac{e^{-Ct}}{2} C_0^2 \int_{\Rd}  |u(t,x)|^2
\eV\\
&\le -\frac{e^{-Ct}}{2} \int_\Rd |\nabla u(t,x)|^2\eV
- \int_\Rd
e^{-Ct}|u(t,x)|^2\left(\phi(x)+C-\frac{C_0^2 }{2}\right)\eV.
\end{aligned}
\end{equation*}
Therefore, by integrating in time, one obtains the following estimate:
\begin{equation}\label{eq:estimation_a_priori}
\begin{aligned}
&\frac{e^{-Ct}}2\int_\Rd|u(t,x)|^2\eV
+\frac12\int_0^t\int_\Rd e^{-Cs}|\nabla u(s,x)|^2\eV\, \rmd s
\\
&+ \int_0^t\int_\Rd e^{-Cs}|u(s,x)|^2\left(\phi(x)+C-\frac{C_0^2
  }{2}\right)\eV\, \rmd s=\frac12\int_\Rd|f(x)|^2\eV.
\end{aligned}
\end{equation}
{From} this estimate, the uniqueness result follows from linearity by taking~$f=0$
in~\eqref{eq:estimation_a_priori}. And thanks to this a priori estimate,
existence can be proved by using a Galerkin method on a
countable family of smooth functions dense
in~$\HH^1(\eV)\cap\LL^2(|\phi(x)|\eV)$, which exists since the
measure~$|\phi(x)|\eV$ is finite on compact sets (see for
example~\cite[chapter II, Theorem 3.5]{malliavin-95}). 
As explained above, the fact that  $u \in {\mathcal
  C}([0,T],\LL^2(\eV))$ is then a standard result, see for
example~\cite[Lemma~$1.2$ p.~$261$]{temam-79}.

In order to obtain the additional regularity, let us take~$\partial_t u(t,x)$
as a test function in~\eqref{eq:solution_EDP}:
\begin{equation*}
\begin{aligned}
 \int_\Rd|\partial_t u(t,x)|^2\eV
&= -\frac{1}{2} \frac{\rmd}{\rmd t} \int_\Rd |\nabla u(t,x)|^2\eV
- \frac{1}{2} \frac{\rmd}{\rmd t} \int_\Rd
|u(t,x)|^2\phi(x)\eV\\
&\quad + \int_{\Rd}  (F_\lambda-F_0) \cdot \nabla
u(t,x) \partial_t u(t,x)
\eV\\
&\le -\frac{1}{2} \frac{\rmd}{\rmd t} \int_\Rd |\nabla u(t,x)|^2\eV
- \frac{1}{2} \frac{\rmd}{\rmd t} \int_\Rd
|u(t,x)|^2\phi(x)\eV\\
&\quad + C_0^2 \int_{\Rd} |\nabla u(t,x)|^2
\eV + \frac{1}{4} \int_{\Rd}  |\partial_t u(t,x)|^2
\eV.
\end{aligned}
\end{equation*}
Therefore, for a constant $C_1>0$ such that $\varphi + C_1$ is nonnegative,
\begin{align*}
&\frac{1}{2} \frac{\rmd}{\rmd t} \int_\Rd |\nabla u(t,x)|^2\eV
+ \frac{1}{2} \frac{\rmd}{\rmd t} \int_\Rd
|u(t,x)|^2 (\phi(x) + C_1)\eV + \frac{3}{4} \int_\Rd|\partial_t u(t,x)|^2\eV\\
&\le C_0^2 \int_{\Rd} |\nabla u(t,x)|^2 + C_1  \int_\Rd
u(t,x) \partial_t u(t,x) \eV \\
& \le C_0^2 \int_{\Rd} |\nabla u(t,x)|^2 + C_1^2  \int_\Rd
|u(t,x)|^2 \eV  + \frac{1}{4}   \int_\Rd
|\partial_t u(t,x)|^2 \eV.
\end{align*}
Using Gr\"onwall's Lemma, one
obtains the estimate after integration in time:
\begin{align*}
& e^{-2C_0^2 t} \int_\Rd |\nabla u(t,x)|^2\eV
+ e^{-2C_0^2 t} \int_\Rd
|u(t,x)|^2 (\phi(x) + C_1)\eV \\
&+  \int_0^t \int_\Rd e^{-2C_0^2 s} |\partial_t
  u(s,x)|^2\eV \, \rmd s \le \int_\Rd |\nabla f(x)|^2\eV \\
&+\int_\Rd
|f(x)|^2 (\phi(x) + C_1)\eV+ 2C_1^2 \int_0^t \int_\Rd e^{-2C_0^2 s} 
|u(s,x)|^2 \eV  \, \rmd s.
\end{align*}
The last term is bounded from above over finite time intervals by a constant times $\int_\Rd
|f(x)|^2 \eV$ thanks to~\eqref{eq:estimation_a_priori}.
Again, this a priori estimate can be made rigorous through a Galerkin
procedure, and yields the additional regularity stated in the
Proposition (see for example~\cite[Proposition
11.1.1]{quarteroni-valli-97} for a similar reasoning).
\end{proof}

\begin{prop}\label{prop:regularite_EDP_phi}
Let~$u$ be a solution to the partial differential
equation~\eqref{eq:EDP_phi} in the sense of
Definition~\ref{defi:solution_EDP}, and assume that the initial
condition~$f$ is of class~$\mathcal C^2$ with locally Lipschitz second order derivatives.
If~$\phi$ is locally Lipschitz,
then~$u$,~$\partial_tu$,~$\nabla u$ and~$\nabla^2u$ are continuous
functions and $u$ is a classical solution to~\eqref{eq:EDP_phi}.
\end{prop}
\begin{proof}
We use a bootstrap argument based on $\LL^p_t\LL^q_x$ regularity results for parabolic partial
differential equations. In order to apply standard results which
require $0$ as an initial condition, we
consider $v=u-f$, which satisfies (in the weak sense, see Definition~\ref{defi:solution_EDP}) the partial differential
equation:
\begin{equation}\label{eq:EDP_v}
\left\{
\begin{aligned}
  \partial_t v (t,x)
  &=\Delta v (t,x)
  +F_\lambda(x)\cdot\nabla v(t,x)
  -\phi(x)v(t,x) + g(x),\quad
  t>0, \, x\in\Rd,\\
  v(0,x)&=0, \quad
 x\in\Rd,
\end{aligned}
\right.
\end{equation}
where
$$g(x) = \Delta f (x)
  + F_\lambda(x)\cdot\nabla f(x)
  -\phi(x)f(x) $$
is a locally
Lipschitz function.

In this proof, we use the following notation
\[
\LL^p_t\LL^q_x=\bigcap_{T>0}\LL^p([0,T],\LL^q(\Rd)),
\]
where~$\LL^q(\Rd)$ is the~$\LL^q$ space associated with the Lebesgue measure.
We will also use the notations~$\LL_t^p\W_x^{s,p}$ where~$\W$ stands for the usual Sobolev space. We moreover introduce the notation
\[
\LL^{\infty-}=\bigcap_{2\leq q<\infty}\LL^q.
\]
Last, we set
\[
\W^{1,p}_t\LL_x^q
=\{
u\in\LL^p_t\LL_x^q,~
\partial_tu\in\LL^p_t\LL_x^q
\}.
\]

Let~$\chi$ be some function in the space~$\mathcal C^\infty_0$ of smooth,
compactly supported functions on~$\Rd$. The function~$\chi v$ satisfies,
in the weak sense, the equation
\[
\left\{
\begin{aligned}
\partial_t(\chi v)
-\Delta(\chi v)
&=\Phi^\chi\mbox{ on }(0,+\infty)\times\R^d,\\
(\chi v)(0,x) &=0, \quad x\in\Rd,
\end{aligned}
\right.
\]
where
$$\Phi^\chi= (\chi F_\lambda
-2\nabla\chi)\cdot\nabla v
-(\Delta \chi+\chi\phi)v + \chi g.$$
{From} parabolic regularity results, see for example~\cite[Theorem III.1]{von-wahl-82},
 one has the implication:
\begin{equation}\label{eq:implic1}
(\Phi^\chi\in\LL^2_t\LL^p_x)
\Rightarrow
(\chi v\in\LL^2_t\W^{2,p}_x\cap\W^{1,2}_t\LL^p_x).
\end{equation}
In addition, {from} the definition of~$\Phi^\chi$, one has
\begin{equation}\label{eq:implic2}
(\forall\chi\in\mathcal C^\infty_0,~\chi v\in\LL^2_t\W^{2,p}_x)
\Rightarrow
(\forall\chi\in\mathcal C^\infty_0,~\Phi^\chi\in\LL^2_t\W^{1,p}_x)
\end{equation}
since $\phi$ and $F_\lambda$ are locally Lipschitz functions (see Assumption~{\bf (\hypcadre)}-(ii)).

Now, by Definition~\ref{defi:solution_EDP}, the function~$v$ lies
in~$\LL^\infty([0,T],\LL^2(\eV))
\cap\LL^2([0,T],\HH^1(\eV))$, so that~$\Phi^\chi$ lies
in~$\LL^2_t\LL^2_x$, for any function~$\chi\in\mathcal C^\infty_0$.
{First} assume~$d>1$, and let~$\bar p$ be the supremum of all those~$p$ such that~$\Phi^\chi$ lies
in~$\LL^2_t\LL^p_x$ for all~$\chi$. Assume~$\bar p<\infty$. 
Since
\[
\frac{d\bar p}{d+\bar p}<\min(\bar p,d),
\]
one can find
some~$p\in\left(\frac{d\bar p}{d+\bar p},d\right)$, such
that~$\Phi^\chi$ belongs to~$\LL^2_t\LL^p_x$ for any~$\chi$ (note
that~$\frac{d\bar p}{d+\bar p}\geq1$, since~$\frac1d+\frac1{\bar p}\leq\frac12+\frac12=1$).
{From}~\eqref{eq:implic1}~$\chi v$ lies in~$\LL^2_t\W^{2,p}_x$, and
hence from~\eqref{eq:implic2},~$\Phi^\chi$ lies
in~$\LL^2_t\W^{1,p}_x$. However, Sobolev embeddings yield
\[
\Phi^\chi\in\LL^2_t\W^{1,p}_x\subset\LL^2_t\LL^{\frac{dp}{d-p}}_x,\mbox{ where
}\frac{dp}{d-p}>\bar p\mbox{ since }p>\frac{d\bar p}{d+\bar p},
\]
which contradicts the definition of~$\bar p$.
As a conclusion,~$\Phi^\chi$ lies in~$\LL^2_t\LL^{\infty-}_x$ for any~$\chi$.
In the case~$d=1$, one can directly deduce
from~$\Phi^\chi\in\LL^2_t\LL^2_x$ that~$\chi v$ is
in~$\LL^2_t\W^{2,2}_x\subset\LL^2_t\W^{1,\infty-}_x$ for any $\chi$, and thus~$\Phi^\chi\in\LL^2_t\LL^{\infty-}_x$ for any $\chi$.
In any case,~$\chi v$ lies in~$\LL^2_t\W^{2,\infty-}_x$ for any~$\chi$.

Now consider the equation satisfied by~$\chi\partial_i v$, for any
coordinate~$i$.
One obtains
\begin{equation}\label{eq:chaleur_chi_di_u}
\left\{
\begin{aligned}
\partial_t(\chi\partial_i v)
-\Delta(\chi\partial_i v)
&=\Psi^\chi\mbox{ on }(0,+\infty)\times\R^d,\\
(\chi\partial_i v) (0,x) &= 0, \quad
  x\in\Rd,
\end{aligned}
\right.
\end{equation}
where
\begin{align*}
\Psi^\chi
&=(\chi F_\lambda - 2\nabla\chi
)\cdot\nabla(\partial_iv)+\left(\chi \partial_i F_\lambda
-(\Delta\chi+\chi\phi)e_i\right)\cdot\nabla v-(\chi\partial_i\phi)v + \chi \partial_i g .\nonumber
\end{align*}
Since~$\chi v$ lies in~$\LL^2_t\W^{2,\infty-}_x$ for
any~$\chi\in\mathcal C^\infty_0$, the function~$\Psi^\chi$ is
in~$\LL^2_t\LL^{\infty-}_x$, from the boundedness
of~$\chi\partial_i\phi$, $\chi\partial_i g$ and~$\chi\partial_i F_\lambda$. Then, parabolic
regularity~\eqref{eq:implic1} for the heat
equation~\eqref{eq:chaleur_chi_di_u} implies
\[
\forall i,~\forall\chi\in\mathcal C^\infty_0,
~\chi\partial_iv\in\LL^2_t\W^{2,\infty-}_x\cap\W^{1,2}_t\LL^{\infty-}_x.
\]
In particular, for any~$\chi\in\mathcal C^\infty_0$, $\chi v$ is
in~$\W^{1,2}_t\W_x^{1,\infty-}$.
{From} Sobolev embeddings, we
deduce that~$\chi v$ lies
in~$\mathcal C^{1/2}_t\mathcal C^{1-\epsilon}_x$, for any~$\epsilon$
in~$(0,1)$ ($\mathcal C^s$ stand for H\"older spaces).
{From} the H\"older regularity of the initial condition, H\"older regularity theory for the
heat equation now yields the desired regularity on~$v$ (and thus on $u$), see for
example~\cite[Theorem 10.3.3]{krylov-96}.
\end{proof}

\begin{prop}\label{prop:feynman-kac}
Assume $f\in\LL^2(\eV)$ and that $\phi$ is locally Lipschitz
and bounded from below. Then the solution $u(t,x)$  to the
partial differential equation~\eqref{eq:EDP_phi} given by Proposition~\ref{prop:EDP_bien_posee} admits
the following probabilistic representation formula: for all $t \ge 0$,
\begin{equation}\label{eq:feynman-kac}
\rmd x\text{-a.e.}, \, u(t,x)
=\E\left[f(Y_t^{\lambda,x})e^{-\int_0^t\phi(Y_s^{\lambda,x})\rmd s}\right]
\end{equation}
where $(Y_t^{\lambda,x})_{t \ge 0}$ is defined by (notice that
$(Y_t^{0,x})_{t \ge 0}=(Y_t^{x})_{t \ge 0}$ is defined by~\eqref{eq:Ytx})
\begin{equation}\label{eq:Ytlambdax}
\left\{
\begin{aligned}
\rmd Y_t^{\lambda,x}&= F_\lambda(Y_t^{\lambda,x}) \, \rmd t  + \sqrt{2} \rmd W_t,\\
Y_0^{\lambda,x}&=x.\\
\end{aligned}
\right.
 \end{equation}
\end{prop}
\begin{proof}
\underline{\em Step 1}:
Let us first prove a maximum principle for solutions to
Equation~\eqref{eq:EDP_phi} in the sense of Definition~\ref{defi:solution_EDP}.
Assume that the initial condition~$f$ of Equation~\eqref{eq:EDP_phi} is
bounded from above by some nonnegative constant~$M$. Let~$C$ be a 
constant such that~$\phi+C$ is nonnegative.
{From}~\cite[Lemma~$1.2$ p.~$261$]{temam-79}, and~\cite[Lemma~7.6]{gilbarg-trudinger-01}, one can 
take~$e^{-Ct}(e^{-Ct}u(t,x)-M)^+$ as the
test function in the weak formulation~\eqref{eq:solution_EDP}, and
obtain (using the fact that from Assumption~{\bf(\hypcadre)}-$(i)$, 
$\|F_\lambda-F_0\|_{\LL^\infty(\R^d)}\le C_0$ for some $C_0>0$)
\begin{align*}
&\frac12 \frac{\rmd}{\rmd t} \int_\Rd|(e^{-Ct}u(t,x)-M)^+|^2\eV\\
&=-\int_\Rd|\nabla (e^{-Ct}u(t,x)-M)^+|^2\eV\\
&\quad-\int_\Rd(\phi(x)+C)e^{-Ct}u(t,x)(e^{-Ct}u(t,x)-M)^+\eV\\
&\quad+\int_\Rd(F_\lambda(x)-F_0(x))\cdot \nabla (e^{-Ct}  u(t,x)
  -M)^+ (e^{-Ct}u(t,x)-M)^+ \eV \\
&\leq -\frac12 \int_\Rd|\nabla (e^{-Ct}u(t,x)-M)^+|^2\eV\\
&\quad-\int_\Rd(\phi(x)+C)e^{-Ct}u(t,x)(e^{-Ct}u(t,x)-M)^+\eV\\
&\quad+\frac{C_0^2}{2}\int_\Rd |(e^{-Ct}u(t,x)-M)^+|^2 \eV.
\end{align*}
By using Gr\"onwall's Lemma, one therefore obtain after integration in time:
\begin{align*}
&\frac{e^{-C_0^2t}}{2}\int_\Rd|(e^{-Ct}u(t,x)-M)^+|^2\eV\\
&\le -\frac12\int_0^t\int_\Rd e^{-C_0^2s}|\nabla
  (e^{-Cs}u(s,x)-M)^+|^2\eV\, \rmd s\\
&\quad-\int_0^t\int_\Rd e^{-C_0^2s}  (\phi(x)+C)e^{-Cs}u(s,x)(e^{-Cs}u(s,x)-M)^+\eV\,
  \rmd s\\
&\quad+\frac12 \int_\Rd|(f(x)-M)^+|^2\eV  \leq0
\end{align*}
so that the function~$u(t,\cdot)$ is bounded from above by~$Me^{Ct}$, for any
positive~$t$.
By a similar argument, if~$f$ is bounded from below by~$-M$, with~$M$ nonnegative,
then~$u(t,\cdot)$ is bounded from below by~$-e^{-Ct}M$ for any positive~$t$.

\underline{\em Step 2}:
Let us now prove the Feynman-Kac formula~\eqref{eq:feynman-kac} under the assumption~$f \in
\mathcal C^\infty \cap \LL^\infty(\R^d)$. Let~$x\in\Rd$,~$t>0$ and~$M>0$. Let~$\tau_M$
be the first exit time from $B(x,M)$ (namely the ball centered at $x$ and of radius $M$) for the
process~$(Y_s^{\lambda,x})_{s\geq0}$. Since~$s \mapsto Y_s^{\lambda,x}$ is
continuous,~$\tau_M$ goes to~$\infty$ as~$M$ goes to~$\infty$. Let us
consider the solution $(t,x) \mapsto u(t,x)$ to~\eqref{eq:EDP_phi}, which is ${\mathcal
  C}^1$ with respect to $t$ and ${\mathcal
  C}^2$ with respect to $x$ thanks to Proposition~\ref{prop:regularite_EDP_phi}.
Applying It\=o's formula to~$u(t-s,Y_s^{\lambda,x})$ in the time interval~$[0,t\wedge\tau_M]$, one obtains
\[
u((t-\tau_M)^+,Y_{t\wedge\tau_M}^{\lambda,x})e^{-\int_0^{t\wedge\tau_M}\phi(Y_s^{\lambda,x})\rmd s}
=u(t,x)
+\sqrt2\int_0^{t\wedge\tau_M}\nabla u(t-s,Y_s^{\lambda,x})e^{-\int_0^s\phi(Y_u^{\lambda,x})\rmd u}\rmd W_s.
\]
On the interval~$[0,t\wedge\tau_M]$, the integrand in the stochastic
integral remains bounded, so that this integral has zero mean. Taking the expectation, one obtains
\[
u(t,x)
=\E\left[u((t-\tau_M)^+,Y_{t\wedge\tau_M}^{\lambda,x})e^{-\int_0^{t\wedge\tau_M}\phi(Y_s^{\lambda,x})\rmd s}\right].
\]
By the above maximum principle
the function~$u$ is bounded on~$[0,t]\times\Rd$. With the lower bound
on~$\phi$, the dominated convergence theorem yields, letting~$M\to\infty$,
\[
u(t,x)
=\E\left[f(Y_t^{\lambda,x})e^{-\int_0^t\phi(Y_s^{\lambda,x})\rmd s}\right].
\]

\underline{\em Step 3}:
Let us now assume that~$f$ is in~$\LL^\infty(\R^d)$.
Let~$f_n$ be a sequence of $\mathcal C^\infty$ functions such that
$\sup_{n \ge 1} \|f_n\|_{\LL^\infty(\R^d)} \le \|f\|_{\LL^\infty(\R^d)}$, and converging
to~$f$ almost everywhere. In particular, $f_n$
converges to $f$ in $\LL^2(\eV)$ by Lebesgue's theorem. Therefore, the solution~$u_n$ to
Equation~\eqref{eq:EDP_phi} starting from~$f_n$ is such
that~$u_n(t,\cdot)$ converges to~$u(t,\cdot)$
as~$n\to\infty$ in~$\LL^2(\eV)$, from the a priori estimate~\eqref{eq:estimation_a_priori}.
Moreover, one has
\[
\forall x\in\Rd,~u_n(t,x)=\E\left[f_n(Y_t^{\lambda,x})e^{-\int_0^t\phi(Y_s^{\lambda,x})\rmd s}\right]
=\E[f_n(Y_t^{\lambda,x})\Gamma(Y_t^{\lambda,x})],
\]
where~$\Gamma$ is a bounded
function
satisfying~$\Gamma(Y_t^{\lambda,x})=\E\left[e^{-\int_0^t\phi(Y_s^{\lambda,x})\rmd
    s}\Big|Y_t^{\lambda,x}\right]$. For the remaining of the proof,
we assume that $t >0$ (the formula~\eqref{eq:feynman-kac} clearly holds for $t=0$).
{F}rom Lemma~\ref{lem:p}, the distribution of~$Y_t^{\lambda,x}$ admits a
density~$p^\lambda(t,x,y)$ with respect
to the Lebesgue measure. Therefore, by the Lebesgue theorem, $\E[f_n(Y_t^{\lambda,y})\Gamma(Y_t^{\lambda,y})]$
converges to~$\E[f(Y_t^{\lambda,y})\Gamma(Y_t^{\lambda,y})]$ as~$n\to\infty$. This shows
the equality~$u(t,x)=\E\left[f(Y_t^{\lambda,x})e^{-\int_0^t\phi(Y_s^{\lambda,x})\rmd
    s}\right]$ for $\rmd x$-a.e. $x$.

\underline{\em Step 4}:
Let us now assume that~$f$ is in~$\LL^2(\eV)$ and let us write
$f=f^+-f^-$ where $f^+=\max(f,0)$ and $f^-= \max(-f,0)$. For $n \in
\mathbb{N}$, the functions $f^+_n=\min(f^+,n)$
(resp. $f^-_n=\min(f^-,n)$) are in $\LL^\infty(\R^d)$ and converge in
$\LL^2(\eV)$ to $f^+$ (resp. $f^-$). Let us consider the solution~$u^+_n$
(resp. $u ^-_n$) to
Equation~\eqref{eq:EDP_phi} starting from~$f^+_n$
(resp. $f^-_n$). Since $f^+_n -f^-_n$ converges as~$n\to\infty$ to $f$ in
$\LL^2(\eV)$, for every $t \ge 0$, $u^+_n(t,\cdot) - u^-_n(t,\cdot)$
converges in~$\LL^2(\eV)$ to $u(t,\cdot)$, where $u$ is the solution to
Equation~\eqref{eq:EDP_phi} starting from~$f$.
Moreover, one has
\[
\rmd x \text{-a.e.},~u^\pm_n(t,x)=\E\left[f^\pm_n(Y_t^{\lambda,x})e^{-\int_0^t\phi(Y_s^{\lambda,x})\rmd s}\right]
=\E[f^\pm_n(Y_t^{\lambda,x})\Gamma(Y_t^{\lambda,x})],
\]
where~$\Gamma$ is the bounded
function defined above. By the monotone convergence theorem,
$\E[f^\pm_n(Y_t^{\lambda,x})\Gamma(Y_t^{\lambda,x})]$ converges to
$\E[f^\pm(Y_t^{\lambda,x})\Gamma(Y_t^{\lambda,x})]$. This shows
the equality~$u(t,x)=\E\left[f(Y_t^{\lambda,x})e^{-\int_0^t\phi(Y_s^{\lambda,x})\rmd
    s}\right]$ for $\rmd x$-a.e.~$x$.
\end{proof}

As a corollary of the previous result, we obtain that the law of
$X^0_t$ satisfies a partial differential equation (the Fokker-Planck equation).
\begin{cor}\label{cor:FP}
Let~$X_0$ be distributed according to some probability
measure~$\mu_0$, and let~$(X_t^0)_{t\geq0}$ evolve according to
Equation~\eqref{eq:EDS}. Let us assume that $\mu_0$ has a density
$r_0$ with respect to $d\pi_0=e^{-V(x)} \, \rmd x$ such that $r_0 \in
\LL^2(\eV)$. Denote by~$\mu_t$ the distribution of the random
variable~$X_t^0$, and by $x \mapsto r(t,x)$ the density of $\mu_t$ with
respect to $d\pi_0=e^{-V(x)} \, \rmd x$ which exists by Lemma \ref{lem:densite_X_t}.

Then, $r(t,x)$ is the unique solution to the partial differential equation
\begin{equation}\label{eq:FP}
\left\{
\begin{aligned}
  \partial_t r(t,x)
  &=\Delta r(t,x)
  -\nabla V(x)\cdot\nabla r(t,x), \quad
  t>0, x\in\Rd,\\
  r(0,x)&=r_0(x), \quad
  x\in\Rd,
\end{aligned}
\right.
\end{equation}
in the sense of Definition~\ref{defi:solution_EDP}.
\end{cor}
\begin{proof}
{F}rom Lemma~\ref{lem:densite_X_t}, we know that
$$\forall t \ge 0, \rmd x\text{-a.e.}, \,  r(t,x) = \E(r_0(Y_{t}^x)).$$
The conclusion is then a consequence of the Feynman-Kac representation formula~\eqref{eq:feynman-kac}.
\end{proof}

\subsection{Long-time behavior of the partial differential equation~\eqref{eq:EDP_phi} when $\lambda=0$}\label{sec:longtime}

In this section, we are interested in the long-time behavior of the
partial differential equation~\eqref{eq:EDP_phi} when $\lambda=0$, which is related to
the stochastic differential equation~\eqref{eq:EDS} through the
Feynman-Kac formula~\eqref{eq:feynman-kac}.

\subsubsection{The case $\phi =0$}

To study the long-time behavior of the solution to~\eqref{eq:EDP_phi}
with $\lambda=\phi= 0$, we introduce the following hypothesis (defined for any
$\eta >0$).
\begin{hyp}[\hyppoinc{$\eta$}]
The measure~$e^{-V}$ satisfies a Poincar\'e inequality with
constant~$\eta>0$: for any function~$v$
in~$\LL^2_0(\eV)\cap\HH^1(\eV)$,
\begin{equation}\label{eq:poincare}
\eta\int_\Rd |v|^2(x)\eV
\leq\int_\Rd |\nabla v|^2(x)\eV.
\end{equation}
\end{hyp}
Recall that $\LL^2_0(\eV)$ denotes the functions in $\LL^2(\eV)$ with
zero mean with respect to $\pi_0$ (see~\eqref{eq:L20}).

\begin{prop}\label{prop:tps_long}
Let Assumption~{\bf(\hyppoinc{$\eta$})} be satisfied for some positive
$\eta$, and let~$u$ be a
solution to~\eqref{eq:EDP_phi} in the sense of
Definition~\ref{defi:solution_EDP}, in the case~$\lambda=\phi= 0$, with an
initial condition $f \in \LL^2(\eV)$.
Then~$u$ converges exponentially fast to the constant
function~$\int_{\R^d} f(x) \eV$ in the following sense:
\[
\forall t \ge 0, \, \left\|u(t,\cdot)- \int_{\R^d} f(x) \eV \right\|_{\LL^2(\eV)}
\leq e^{-\eta t} \left\|f - \int_{\R^d} f(x) \eV \right\|_{\LL^2(\eV)}.
\]
\end{prop}
\begin{proof}
  Taking the constant function~$\mathbf1$ as the test function
  in~\eqref{eq:solution_EDP}, one obtains that $\int_\Rd u(t,x)\eV=\int_{\R^d} f(x) \eV$
  for any~$t\ge 0$. In
  particular,~$u(t,\cdot)-\int_{\R^d} f(x) \eV \in\LL_0^2(\eV)$. In addition,
  from~\cite[Lemma~$1.2$ p.~$261$]{temam-79}, one can take~$u(t,\cdot)-\int_{\R^d} f(x) \eV $ as the
  test function in~\eqref{eq:solution_EDP}, which yields, using the Poincar\'e inequality,
  \begin{align*}
  \frac12\frac{\rmd}{\rmd t}\int_\Rd \left|u(t,x)-\int_{\R^d} f(x) \eV \right|^2\eV
  &=-\int_\Rd|\nabla u(t,x)|^2\eV
\\ & \leq-\eta\int_\Rd \left|u(t,x)-\int_{\R^d} f(x) \eV \right|^2\eV.
  \end{align*}
  One concludes from Gr\"onwall's lemma.
\end{proof}

This Proposition shows that, under the assumption of
Corollary~\ref{cor:FP} (namely ${\rmd}\mu_0=r(0,x) \eV $ with $r(0,\cdot) \in \LL^2(\eV)$),
the density $r(t,x)$ of $X^0_t$ with respect to $\pi_0$ converges
exponentially fast to $1$ if {\bf(\hyppoinc{$\eta$})} is satisfied  for some positive
$\eta$. Actually,
the convergence of $\rmd \mu_t$ to $\eV$ holds in total variation norm for any
initial condition $\mu_0$.
\begin{cor}\label{cor:CV_L1}
Let Assumption~{\bf(\hyppoinc{$\eta$})} be satisfied  for some positive
$\eta$, and let~$(X_t^0)_{t\geq0}$ evolve according to
Equation~\eqref{eq:EDS}. Let us assume that $X_0$ is distributed according to some probability
measure~$\mu_0$. Denote by~$\mu_t$ the distribution of the random
variable~$X_t^0$, and for all $t>0$, denote by $q(t,x)$ the density of $\mu_t$ with
respect to the Lebesgue measure (which exists according to Lemma~\ref{lem:p}). Then
$$\lim_{t \to \infty} \|q(t,\cdot ) - e^{-V}\|_{\LL^1(\rmd x)} = 0.$$
\end{cor}
\begin{proof}
{F}rom Equation~\eqref{eq:densite_X_t}, for all $t>0$, one has
$$\rmd y\text{-a.e.,}\, q(t,y)= \int p(t,x,y) \mu_0(\rmd x).$$

Let us fix a positive $\epsilon$. Let us consider $t_0 >0$ (to be
fixed later on) and $q^\epsilon (t_0,x)$ a function in
$\LL^\infty(\R^d)$ which is non-negative, with compact support, such that $\int
q^\epsilon(t_0,x) \, \rmd x =1$ and 
$$\int_{\R^d} |q(t_0,x) - q^\epsilon(t_0,x)| \, \rmd x \le \epsilon.$$
To build such a function $q^\epsilon(t_0,\cdot)$, one could for
example consider for $n$ large enough $\frac{\min(q(t_0,x), n) 1_{|x| \le n}}{\int \min(q(t_0,x), n) 1_{|x| \le n}}$ which
indeed converges to $q(t,x)$ in $\LL^1(\rmd x)$ when $n \to \infty$.
Let us define
the function
$$\, q^\epsilon (t,y) = \int p(t-t_0,x,y)
q^\epsilon(t_0,x) \rmd x \qquad \forall t \ge t_0, \forall y \in \R^d.$$
For $t \ge t_0$, $q^\epsilon (t,\cdot)$ is the density
at time $t$ of the process $(X_t^{0,\epsilon})_{t \ge t_0}$ solution to~\eqref{eq:EDS}, with
$X_{t_0}^{0,\epsilon}$ distributed according to $q^\epsilon(t_0,x)
\rmd x$. Let us now set $r^\epsilon(t,x)= e^{V(x)} q^\epsilon (t,x)$, the density
of $X_t^{0,\epsilon}$ with respect to $\eV$.
Since $r^\epsilon(t_0,x) = e^{V(x)} q^\epsilon (t_0,x) \in \LL^2(\eV)$, from Corollary~\ref{cor:FP}, $(s,x) \mapsto
r^\epsilon(t_0+s,x)$ satisfies the following partial differential equation (with unknown $r$)
$$
\left\{
\begin{aligned}
  \partial_t r(s,x)
  &=\Delta r(s,x)
  -\nabla V(x)\cdot\nabla r(s,x)
  ,\quad s\ge 0, x\in\Rd,\\
  r(0,x)&=r^\epsilon(t_0,x),
  \quad x\in\Rd,
\end{aligned}
\right.
$$
in the sense of Definition~\ref{defi:solution_EDP}. In particular,
from Proposition~\ref{prop:tps_long}, since $\int r^\epsilon(t_0,x)
\eV =1$,
$$\forall t \ge t_0, \, \|r^\epsilon(t,\cdot)- 1\|_{\LL^2(\eV)}
\leq e^{-\eta (t-t_0)}\|r^\epsilon(t_0,\cdot) -1 \|_{\LL^2(\eV)}$$
which is equivalent to
$$\forall t \ge t_0, \, \|q^\epsilon(t,\cdot)- e^{-V}
\|_{\LL^2(e^{V(x)} \, \rmd x)}
\leq e^{-\eta (t-t_0)}\|q^\epsilon(t_0,\cdot) -
e^{-V}\|_{\LL^2(e^{V(x)} \, \rmd x)}.$$
By Cauchy-Schwarz inequality, we deduce that $\forall t \ge t_0$
\begin{align*}
\|q^\epsilon(t,\cdot)- e^{-V}
\|_{\LL^1(\rmd x)}&=
\int_{\R^d} \left|q^\epsilon(t,x)- e^{-V(x)} \right| e^{V(x)/2}
\times e^{-V(x)/2} \, \rmd x 
\\
&\le \|q^\epsilon(t,\cdot)- e^{-V}
\|_{\LL^2(e^{V(x)} \, \rmd x)}
\leq e^{-\eta (t-t_0)}\|q^\epsilon(t_0,\cdot) -
e^{-V}\|_{\LL^2(e^{V(x)} \, \rmd x)}.
\end{align*}

Moreover, we also have: $\forall t \ge t_0$
\begin{align*}
 \|q(t,\cdot)- q^\varepsilon(t,\cdot)\|_{\LL^1(\rmd x)}
&=
\int_{\R^d}  \left| \int_{\R^d} (q(t_0,x)-
  q^\varepsilon(t_0,x) ) p(t-t_0,x,y) \, \rmd x\right|\, \rmd y\\
& \le
\int_{\R^d} \int_{\R^d}  \left| q(t_0,x)-
  q^\varepsilon(t_0,x) \right| p(t-t_0,x,y) \, \rmd x\, \rmd y\\
&=\int_{\R^d}  \left| q(t_0,x)-
  q^\varepsilon(t_0,x) \right| \, \rmd x \le \epsilon.
\end{align*}
We thus obtain: $\forall t \ge t_0$,
\begin{align*}
\|q(t,\cdot)- e^{-V}\|_{\LL^1(\rmd x)}
&\le \|q(t,\cdot) - q^\varepsilon(t,\cdot)\|_{\LL^1(\rmd x)}
+ \|q^\varepsilon(t,\cdot)- e^{-V} \|_{\LL^1(\rmd x)} \\
& \leq \epsilon
+  e^{-\eta (t-t_0)}\|q^\epsilon(t_0,\cdot) - e^{-V}\|_{\LL^2(e^{V(x)} \, \rmd x)}
\end{align*}
and the right-hand side is smaller than $2 \epsilon$ for $t$
sufficiently large. This concludes the proof.
\end{proof}

\subsubsection{The case $\phi \neq 0$}

In this section, we are going to investigate the long-time behavior of
the function~$u$ defined by
\begin{equation}\label{eq:a_majorer}
u(t,x)=\E\left[e^{-\int_0^t\phi(Y_s^x)\rmd s}\right],
\end{equation}
for a generic function~$\phi$ where, we recall, $(Y^x_s)_{s \ge0}$
satisfies~\eqref{eq:Ytx}.
When~$\phi\geq\alpha$ for some positive constant~$\alpha$, $u$
converges to $0$ exponentially fast as $t\to\infty$. We now look for
hypotheses on $\phi$ under which this convergence is preserved in the
case $\inf \phi \le 0$. 

Notice that by ergodicity, almost
surely, 
$\lim_{t \to \infty} \frac 1 t \int_0^t \phi(Y^x_s) \, ds =  \int \phi
\eV$ and therefore, the almost sure exponential decay to zero is
ensured if $\int_\Rd \phi(x) \eV > 0$, at any rate in $(0,\int_\Rd
\phi(x) \eV)$. The exponential decay to zero in $L^1$ is more
complicated to establish. In Proposition~\ref{prop:conv_expo} below,
we prove this exponential decay under a sufficient condition which contains the assumption $\int_\Rd \phi(x) \eV > 0$.

When~$\phi$ is bounded from below and locally Lipschitz,
from Proposition~\ref{prop:feynman-kac}, the function~$u$
defined by~\eqref{eq:a_majorer} 
is solution (in the sense of Definition~\ref{defi:solution_EDP})
to the partial differential equation~\eqref{eq:EDP_phi}
with~$f = 1$ and $\lambda=0$.
As a consequence, the long-time behavior of~\eqref{eq:a_majorer} is related to the spectrum of the
operator~$\Delta-\nabla V\cdot\nabla-\phi$ which is self-adjoint in~$\LL^2(\eV)$.

One way to study this spectrum is to perform the change of
variable~$v(t,x)=e^{-\frac12V(x)}u(t,x)$, making
Equation~\eqref{eq:EDP_phi} become
\[
\partial_tv
=\Delta v
+\frac14\left(2\Delta V-|\nabla V|^2-4\phi\right)v.
\]
As a consequence, the long-time behavior of~$v$ is characterized by the spectrum
of the Schr\"odigner
operator~$\Delta+\frac14\left(2\Delta V-|\nabla V|^2-4\phi\right)$,
which can be controlled by the Cwikel-Lieb-Rozenblum bound (see for
example~\cite{cwikel-77,lieb-76,rozenbljum-72}). Indeed, for~$d\geq3$, this bound  states that the
number~$N$ of nonnegative eigenvalues of~$\Delta+W$ satisfies
\[
N\leq L_d\int_\Rd\max(W(x),0)^{d/2}\rmd x,
\]
where~$L_d$ is some constant independent of~$W$.
In particular, if there exists~$\epsilon>0$ such that
\begin{equation}\label{eq:critere_lieb_thirring}
\int_\Rd
\max\left(
\epsilon+\frac12\Delta V(x)-\frac14|\nabla V(x)|^2-\phi(x)
,0\right)^{d/2}
\rmd x
<\frac1{L_d},
\end{equation}
then the spectrum
of~$\Delta+\frac14\left(2\Delta V-|\nabla V|^2-4\phi\right)$
is included in~$(-\infty,-\epsilon)$.

There are two main concerns with this approach. First, the
constant~$L_d$ is unknown, so that the criterion is not quantitative.
Moreover, by Jensen's inequality, the
exponential convergence to~$0$ of
\[
\E\left[e^{-\delta\int_0^t\phi(X_s^0)\rmd s}\right]
\]
for~$\delta>1$
implies the exponential convergence of the function~$u(t,x)$
in~\eqref{eq:a_majorer}. However, in some
cases, the criterion~\eqref{eq:critere_lieb_thirring} may apply
to~$\delta\phi$ for some $\delta >1$ and not to~$\phi$. We are going to present another
criterion which does not present these flaws.

\begin{prop}\label{prop:conv_expo}
  Assume that~{\bf(\hyppoinc{$\eta$})} holds  for some positive
$\eta$, and that
  \begin{align}
    &-\infty<\inf\phi\leq 0,\;\;
    \int_\Rd\phi(x)\eV>0,\;\;
    \int_\Rd\phi(x)^2\eV<\infty,\label{eq:crit_conv1}\\
    \mbox{and }&-(\inf\phi)\frac{\int_\Rd\phi^2(x)\eV}{\left(\int_\Rd\phi(x)\eV\right)^2}
    <\eta.\label{eq:crit_conv2}
  \end{align}
  
  Let~$\mathbf E=\int_\Rd\phi(x)\eV$,~$\mathbf
  V=\int_\Rd\left(\phi(x)-\mathbf E\right)^2\eV$ and let~$u$ be the weak solution to
  Equation~\eqref{eq:EDP_phi} in the sense of
  Definition~\ref{defi:solution_EDP} for $\lambda=0$, with an initial condition $f \in \LL^2(\eV)$.
  The quantity $\int_\Rd u^2e^{-V}$ converges exponentially fast to~$0$ as $t\to\infty$:
  \[
  \exists C>0,~\forall t>0,~\int_\Rd u^2(t,x)\eV\leq Ce^{-\beta t},
  \]
  with a rate~$\beta$ given by
  \begin{equation}\label{eq:taux}
  \beta=\left(\eta+\inf\phi+\frac{\mathbf E^2+\mathbf V}{2\mathbf
      E}\right) -
  \sqrt{\left(\eta+\inf\phi-\frac{\mathbf E^2+\mathbf V}{2\mathbf E}\right)^2+2\frac{\eta\mathbf V}{\mathbf E}}
  >0.
  \end{equation}
\end{prop}

Note that the positivity of the rate~\eqref{eq:taux} is equivalent to
the condition~\eqref{eq:crit_conv2}.
Moreover, note that the left-hand side in condition~\eqref{eq:crit_conv2}
is homogeneous of order~$1$ in~$\phi$, unlike the
criterion~\eqref{eq:critere_lieb_thirring} obtained using
the Cwikel-Lieb-Rozenblum bound. As a consequence, if the
criterion~\eqref{eq:crit_conv2} applies to~$\delta\phi$ for some real
number~$\delta>1$, then it applies to~$\phi$, as expected.
\begin{proof}[Proof of Proposition \ref{prop:conv_expo}]
In this proof, for notational simplicity, we omit the time and space variables in the
integrals, which are all considered with respect to the Lebesgue measure
on~$\Rd$.

{From}~\cite[Lemma~$1.2$ p.~$261$]{temam-79}, one can take~$u$ as the test
function in~\eqref{eq:solution_EDP},
obtaining
\begin{align*}
  \frac12\frac{\rmd}{\rmd t}\int_\Rd u^2e^{-V}
  &=-\int_\Rd|\nabla u|^2e^{-V}-\int_\Rd \phi u^2 e^{-V}\\
  &\leq -\int_\Rd|\nabla u|^2e^{-V}-\inf\phi\int_\Rd u^2 e^{-V}.
\end{align*}
Using \eqref{eq:poincare},
one deduces that
\begin{equation}\label{eq:e1}
  \frac12\frac{\rmd}{\rmd t}\int_\Rd u^2e^{-V}
  \leq-(\eta+\inf\phi)\int_\Rd u^2e^{-V}
  +\eta\left(\int_\Rd ue^{-V}\right)^2.
\end{equation}
On the other hand, taking the constant function~$\mathbf1$ as the test
function in~\eqref{eq:solution_EDP},
\begin{align*}
\frac12\frac{\rmd}{\rmd t}\left(\int_\Rd u e^{-V}\right)^2
&=-\int_\Rd ue^{-V}\int_\Rd\phi ue^{-V}\\
&=-\int_\Rd \phi e^{-V}\left(\int_\Rd u e^{-V}\right)^2
+\int_\Rd ue^{-V}\left(\int_\Rd\phi e^{-V}\int_\Rd u e^{-V}-\int_\Rd\phi ue^{-V}\right).
\end{align*}
By Cauchy-Schwarz inequality,
\begin{align}\label{eq:e2}
\left|\int_\Rd\phi e^{-V}\int_\Rd u e^{-V}-\int_\Rd \phi ue^{-V}\right|
&=\left|\int_\Rd\left(\phi-\int_\Rd\phi e^{-V}\right) \left(u-\int_\Rd u e^{-V}\right)e^{-V}\right|\nonumber\\
&\le\mathbf V^{1/2}\left(\int_\Rd\left(u-\int_\Rd ue^{-V}\right)^2e^{-V}\right)^{1/2}.
\end{align}
Therefore, from the inequality~$2ab\leq a^2+b^2$,
\begin{align}\label{eq:e3}
\frac12\frac{\rmd}{\rmd t}\left(\int_\Rd u e^{-V}\right)^2
&\leq-\frac{\mathbf E}2\left(\int_\Rd u e^{-V}\right)^2
+\frac{\mathbf V}{2\mathbf E}
\left(\int_\Rd \left(u-\int_\Rd ue^{-V}\right)^2e^{-V}\right)\nonumber\\
&=-\frac{\mathbf E^2+\mathbf V}{2\mathbf E}\left(\int_\Rd u e^{-V}\right)^2
+\frac{\mathbf V}{2\mathbf E}\int_\Rd u^2e^{-V}.
\end{align}
By combining~\eqref{eq:e1} and~\eqref{eq:e3}, one obtains for $\delta\geq 0$,
\begin{align}\label{eq:majo_cle}
\frac12\frac{\rmd}{\rmd t}\left(  \int_\Rd u^2e^{-V} + \delta \left(\int_\Rd u
    e^{-V}\right)^2\right)
\leq-c_1(\delta)\int_\Rd u^2e^{-V}
-c_2(\delta)\left(\int_\Rd ue^{-V}\right)^2
\end{align}
with
\begin{equation}\label{eq:expr_c1_c2}
  c_1(\delta)
  =\eta
  +\inf\phi
  -\frac{\delta\mathbf V}{2\mathbf E}
  \mbox{ and }
  c_2(\delta)
  =-\eta
  +\frac\delta2\frac{\mathbf V+\mathbf E^2}{\mathbf E}.
\end{equation}
We want to find~$\delta\geq0$ such that \eqref{eq:majo_cle} ensures exponential
convergence to~$0$ of~$\int_\Rd u^2e^{-V}$ as~$t\to+\infty$. If~$\delta=0$,
one has $c_2(0)=-\eta<0$, so we need $\delta>0$. We look for~$\delta>0$
such that both~$c_1(\delta)$ and~$c_2(\delta)$ are positive.

{From}~\eqref{eq:expr_c1_c2},~$c_1(\delta)$ is positive if and only if
\[
\frac\delta{2\mathbf E}<\frac{\eta+\inf\phi}{\mathbf V}
\]
and~$c_2(\delta)$ is positive if and only if
\[
\frac\delta{2\mathbf E}>\frac\eta{\mathbf V+\mathbf E^2}.
\]
One concludes by checking that condition~\eqref{eq:crit_conv2} is
necessary and sufficient for the interval
~$\left(\frac\eta{\mathbf V+\mathbf E^2},\frac{\eta+\inf\phi}{\mathbf V}\right)$
to be nonempty.

{For} a given~$\delta>0$, Equation~\eqref{eq:majo_cle} gives a convergence rate
of~$2\min(c_1(\delta),\frac{c_2(\delta)}\delta)$.
{From} the definition of~$c_1$ and~$c_2$, one can see
that~$c_1(\delta)$ is nonincreasing and, under~\eqref{eq:crit_conv2},~$\frac{c_2(\delta)}\delta$ is
nondecreasing in~$\delta$. As a
consequence,~$\min(c_1(\delta),\frac{c_2(\delta)}\delta)$ is maximized
for~$\delta c_1(\delta)=c_2(\delta)$. This last equation is quadratic,
and one can check that its unique positive solution is
\[
\delta
=\frac{\mathbf E}{\mathbf V}
\left(
  \eta+\inf\phi-\frac{\mathbf V+\mathbf E^2}{2\mathbf E}+\sqrt{\left(\eta+\inf\phi-\frac{\mathbf V+\mathbf E^2}{2\mathbf E}\right)^2+2\frac{\eta \mathbf V}{\mathbf E}}
\right),
\]
giving the rate~\eqref{eq:taux}.

\end{proof}

One can see that Equation~\eqref{eq:crit_conv2} is necessary and sufficient
for the existence of~$\delta>0$ such that~$c_1(\delta)>0$ and~$c_2(\delta)>0$.
One can naturally wonder whether introducing more flexibility in
the inequalities used in the proof of Proposition~\ref{prop:conv_expo} could lead to a weaker
condition.
Actually, keeping track of the positive term~$\int_\Rd|\nabla u|^2e^{-V}$
in~\eqref{eq:e1}, using inequality~\eqref{eq:poincare} in~\eqref{eq:e2},
and using the inequality~$2ab\leq\gamma a^2+\frac1\gamma b^2$
in~\eqref{eq:e3} leads to the
exact same necessary and sufficient condition to ensure exponential
convergence to~$0$.

\begin{rem}
One can use the theory of large deviations to prove that the long-time
behavior of quantities of the form~\eqref{eq:a_majorer} is necessarily
exponential, with a rate given by a variational formula.

Let~$(X_t)_{t\geq0}$ evolve according to
Equation~\eqref{eq:EDS}. According to Donsker-Varadhan's lemma, the
random probability measure~$\mu_t$
defined by the formula
\[
\mu_t(A)
=\frac1t\int_0^t\mathbf 1_{X_s\in A}\rmd s,
\]
satisfies a large deviation principle with rate function
\[
I(\nu)
=\begin{cases}
\displaystyle\int_\Rd\left|\nabla\sqrt f\right|^2\eV&\mbox{ if } \exists f:\R^d
\to \R, \, \nu=f(x)\eV,\\
\infty&\mbox{ otherwise,}
\end{cases}
\]
see for example~\cite[Chapter~IV.4]{den-hollander-00}.
As a consequence, in
the long-time limit,
\[
-\frac1t\log\left(\E\left[e^{-\int_0^t\phi(X_s)\rmd s}\right]\right)
=-\frac1t\log\left(\E\left[e^{-t\left<\phi,\mu_t\right>}\right]\right)
\]
converges to the constant~$\alpha$ defined by
\[
\alpha=\inf_f
\int_\Rd\left|\nabla \sqrt f\right|^2(x)\eV
+\int_\Rd\phi(x)f(x)\eV
\]
where the infimum is taken over all probability densities with respect
to the measure~$\eV$, from Varadhan's lemma in large deviations theory.
By the change of variables~$g^2=f$,~$\alpha$
is equal to
\[
\inf_g
\frac{\int_\Rd\left|\nabla g\right|^2(x)\eV
+\int_\Rd\phi(x)g^2(x)\eV}
{\int_\Rd g^2(x)\eV},
\]
which is the bottom of the spectrum of the operator~$-\Delta+\nabla
V\cdot\nabla+\phi$ which is self-adjoint in~$\LL^2(\eV)$, already
discussed at the beginning of this section.
\end{rem}

Let us give two examples where the result from Proposition~\ref{prop:conv_expo}
applies.

\begin{example}
  A first example is given by the double-well potential in
  dimension~$1$, defined by
  \[
  V_\gamma(x)=x^4-\gamma x^2 + C_\gamma,
  \]
  where~$\gamma>0$ and $C_\gamma=\ln \left( \int_{\R} \exp(-x^4+\gamma x^2)\,
  \rmd x \right)$.
  We want to apply Proposition~\ref{prop:conv_expo} to the
  case where the function~$\phi$ is a multiple of~$\min{\rm
    Spec}\nabla^2V_\gamma(x)$ (see Section~\ref{sec:expo_cv_R} for a justification
  of this choice for $\phi$). In the present case,
  this writes~$\phi_{\gamma,\delta}(x)=\delta(12x^2-2\gamma)$, where
  $\delta$ is the positive multiplicative factor.
  
  Denote by~$\eta_{\gamma}$ the optimal Poincar\'e constant
  associated to the potential~$V_\gamma$. As~$\gamma$ goes to~$0$, the limit
  potential~$x^4$ satisfies a Poincar\'e inequality with
  constant~$\eta_0>0$, owing to its convexity. As a
  consequence, the Poincar\'e constants~$\eta_{\gamma}$ converge to a
  positive limit.
  On the other hand, as~$\gamma$ goes to~$0$, the quantity
  \[
  -(\inf\phi_{\gamma,\delta})
  \frac
  {\int_\R\phi_{\gamma,\delta}^2(x)e^{-V_\gamma(x)}\rmd x}
  {\left(\int_\R\phi_{\gamma,\delta}(x)e^{-V_\gamma(x)}\rmd x\right)^2}
  \]
  goes to~$0$, since~$\inf\phi_{\gamma,\delta}$ goes to~$0$
  while~$\frac
  {\int_\R\phi_{\gamma,\delta}^2(x)e^{-V_\gamma(x)}\rmd x}
  {\left(\int_\R\phi_{\gamma,\delta}(x)e^{-V_\gamma(x)}\rmd x\right)^2}$
  converges to some positive constant. As a consequence, for any~$\delta>0$ the
  inequality~\eqref{eq:crit_conv2} is satisfied for~$\gamma$ smaller than
  some critical value depending on $\delta$. Notice that the
  inequalities~\eqref{eq:crit_conv1} are satisfied for any $\gamma
  >0$ since, for any smooth potential $V:\R \mapsto \R$,  $\int_\R V''(x)\eV>0$ holds from a mere
integration by parts.
\end{example}

\begin{example}
  The second example is given
  by an identically vanishing potential~$V(x)=0$, with the equation
  considered on the one-dimensional torus, identified with the segment~$[0,2\pi]$ with periodic
  boundary conditions. The invariant measure is then the uniform measure
  on the torus. Consider the function~$\phi(x)=\sin(x)+\alpha$
  with~$\alpha\geq0$. In that cases, the mean value of~$\phi$ is given
  by~$\alpha$, and~$\phi(x)$ is not nonnegative for all values of $x$ as soon
  as~$\alpha<1$.
  
  In that case, the Poincar\'e constant is given
  by~$\eta=1$ and~$\inf\phi$ is given by~$\alpha-1$. As a
  consequence, Equation~\eqref{eq:crit_conv2} writes
  \[
  (1-\alpha)\frac{1+2\alpha^2}{2\alpha^2}<1
  \]
  The condition is thus satisfied if~$\alpha>\alpha_0$
  where~$\alpha_0$ is the unique real root of the equation
  \[
  \alpha^3+\frac12\alpha-\frac12=0,
  \]
  given by~$\alpha_0\simeq0.590$. As a consequence,
  for~$\alpha\in(\alpha_0,1)$, one has exponential convergence of~\eqref{eq:a_majorer} to zero
  while the function~$\phi$ is not uniformly positive.
\end{example}

\subsection{Existence and uniqueness of an invariant measure
  $\pi_\lambda$ for~\eqref{eq:EDS_lam}}\label{sec:inv_meas}

In all this section, we assume that
Assumption~{\bf(\hyppoinc{$\eta$})} holds for some positive $\eta$.
We would like to show that the stochastic differential
equation~\eqref{eq:EDS_lam} admits a unique invariant probability
measure that we denote in the following $\pi_\lambda$, and to give an
explicit formula for this measure. Of course, for $\lambda=0$, we have
$$\rmd \pi_0=\eV$$
and one result of this section is that it is the unique invariant
measure for~\eqref{eq:EDS}. We will use $\pi_0$ as a reference measure
to build functional spaces, and to construct the invariant measure
$\pi_\lambda$ by perturbative arguments, using the crucial assumption
on the boundedness of~$F_\lambda+\nabla
V=F_\lambda-F_0$ (see Assumption{\bf(\hypcadre)}-$(i)$): for $\lambda
\in [0,\lambda_0],$
$$\|F_\lambda-F_0\|_{\LL^\infty(\Rd)}\le C \lambda.$$

Let us begin with some notation.
We denote by~$\mathcal L_\lambda=F_\lambda\cdot\nabla+\Delta$ the generator of
the process~$(X_t^\lambda)_{t\geq0}$. In
particular,~$\mathcal L_0=-\nabla V\cdot\nabla+\Delta$. Also
denote
\begin{equation}\label{eq:Tlambda}
\mathcal T_\lambda
=\mathcal L_\lambda-\mathcal L_0
=(F_\lambda+\nabla V)\cdot\nabla
\end{equation}
The space $\LL_0^2(\eV)\cap\HH^1(\eV)$ endowed with the symmetric bilinear form\begin{equation}\label{eq:produit_scalaire}
(u,v)
\mapsto
\int_\Rd\nabla u(x)\cdot\nabla v(x)\eV
\end{equation} is a Hilbert space by Assumption~{\bf(\hyppoinc{$\eta$})}.
A consequence of
the Riesz theorem is that for any~$u\in\LL^2(\eV)$, there
exists a unique function~$v$ in~$\LL^2_0(\eV)\cap\HH^1(\eV)$ such that
\[
\forall w\in\LL^2_0(\eV)\cap\HH^1(\eV),~
\int_\Rd \nabla v(x)\cdot\nabla w(x)\eV=\int_\Rd u(x)w(x)\eV.
\]
This
function is denoted~$v=- \mathcal L_0^{-1}u$ since when $v$ is smooth,
$\int_\Rd \nabla v(x)\cdot\nabla w(x)\eV= - \int_\Rd \mathcal L_0v(x)w(x)\eV$. We denote
by~$\mathcal D(\mathcal L_0)$ the domain of~$\mathcal L_0$, defined by
\[
\mathcal D(\mathcal L_0)
=\left\{ v \in \LL^2_0(\eV)\cap\HH^1(\eV),\, \mathcal L_0 v  \in \LL_0^2(\eV)\right\}.
\]

For a function~$u \in \LL^2(\eV)$, from the Poincar\'e inequality, one has
\begin{align*}
\eta\|\mathcal L_0^{-1} u\|^2_{\LL^2(\eV)}
\leq\int_\Rd |\nabla (\mathcal L_0^{-1} u)(x)|^2\eV
&=-\int_\Rd (\mathcal L_0^{-1} u)(x) u(x)\eV,\\
&\leq\|\mathcal L_0^{-1} u\|_{\LL^2(\eV)}\| u\|_{\LL^2(\eV)}
\end{align*}
which implies
\begin{equation}\label{eq:spectre_L0}
\eta\|\mathcal L_0^{-1} u\|_{\LL^2(\eV)}\leq\|u\|_{\LL^2(\eV)}.
\end{equation}
In the following, we will use the orthogonal projection operator
$\Pi_0$ from~$\LL^2(\eV)$ onto~$\LL^2_0(\eV)$ defined by:
\begin{equation}\label{eq:PI0}
\forall f\in\LL^2(\eV),~
\Pi_0f
=f-\int_\Rd f(x)\eV.
\end{equation}

Let us now explain formally how we obtain an explicit formula for the invariant
measure~$\pi_\lambda$ of~\eqref{eq:EDS_lam}. For any  test function $\phi$ and since ${\mathcal
  L}_{\lambda} 1=0$, $\int_{\Rd}
{\mathcal L}_{\lambda} \Pi_0(\phi) \, \rmd \pi_\lambda =0 $. Thus, by
considering $f={\mathcal L}_0 \Pi_0(\phi)$,
for any test function $f$, $\int_{\Rd}
{\mathcal L}_{\lambda} {\mathcal
  L}_{0}^{-1} \Pi_0 f \, \rmd \pi_\lambda =0 $ which also writes
$\int_{\Rd} (I + {\mathcal T}_\lambda
{\mathcal
  L}_{0}^{-1} \Pi_0) \Pi_0 f \, \rmd \pi_\lambda =0 $ where $I$ denotes the
identity operator. This is equivalent to: for any test function $f$, $\int_{\Rd} (I + {\mathcal T}_\lambda
{\mathcal
  L}_{0}^{-1} \Pi_0) f \, \frac{\rmd \pi_\lambda}{\rmd \pi_0} \rmd
\pi_0 =\int_{\Rd} f \rmd \pi_0 $ which yields $(I + {\mathcal T}_\lambda
{\mathcal  L}_{0}^{-1} \Pi_0)^* \frac{\rmd \pi_\lambda}{\rmd \pi_0}= 1$, where $*$ denotes the dual
operator on the
Hilbert space~$\LL^2(\eV)$. As a consequence, we are naturally led to study the
operator~$\mathcal T_\lambda\mathcal L_0^{-1}\Pi_0$ defined
from~$\LL^2(\eV)$ to~$\LL^2(\eV)$. The aim of the next Lemma is to
show rigorously that we can define an invariant measure $\pi_\lambda$
of~\eqref{eq:EDS_lam} by defining its Radon-Nikodym derivative with respect to
$\pi_0$ as $(I + ({\mathcal T}_\lambda
{\mathcal  L}_{0}^{-1} \Pi_0)^*)^{-1} 1$.

We can now state the result concerning the existence of an invariant
measure for~\eqref{eq:EDS_lam}.
\begin{lem}\label{lem:pil}
Let us assume that
Assumption~{\bf(\hyppoinc{$\eta$})} holds for some positive
$\eta$. Then there exists $\lambda_1 \in (0,\lambda_0]$ such that for~$\lambda \in [0,\lambda_1]$, the dual operator $I +(\mathcal T_\lambda\mathcal L_0^{-1}\Pi_0)^*$ on the
Hilbert space~$\LL^2(\eV)$ of the operator $I +\mathcal
T_\lambda\mathcal L_0^{-1}\Pi_0$ is invertible and has a bounded inverse.

Let us then introduce, for~$\lambda \in [0,\lambda_1]$,  the function $g_\lambda \in \LL^2(\eV)$
and the associated measure $\pi_\lambda$ such that
\begin{equation}\label{eq:pi_lam}
\rmd\pi_\lambda= g_\lambda {\rmd\pi_0} \text{ where } g_\lambda=(I+(\mathcal T_\lambda\mathcal
L_0^{-1}\Pi_0)^*)^{-1}\mathbf1\,.
\end{equation}
The measure $\pi_\lambda$ is a probability measure which is invariant
for the process~$(X_t^\lambda)_{t\geq0}$ solution
to~\eqref{eq:EDS_lam}. 
\end{lem}

\begin{proof}

\underline{\em Step 1}: Let us first study the operator $\mathcal T_\lambda
\mathcal L_0^{-1} \Pi_0$. {From} the boundedness assumption on~$\nabla
V+F_\lambda=F_\lambda-F_0$ (see Assumption{\bf(\hypcadre)}-$(i)$), the
definition of ${\mathcal L}_0^{-1}$
and~\eqref{eq:spectre_L0}, for any~$u \in \mathcal \LL^2_0(\eV)$,
\begin{align*}
\|\mathcal T_\lambda \mathcal L_0^{-1}u\|_{\LL^2(\eV)}^2
&=\int_\Rd |(F_\lambda+\nabla V)(x)\cdot\nabla (\mathcal L_0^{-1}u)(x)|^2\eV\\
&\leq C\lambda^2\int_\Rd |\nabla (\mathcal L_0^{-1}u)(x)|^2\eV\\
&=-C\lambda^2\int_\Rd (\mathcal L_0^{-1}u)(x) u(x)\eV\\
&\leq C\lambda^2\|\mathcal L_0^{-1} u\|_{\LL^2(\eV)}\|u\|_{\LL^2(\eV)}\leq C\frac{\lambda^2}\eta\|u\|^2_{\LL^2(\eV)}.
\end{align*}

As a consequence, the
operator~$\mathcal T_\lambda\mathcal L_0^{-1}$ is bounded
from~$\LL^2_0(\eV)$ to $\LL^2(\eV)$, with:
\begin{equation}\label{eq:Ol}
\|\mathcal T_\lambda\mathcal L_0^{-1}\|_{\mathcal
  L(\LL^2_0(\eV),\LL^2(\eV))} \le\sqrt{\frac{C}{\eta}} \lambda.
\end{equation}
By composition,~$\mathcal T_\lambda\mathcal L_0^{-1}\Pi_0$
is thus a bounded operator from~$\LL^2(\eV)$ to itself, with a norm of
order~$\mathcal O(\lambda)$, and so is~$(\mathcal T_\lambda\mathcal L_0^{-1}\Pi_0)^*$.
As a consequence, for~$\lambda$ small enough, the
operator~$I+(\mathcal T_\lambda\mathcal L_0^{-1}\Pi_0)^*$
is invertible from~$\LL^2(\eV)$ to itself.

\underline{\em Step 2}: Let us now introduce the function $g_\lambda
\in \LL^2(\eV)$ defined by $$g_\lambda=(I+(\mathcal T_\lambda\mathcal
L_0^{-1}\Pi_0)^*)^{-1} \mathbf 1$$
and let us prove that $\rmd \pi_\lambda= g_\lambda \rmd \pi_0$ is invariant for the stochastic
differential equation~\eqref{eq:EDS_lam}.  Let $(Y^{\lambda,x}_t)_{t \ge}$ be the solution
to~\eqref{eq:EDS_lam} with initial condition $Y^{\lambda,x}_0=x$
(see~\eqref{eq:Ytlambdax}). Using the Markov property, the aim is to prove that, for any
${\mathcal C}^\infty$ bounded test function $f:\Rd \to \R$,
\begin{equation}\label{eq:inv_egalite}
\int_{\Rd} \E(f(Y^{\lambda,x}_t)) g_\lambda(x) \eV= \int_{\Rd} f(x)
g_\lambda(x) \eV.
\end{equation}
{From} Proposition~\ref{prop:feynman-kac}, we know that
$u(t,x)=\E(f(Y^{\lambda,x}_t))$ is the solution to~\eqref{eq:EDP_phi}
(with $\phi=0$), and from Proposition~\ref{prop:EDP_bien_posee}, we have for any~$T>0$,
\[u \in
\LL^\infty\left([0,T],\HH^1(\eV)
\right)
\cap\HH^1\left([0,T],\LL^2(\eV)\right).
\] Moreover, from Proposition~\ref{prop:regularite_EDP_phi}, $u$ is a
classical solution to~\eqref{eq:EDP_phi}.
Therefore,
\begin{align*}
\frac{d}{d t} \int_{\Rd}  \E(f(Y^{\lambda,x}_t)) g_\lambda(x) \eV
&=
\frac{d}{d t} \int_{\Rd}  u(t,x) g_\lambda(x) \eV\\
&=
\int_{\Rd}  \partial_t u(t,x) g_\lambda(x) \eV\\
&=
\int_{\Rd}  {\mathcal L}_\lambda u(t,x) g_\lambda(x) \eV,
\end{align*}
and ${\mathcal L}_\lambda u = \partial_t u \in \LL^2(\eV)$.
Now, notice that for any function $\psi$ which is the sum of a ${\mathcal C}^\infty$ function with compact support and a constant,
 $$\mathcal T_\lambda \psi=\mathcal T_\lambda\mathcal
 L_0^{-1}\Pi_0\mathcal L_0 \psi$$
holds true since~$\mathcal T_\lambda$ sends constant functions to~$0$.
Therefore, for any such function $\psi$, one has
\begin{align*}
\int_{\R^d} \mathcal L_\lambda \psi  g_\lambda \rmd \pi_0=
\int_{\R^d} \left[ (\mathcal L_0+\mathcal T_\lambda)\psi \right]  g_\lambda \rmd \pi_0
&=
\int_{\R^d}  \left[ (I+\mathcal T_\lambda\mathcal L_0^{-1}\Pi_0)\mathcal
L_0\psi \right]  g_\lambda \rmd \pi_0\\
&=
\int_{\R^d}  \left[ \mathcal
L_0\psi \right]  (I+(\mathcal T_\lambda\mathcal L_0^{-1}\Pi_0)^*)  g_\lambda \rmd \pi_0\\
&=
\int_{\R^d} \mathcal
L_0\psi \, \rmd \pi_0.
\end{align*}
Since $\pi_0$ is invariant for the dynamics~\eqref{eq:EDS} with
infinitesimal generator ${\mathcal L}_0$, the right-hand side is zero.
By density, the equality $\int_{\R^d} \mathcal L_\lambda \psi
g_\lambda \rmd \pi_0=0$ holds for any function $\psi$ such that $\mathcal
L_\lambda \psi \in \LL^2(\eV)$. Therefore, $\frac{d}{d t} \int_{\Rd}
\E(f(Y^{\lambda,x}_t)) g_\lambda(x) \eV=0$ which
yields~\eqref{eq:inv_egalite} after integration in time over $[0,t]$.

\underline{\em Step 3}: Let us finally check that $\pi_\lambda$ is a probability
measure. First, one has 
\begin{align*}
   \int_{\R^d}g_\lambda \rmd \pi_0=\int_{\R^d}(I+(\mathcal T_\lambda\mathcal
L_0^{-1}\Pi_0)^*)^{-1} \mathbf 1(I+(\mathcal T_\lambda\mathcal
L_0^{-1}\Pi_0))\mathbf 1\rmd \pi_0=\int_{\R^d}\rmd \pi_0=1.
\end{align*}
Second, one can prove that $g_\lambda \ge 0$. Indeed,
from~\eqref{eq:inv_egalite} and the fact that $Y^{\lambda,x}_t$ admits
a density $p^\lambda(t,x,y)$ with respect to the Lebesgue
measure (see Lemma~\ref{lem:p}), we have
$$
\int_{\Rd} \int_{\Rd} f(y) p^\lambda(t,x,y) \, \rmd y  g_\lambda(x) \eV= \int_{\Rd} f(x)
g_\lambda(x) \eV.
$$
This equality holds for any smooth test function $f$ and, by a density
argument, one can apply it to the bounded function
$f(x)=\textrm{sgn}(g_\lambda(x))$, where $\textrm{sgn}(y)=1_{y \ge 0}
-1_{y <0}$ denotes the sign function. One thus obtains
$$
\int_{\Rd} \int_{\Rd} \big(\textrm{sgn}(g_\lambda(y)) \textrm{sgn}(g_\lambda(x)) -1\big) p^\lambda(t,x,y) \, \rmd y  |g_\lambda|(x) \eV= 0.
$$
Thus, $\big(\textrm{sgn}(g_\lambda(y)) \textrm{sgn}(g_\lambda(x))
-1\big) p^\lambda(t,x,y)  |g_\lambda|(x)=0$  $\rmd x \otimes \rmd
y$-a.e.. Since $p^\lambda(t,x,y)>0$  $\rmd x \otimes \rmd
y$-a.e. (see Lemma~\ref{lem:p}) and $\int_{\Rd} |g_\lambda|(x) \,
\rmd x >0$,
this implies that $\rmd y$-a.e.,  $\textrm{sgn}(g_\lambda(y))=1$ or  $\rmd y$-a.e $\textrm{sgn}(g_\lambda(y))=-1$. The
conclusion then follows from the fact that  $\int_{\R^d}g_\lambda \rmd \pi_0=1$.
\end{proof}

Notice that in the case $\lambda=0$, we indeed have $g_0=1$.  The next result states the uniqueness of the invariant measure for~\eqref{eq:EDS_lam}.
\begin{lem}\label{lem:erggen}
Let us assume that
Assumption~{\bf(\hyppoinc{$\eta$})} holds for some positive
$\eta$. For $\lambda\in[0,\lambda_1]$ ($\lambda_1$ being the constant
introduced in Lemma~\ref{lem:pil}), the unique invariant measure of
the stochastic differential
equation~\eqref{eq:EDS_lam} is the
probability measure $\pi_\lambda$ defined by~\eqref{eq:pi_lam}. This probability measure is equivalent to the Lebesgue measure on $\R^d$ and for any initial condition $X_0$,
\begin{equation}\label{eq:erggen}
\forall f\in\LL^1(\pi_\lambda),\;\PP\left(\lim_{t \to \infty} \frac1t\int_0^tf(X^{\lambda}_s)\rmd s
= \int_\Rd f \rmd\pi_\lambda\right)=1. 
\end{equation}
\end{lem}
\begin{proof}
Let $\lambda\in[0,\lambda_1]$. Lemma~\ref{lem:pil} ensures that
$\pi_\lambda$ defined by~\eqref{eq:pi_lam} is an invariant probability measure for $\rmd
X_t^\lambda=F_\lambda(X_t^\lambda)\rmd t+\sqrt2\rmd W_t$. For $X_0$
distributed according to any invariant probability measure, Lemma
\ref{lem:p} ensures that this measure is equivalent to the Lebesgue
measure. As a consequence, all the invariant probability measures are
equivalent and the dynamics admits exactly one invariant probability
measure $\pi_\lambda$. Since $\pi_\lambda$ is the only invariant
probability measure, it is ergodic (see for example~\cite[Theorem 3.8
and Equation (52)]{rey-bellet-06}) and denoting by $(Y^{\lambda,x}_t)_{t\geq 0}$ the solution to \eqref{eq:EDS_lam} started from $Y_0=x\in\R^d$,
\begin{equation*}
\forall f\in\LL^1(\pi_\lambda),\;\rmd x\mbox{ a.e.},\;\PP\left(\lim_{t \to \infty} \frac1t\int_0^tf(Y^{\lambda,x}_s)\rmd s
= \int_\Rd f \rmd\pi_\lambda\right)=1. 
\end{equation*}
For any initial condition $X_0$, since the law of $X^{\lambda}_1$ is
absolutely continuous with respect to the Lebesgue measure (see Lemma~\ref{lem:p}), \eqref{eq:erggen} follows by the Markov property.
\end{proof}

\section{Tangent vector of the diffusion}\label{sect:vecteur_tangent}

In all this Section, $(X^\lambda_t)_{t \ge 0}$ denotes the process
solution to~\eqref{eq:EDS_lam}, with an initial condition $X_0$ which,
we recall, does not depend on $\lambda$.
 We establish various results on the tangent vector
$T_t$ defined by~\eqref{eq:definition_Tt}, which naturally appears in the
estimators~\eqref{eq:estim_1} and~\eqref{eq:estim_2} to evaluate $\dlam\left(\int_\Rd f\rmd\pi_\lambda\right)$.

\subsection{Definition and interpretations of the tangent vector}

If the function~$f$ is differentiable, one can
write~$\dlam\left(f(X_t^\lambda)\right)=T_t\cdot\nabla f(X_t^0)$,
where the process~$(T_t)_{t\geq0}$ is the so-called
\emph{tangent vector}, defined as
\begin{equation}\label{eq:definition_Tt}
T_t
=\dlam X_t^\lambda,
\end{equation}
and the existence of which is ensured by the following proposition.\begin{prop}\label{proptt}For any~$t\geq0$, the function~$\lambda\mapsto X_t^\lambda$ is almost surely
differentiable, and the definition of the tangent vector~\eqref{eq:definition_Tt} makes sense.
Moreover,~$(T_t)_{t\geq0}$ almost surely satisfies the following ordinary differential
equation whose coefficients depend on~$(X_t^0)_{t\geq0}$:
\begin{equation}\label{eq:edo_Tt}
\left\{
\begin{aligned}
\frac{\rmd T_t}{\rmd t}
&=\dlam F_\lambda(X_t^0)-\nabla^2V(X_t^0) T_t,\\
T_0&=0.
\end{aligned}
\right.
\end{equation}
\end{prop}
\begin{proof}
 By {\bf(\hypcadre)}-$(i)$ and the continuity of $\nabla V$ and $(X^0_t)_{t\geq 0}$, $t\mapsto |\dlam F_\lambda(X_t^0)|+|\nabla^2V(X_t^0)|$ is locally bounded. Hence \eqref{eq:edo_Tt} admits a unique solution $(T^0_t)_{t\geq 0}$ by the Cauchy-Lipschitz theorem. Let us prove that for $\bar t>0$, $\lambda\mapsto X^\lambda_{\bar t}$ is differentiable at $\lambda =0$ with derivative equal to $T^0_{\bar t}$. For $\lambda \in[0,\lambda_0]$, we set $\tau_\lambda=\inf\{t\geq 0:|X^\lambda_t|\geq \sup_{t\in[0,\bar t]}|X^0_t|+1\}$ with convention $\inf\emptyset=+\infty$. Let $L^{X^0}_{\bar t}=\sup_{x\in\R^d:|x|\leq \sup_{t\in[0,\bar t]}|X^0_t|+1}|\nabla^2V(x)|$. For $t\in[0,\bar{t}]$, one has
\begin{align*}
 \sup_{s\in[0,t]}|X^\lambda_{s\wedge\tau_\lambda}-X^0_{s\wedge\tau_\lambda}|&\leq \int_0^{t\wedge\tau_\lambda}\left|F_\lambda(X^\lambda_s)+\nabla V(X^\lambda_s)\right|+\left|\nabla V(X^0_s)-\nabla V(X^\lambda_s)\right|\rmd s\\&\leq C\lambda t+L^{X^0}_{\bar t}\int_0^t|X^\lambda_{s\wedge\tau_\lambda}-X^0_{s\wedge\tau_\lambda}|\rmd s
\end{align*} 
so that $\sup_{s\in[0,t]}|X^\lambda_{s\wedge\tau_\lambda}-X^0_{s\wedge\tau_\lambda}|\leq \frac{C\big(e^{L^{X^0}_{\bar t}t}-1\big)}{L^{X^0}_{\bar t}}\lambda$. For $\lambda\leq \frac{L^{X^0}_{\bar t}}{C\big(e^{L^{X^0}_{\bar t}t}-1\big)}$, one deduces that $\tau_\lambda\geq \bar{t}$ and 
$\sup_{s\in[0,t]}|X^\lambda_{s}-X^0_{s}|\leq \frac{C \big(e^{L^{X^0}_{\bar t}t}-1\big)}{L^{X^0}_{\bar t}}\lambda$. In particular, $X^\lambda_t$ converges to $X^0_t$ uniformly for $t\in[0,\bar t]$. Now, for $t\geq 0$,
\begin{align*}
   X^\lambda_{t}-X^0_{t}=\int_0^{t}(F_\lambda(X^\lambda_s)-F_0(X^\lambda_s))\rmd s+\int_0^t\nabla^2V(\xi^\lambda_s)(X^0_{s}-X^\lambda_{s})ds,
\end{align*}
where, by a slight abuse of notations, $\nabla^2V(\xi^\lambda_s)$
stands for the matrix $(\partial_{ij}V(\xi^{\lambda,i}_s))_{1\leq
  i,j\leq d}$ and $\forall i\in\{1,\hdots,d\}$,
$\xi^{\lambda,i}_s\in[X^0_{s},X^\lambda_{s}]$. For
$s\in[0,\bar t]$ and $\lambda\in \bigg (0,\lambda_0\wedge
\frac{L^{X^0}_{\bar t}}{C(e^{L^{X^0}_{\bar t}t}-1)}\bigg]$, $|\xi^{\lambda,i}_s|\leq\sup_{t\in[0,\bar t]}|X^0_t|+1$. Hence for $t\in[0,\bar{t}]$,
\begin{align*}
 \sup_{s\in[0,t]} \left|\frac{X^\lambda_{s}-X^0_{s}}{\lambda}-T^0_s\right|\leq &\int_0^{t}\left|\frac{F_\lambda(X^\lambda_s)-F_0(X^\lambda_s)}{\lambda}-\dlam F_\lambda(X_s^0)\right|+|\nabla^2V(X^0_s)-\nabla^2V(\xi^\lambda_s)||T^0_s|\rmd s\\&+L^{X^0}_{\bar t}\int_0^t\left|\frac{X^\lambda_{s}-X^0_{s}}{\lambda}-T^0_s\right|\rmd s.
\end{align*}
By {\bf(\hypcadre)}-$(i)$-$(ii)$ and the uniform convergence of
$X^\lambda_t$ to $X^0_t$ for $t\in[0,\bar t]$, $$\lim_{\lambda \to 0}
\int_0^{\bar
  t}\left|\frac{F_\lambda(X^\lambda_s)-F_0(X^\lambda_s)}{\lambda}-\dlam
  F_\lambda(X_s^0)\right|+|\nabla^2V(X^0_s)-\nabla^2V(\xi^\lambda_s)||T^0_s|\rmd
s = 0.$$With Gr\"onwall's lemma, one concludes that $\sup_{s\in[0,\bar{t}]}\left|\frac{X^\lambda_{s}-X^0_{s}}{\lambda}-T^0_s\right|$ converges to $0$ as $\lambda\to 0$.
\end{proof}

We have the following expression of~$(T_t)_{t\geq0}$ as an integral:

\begin{prop}\label{prop:tangent}
Define the \emph{resolvent}~$(R_{X^0}(s,t))_{s,t\geq0}$ associated
with 
Equation~\eqref{eq:edo_Tt} as the
solution, with values in~$\R^{d\times d}$, to the following ordinary differential equation:
\begin{equation}\label{eq:resolvante}
\left\{
  \begin{aligned}
    \partial_tR_{X^0}(s,t)&=-\nabla^2V(X_t^0)R_{X^0}(s,t),~s,t\geq0,\\
    R_{X^0}(s,s)&=I_d,~s\geq0,
  \end{aligned}
\right.
\end{equation}
where~$I_d$ is the~$d\times d$ identity matrix.
The resolvent satisfies the following semigroup property
\begin{equation}\label{eq:semigroupe_R}
\forall r,s,t\in[0,\infty),~R_{X^0}(s,t)R_{X^0}(r,s)=R_{X^0}(r,t).
\end{equation}
One can recover the tangent vector from the resolvent through the
following formula:
\begin{equation}\label{eq:expression_Tt}
\forall t\geq0,
~T_t=\int_0^tR_{X^0}(s,t)\dlam F_\lambda(X_s^0)\rmd s.
\end{equation}
\end{prop}
\begin{proof}
The semigroup property~\eqref{eq:semigroupe_R} is a consequence of uniqueness for
Equation~\eqref{eq:resolvante}, satisfied by the two
processes~$(R_{X^0}(s,t))_{t\geq 0}$ and~$(R_{X^0}(r,t)R_{X^0}(r,s)^{-1})_{t\geq0}$.

In view of the differential equations satisfied by~$(T_t)_{t\geq0}$
and~$(R_{X^0}(s,t))_{t\geq0}$, one has, from the
equality~$R_{X^0}(t,0)=R_{X^0}(0,t)^{-1}$,
\begin{align*}
\partial_t(R_{X^0}(t,0) T_t)
&=-R_{X^0}(t,0)\partial_t(R_{X^0}(0,t))R_{X^0}(t,0) T_t
+R_{X^0}(t,0)\partial_tT_t\\
&=R_{X^0}(t,0)\nabla^2V(X_t^0)R_{X^0}(0,t)R_{X^0}(t,0) T_t\\
&\quad+R_{X^0}(t,0)\dlam F_\lambda(X_t^0)
-R_{X^0}(t,0)\nabla^2V(X_t^0) T_t\\
&=R_{X^0}(t,0)\dlam F_\lambda(X_t^0).
\end{align*}
Integrating over~$[0,t]$, one obtains
\[
R_{X^0}(t,0) T_t=\int_0^tR_{X^0}(s,0)\dlam F_\lambda(X_s^0)\rmd s,
\]
and the result follows by using the semigroup property~\eqref{eq:semigroupe_R}.
\end{proof}

Notice that the resolvent is also the
differential of the trajectory with respect to its initial
condition. 
\begin{lem}\label{lemderci}Let $(Y_t^x)_{t \ge 0}$ solve~\eqref{eq:Ytx}. Then for any $t\geq 0$, $x\mapsto Y^x_t$ is ${\mathcal C}^1$ on $\R^d$ with Jacobian matrix $(D Y_t^x)_{i,j}=\partial_{x_j}Y_t^{i,x}$ given by $DY_t^x=R_{Y^x}(0,t)$.
\end{lem}
\begin{proof}
By standard results on ordinary differential equations,  $x\mapsto Y^x_t$ is ${\mathcal C}^1$ with Jacobian matrix $DY_t^x$ solving the equation \[
\forall t\geq0,
~DY_t^x
=I_d
-\int_0^t\nabla^2V(Y_s^x)DY_s^x\rmd s,
\]
obtained by spatial derivation of $Y_t^x
=x
-\int_0^t\nabla V(Y_s^x)\rmd s
+\sqrt2W_t.$
By uniqueness for~\eqref{eq:resolvante}, one has $DY_t^x=R_{Y^x}(0,t)$.
\end{proof}

In the following, we will need the following result about the link between
the forward resolvent and its backward counterpart.
\begin{lem}\label{lem:retournement_resolvante}
Let~$(Y_s)_{0\leq s\leq t}$ satisfy Equation~\eqref{eq:EDS} with~$Y_0$
distributed according to~$\pi_0$. {From} the
reversibility of the
dynamics~\eqref{eq:EDS}, the process~$(Z_s)_{0\leq s\leq t}$ defined
by~$Z_s=Y_{t-s}$ has the same law as $(Y_s)_{0\leq s\leq t}$, and one has
the relation
\[
R_Y(0,s)
=R_Z^T(t-s,t),
\]
where~$R_Z^T$ is the transposed matrix of the resolvent associated with~$Z$.
\end{lem}
\begin{proof}
Uniqueness holds for the ordinary differential equation satisfied
by~$s \mapsto R_Y(0,s)$:
\begin{equation}\label{eq:EDO_semigroupe}
\left\{
\begin{aligned}
\frac{\rmd R}{\rmd s} (s)&=-\nabla^2V(Y_s)R(s),\\
R(0)&=I_d.
\end{aligned}
\right.
\end{equation}
One can check that $s \mapsto R_Z^T(t-s,t)$ also
solves~\eqref{eq:EDO_semigroupe}. Indeed, since, by the semigroup
property, $R_Z(t-s,t)=R_Z(t,t-s)^{-1}$, one has, for $s \in [0,t]$,
\begin{align*}
\partial_sR_Z(t-s,t)
&=-R_Z(t-s,t)\left(\partial_sR_Z(t,t-s)\right)R_Z(t-s,t)\\
&=-R_Z(t-s,t)\nabla^2V(Z_{t-s})R_Z(t,t-s)R_Z(t-s,t)\\
&=-R_Z(t-s,t)\nabla^2V(Z_{t-s})\\
&=-R_Z(t-s,t)\nabla^2V(Y_s).
\end{align*}
This concludes the proof.
\end{proof}

\subsection{Almost sure boundedness of~$R_{X^0}(0,t)$ and~$T_t$}

The tangent vector can take large values, since the second term in
the right-hand side of~\eqref{eq:edo_Tt} will provide exponential growth
for~$(T_t)_{t\geq0}$, typically when the trajectory~$(X_t^0)_{t\geq0}$
is close to a local maximum of~$V$,
or when it crosses a saddle point of~$V$.
In the sequel, we need some assumptions on~$V$
to control this behavior.

\subsubsection{Local-in-time boundedness of~$R_{X^0}(s,t)$ and~$T_t$}

Let us first introduce an assumption which will be sufficient to get
the local-in-time boundedness of~$R_{X^0}(s,t)$ and~$T_t$.
\begin{hyp}[\hypspec]
The matrix-valued function~$\nabla^2V:\Rd\to\R^{d\times d}$ is bounded
from below, in the sense that there
exists~$\alpha\in\R$ such that, for all~$x,h\in\Rd$,
\[
h\cdot\nabla^2V(x)h\geq\alpha|h|^2.
\]
Equivalently, the spectrum of~$\nabla^2V(x)$ is bounded from below
by~$\alpha$, uniformly in~$x$.
\end{hyp}

Under Assumption~{\bf(\hypspec)}, the random variables~$T_t$
and~$R_{X^0}(s,t)$ are bounded:
\begin{lem}\label{lem:R_T_Linfty}
One has
\begin{equation}\label{eq:R_T_Linfty}
\forall 0 \le s<t,~
|R_{X^0}(s,t)|
\leq e^
{-\int_s^t\min\mathrm{Spec}\left(\nabla^2V(X_u^0)\right)\rmd u},
\end{equation}
$\R^{d\times d}$ being endowed with the
matricial norm associated with the Euclidean norm on~$\Rd$.
In addition, if the Assumption~{\bf(\hypspec)} is satisfied, for any~$T>0$, the
random variables~$\sup_{0\leq s\leq t\leq T}|R_{X^0}(s,t)|$ and~$\sup_{0\leq t\leq T}|T_t|$
lie
in~$\LL^\infty(\Omega)$.
\end{lem}
\begin{proof}
{For} any vector~$x$, one has
\begin{align*}
\partial_t|R_{X^0}(s,t)x|^2
&=-2(R_{X^0}(s,t)x)^T\nabla^2V(X_t^0)(R_{X^0}(s,t)x)\\
&\leq-2\min\textrm{Spec}\left(\nabla^2 V(X_t^0)\right)|R_{X^0}(s,t)x|^2.
\end{align*}
As a consequence, one has the estimation
\[
|R_{X^0}(s,t)x|^2
\leq|x|^2e^
{-2\int_s^t\min\textrm{Spec}\left(\nabla^2V(X_u^0)\right)\rmd u}
\]
so that~\eqref{eq:R_T_Linfty} holds. If the Assumption~{\bf(\hypspec)}
is satisfied, this
inequality proves that~$R(s,t)$ is in~$\LL^\infty(\Omega)$ locally
uniformly in time.
{From} the expression~\eqref{eq:expression_Tt} of~$T_t$ and the
boundedness of~$\dlam F_\lambda$, one also concludes
that~$T_t\in\LL^\infty(\Omega)$, locally uniformly in time.
\end{proof}

\subsubsection{Global-in-time boundedness of~$R_{X^0}(0,t)$}

We need some additional assumption on the convexity of the potential for~$(R_{X^0}(0,t))_{t\geq0}$
to be bounded globally in time. 
\begin{hyp}[\hypconv]
The potential~$V$ is such that
\begin{equation}\label{eq:trou_moyen}
\int_\Rd\max\left(0,-\min\mathrm{Spec}\left(\nabla^2V(x)\right)\right)\eV<\infty
\mbox{ and }
\int_\Rd\min\mathrm{Spec}\left(\nabla^2V(x)\right)\eV>0.
\end{equation}
\end{hyp}
In this assumption, the first inequality, always satisfied
under~{\bf(\hypspec)}, ensures that the
integral~$\int_\Rd\min\mathrm{Spec}\left(\nabla^2V(x)\right)\eV$
is well defined in~$(-\infty,\infty]$. We refer to
Appendix~\ref{sec:annex_conv} for a discussion of this Assumption.

\begin{lem}\label{lem:convergence_R}
Under Assumptions~{\bf(\hypspec)} and~{\bf(\hypconv)}, the
resolvent matrix~$R_{X^0}(0,t)$ almost surely
converges to~$0$ as~$t$ goes to infinity, with exponential
rate. Namely, for any~$\beta$ with
\[
0
<\beta
<\int_\Rd\min\mathrm{Spec}\left(\nabla^2V(x)\right)\eV,
\]
there exists an almost surely finite random variable~$C>0$
such that
\begin{equation}\label{eq:convergence_R}
\forall t\geq0,~|R_{X^0}(0,t)|\leq Ce^{-\beta t}.
\end{equation}
\end{lem}
\begin{proof}
{From} Lemma~\ref{lem:erggen}, by
ergodicity (see~\eqref{eq:erggen}), one has
\[
\lim_{t\to\infty}\frac1t\int_0^t\min\textrm{Spec}\left(\nabla^2 V(X_s^0)\right)\rmd s
=\int_\Rd\min\textrm{Spec}\left(\nabla^2V(x)\right)e^{-V(x)}\rmd x
~~\mbox{ a.s.}
\]
We conclude by combining this limit with Equation~\eqref{eq:R_T_Linfty}.
\end{proof}

\begin{rem}[On the Assumption~{\bf(\hypconv)}]\label{rem:hypconv}
While Assumption~{\bf(\hypconv)} is automatically
satisfied in dimension~$1$ from a mere
integration by parts, this is not
the case in higher dimension. Indeed, if one applies the integration by parts
formula in this case, one only obtains that
\[
\int_\Rd\nabla^2V(x)\eV
=\int_\Rd\nabla V(x)\otimes\nabla V(x)e^{-V(x)}\rmd x
\]
is a positive definite matrix (because of the integrability
of~$e^{-V}$, for any~$y$ in~$\Rd$, the function~$x\mapsto\nabla V(x)\cdot y$ cannot be the zero function), so
that the minimum of its spectrum is positive. A counterexample
to Assumption~{\bf(\hypconv)} is given
by a tensor potential~$V(x)=U(x_1)+\hdots+U(x_d)$ with a
well chosen function~$U$. Indeed, in this case the left hand side of equation
\eqref{eq:trou_moyen} rewrites
\[
\int_\Rd\min_{i \in \{1, \ldots,d\}}(U''(x_i))e^{-\sum_i
U(x_i)}\rmd x_1\hdots\rmd x_d=\E\left[\min_{i \in \{1, \ldots,d\}} U''(X_i)\right],
\]
where
$X_i$ are i.i.d random variables with distribution~$e^{-U(x)}\rmd x$.
If~$U$ is chosen so that~$U''$ is bounded and has a strictly negative
minimum, then the sequence
$\left(\displaystyle\min_{i \in \{1, \ldots,d\}}
  U''(X_i)\right)_{d\geq1}$ converges almost surely as~$d$ goes to infinity
to the negative constant~$\min U''$. Then, from
the dominated convergence theorem, the quantity
$\E\left[\displaystyle\min_{i \in \{1, \ldots,d\}} U''(X_i)\right]$
converges to $\min U''$, and is thus negative when~$d$
is large enough.
\end{rem}

\begin{rem}[On the assumptions of Lemma~\ref{lem:convergence_R}]
Assumption~{\bf(\hypconv)} is not necessary for~\eqref{eq:convergence_R} to hold.
Indeed, if the matrices~$\nabla^2V(x)$ commute, the matrix~$R_{X^0}$ is
given by
\[
R_{X^0}(0,t)=e^{- \int_0^t\nabla^2V(X_s^0)\rmd s}.
\]
and the convergence of~$\frac1t\int_0^t\nabla^2V(X_s^0)\rmd s$ to the
positive definite matrix~$\int_\Rd\nabla^2V(x)\eV$ implies
that~\eqref{eq:convergence_R} holds for $\beta < \min {\rm Spec} \left(\int_\Rd\nabla^2V(x)\eV\right)$, even in the cases when~$V$
does not satisfy Assumption {\bf(\hypconv)}.
An example where the matrices~$\nabla V^2(x)$ commute is the case of a
tensor potential $V(x)=U(x_1)+\hdots+U(x_d)$. As seen before,~$U$ and~$d$ can be
chosen such that~$V$ does not satisfy Assumption~{\bf(\hypconv)}.

However, it is likely that Lemma~\ref{lem:convergence_R} does not hold
under the sole ergodicity property:
\[
\lim_{t\to\infty}\frac1t\int_0^t \left(\nabla^2 V(X_s^0)\right)\rmd s
=\int_\Rd \nabla V(x) \otimes \nabla V(x) e^{-V(x)}\rmd x
~~\mbox{ a.s.}
\]
Indeed, there exists some family of symmetric matrices~$(A_t)_{t\geq0}$ converging in the Ces\`aro sense to a
positive-definite matrix, for which the solution
to~$\frac{\rmd}{\rmd t} R_t=-A_tR_t$,~$R_0=I_d$ does not converge to~$0$ as~$t$
goes to infinity. An example of this phenomenon is given by
\[
A_t
= \Omega_t\begin{pmatrix}-1&0\\0&3\end{pmatrix}\Omega_t^T,
\mbox{ where }
\Omega_t
=\begin{pmatrix}\cos t&-\sin t\\\sin t&\cos t\end{pmatrix}
.\]
Indeed, the family~$(A_t)_{t\geq0}$ converges in the Ces\`aro sense
to~$I_d$ as~$t$ goes to infinity, but
the associated matrix~$(R_t)_{t\geq0}$ diverges. To show this last
point, consider
the matrix~$M_t=\Omega_t^TR_t$.
Since
$\partial_t\Omega_t=\Omega_t\begin{pmatrix}0&-1\\1&0\end{pmatrix}$, then
\[
\partial_tM_t
=\begin{pmatrix}1&1\\-1&-3\end{pmatrix}M_t
\]
holds.
As a consequence,
$R_t=\Omega_t\exp\left(t\begin{pmatrix}1&1\\-1&-3\end{pmatrix}\right)$. The
eigenvalues of the matrix
$\begin{pmatrix}1&1\\-1&-3\end{pmatrix}$ are
$-1-\sqrt3$ and~$-1+\sqrt3$, the latter being
positive, so that~$R_t$ diverges as~$t$ goes to infinity.
\end{rem}

\subsection{Boundedness of moments
  of~$R_{X^0}(s,t)$ and $T_t$}\label{sec:expo_cv_R}
In the sequel, we will need to control the moments of~$T_t$. {From}
Equation~\eqref{eq:expression_Tt} and the boundedness of $\dlam
F_\lambda$ (see Assumption~{\bf (\hypcadre)}-(i)), this boils down to estimating the
moments of~$R_{X^0}(s,t)$.
For this purpose, from~\eqref{eq:R_T_Linfty}, it is enough
control expectations of the
form$~\E\left[e^{-\beta \int_0^t\min\mathrm{Spec}(\nabla^2V)(Y_s^x)\rmd s}\right]$,
where~$\beta$ is a positive constant.

\subsubsection{Preliminary result when $X_0 \sim \pi_0$}
One can deduce from Proposition~\ref{prop:conv_expo} a criterion for
exponential convergence of the moments of $R_{X^0}(0,t)$ to $0$ as
$t\to\infty$. To state the result, we need to strengthen the
assumptions~{\bf(\hypspec)} and~{\bf(\hypconv)} which is the point of the following assumption. For any $\rho>0$, let
us consider:
\begin{hyp}[\hypspecbeta{$\rho$}]
  Assume that
  \begin{align*}
    &-\infty<\inf_{x \in \R^d} \min\mathrm{Spec}(\nabla^2V(x)) \leq 0,\\
  &  \int_\Rd \min\mathrm{Spec}(\nabla^2V(x)) \eV>0 \text{ and }  \int_\Rd \left(\min\mathrm{Spec}(\nabla^2V(x)) \right)^2\eV<\infty,\\
&- (\inf
    \min\mathrm{Spec}(\nabla^2V(x)))\frac{\int_\Rd
      \left(\min\mathrm{Spec}(\nabla^2V(x)) \right)^2\eV}{\left(\int_\Rd \min\mathrm{Spec}(\nabla^2V(x))\eV\right)^2}
    <\rho.
  \end{align*}
\end{hyp}
Notice that for $\eta >0$ and $\beta >0$ under assumptions
{\bf(\hyppoinc{$\eta$})} and {\bf (\hypspecbeta{$\eta
  / \beta $})} then the assumptions~\eqref{eq:crit_conv1} and
\eqref{eq:crit_conv2} of Proposition \ref{prop:conv_expo} are satisfied
with $\phi(x)=\beta \min\mathrm{Spec}(\nabla^2V(x))$.

We are now in position to state a simple consequence of Proposition~\ref{prop:conv_expo}:
\begin{prop}\label{cor:momr}
   Let $(X^0_t)_{t\geq 0}$ solve \eqref{eq:EDS} starting from $X_0$
   distributed according to $\pi_0$. Assume
   that~{\bf(\hyppoinc{$\eta$})} and~{\bf (\hypspecbeta{$\eta
  / \beta $})} hold for some $\eta >0$ and $\beta >0$.
Then there is a constant $C\in(0,+\infty)$ such that
$$\forall t\geq 0,\;\E_{\pi_0}[|R_{X^0}(0,t)|^\beta]\leq Ce^{-t/C}.$$
\end{prop}
Here and in the following, the notation $\E_{\pi_0}$ means that the
initial condition $X_0$ of the processes $(X^\lambda_t)_{t \ge 0}$
solution to~\eqref{eq:EDS_lam} is distributed according to $\pi_0$. 
\begin{proof}[Proof of Proposition~\ref{cor:momr}]
To apply Proposition~\ref{prop:feynman-kac} to the
function~$\phi(x)=\beta \min\mathrm{Spec}(\nabla^2V(x))$, we need this
function to be locally Lipschitz. Since~$\nabla^2 V$ is
locally Lipschitz (see Assumption~{\bf(\hypV)}-$(i)$), this is a consequence of the
Lemma~\ref{lem:min_spec_lipschitz} given below.
By Proposition~\ref{prop:feynman-kac}, the
function~$u(t,x)=\E\left[
e^{-\int_0^t\beta\min\mathrm{Spec}\left(\nabla^2V(Y_s^x)\right)\rmd s}
\right]$
is the solution to Equation~\eqref{eq:EDP_phi} in the sense of
Definition~\ref{defi:solution_EDP}
for~$\phi(x)=\beta\min\mathrm{Spec}(\nabla^2V(x))$ and~$f(x)=1$.

Since conditions~\eqref{eq:crit_conv1} and~\eqref{eq:crit_conv2} hold
for this choice of~$\phi$, Equation~\eqref{eq:R_T_Linfty} and 
Proposition~\ref{prop:conv_expo} give
\[
\E_{\pi_0}\left[|R_{X^0}(0,t)|^\beta\right]\leq \int_{\R^d}u(t,x)\eV\leq \left(\int_{\R^d}u^2(t,x)\eV\right)^{1/2}
\leq Ce^{-t/C}
\]
for some positive constant $C$.\end{proof}
\begin{lem}\label{lem:min_spec_lipschitz}
The function~$A\mapsto\min\mathrm{Spec}(A)$ is a Lipschitz 
function on the space of symmetric~$d\times d$ matrices.
\end{lem}
\begin{proof}
Let~$A$ be a symmetric matrix, and let~$x$ be a vector in~$\Rd$
such that~$|x|=1$ and~$\min\mathrm{Spec}(A)=x\cdot Ax$. Then, for any
symmetric matrix~$B$, one has
\[
\min\mathrm{Spec}(B)
\leq x\cdot Bx
=x\cdot (B-A)x+x\cdot Ax
\leq|A-B|+\min\mathrm{Spec}(A),
\]
$\R^{d\times d}$ being endowed with the
matricial norm associated with the Euclidean norm on~$\Rd$.
By exchanging~$A$ and~$B$ in the previous inequality, one obtains
\[
|\min\mathrm{Spec}(A)-\min\mathrm{Spec}(B)|
\leq|A-B|.
\]
\end{proof}

  \subsubsection{Uniform-in-time boundedness of moments of $T_t$}\label{sect:variance_bornee}
Numerically, the computation of~\eqref{eq:but} through the long-time
limit of a Monte Carlo
approximation of the expression
$\E[T_t\cdot\nabla f(X_t^0)]$
is only possible if~$T_t$ has a bounded variance uniformly in time.

A case where this fact is easily proved is when the function~$V$ is
$\eta$-convex, where $\eta$ is a positive constant. We recall that this means that the spectrum of~$\nabla^2V(x)$ is bounded from below by~$\eta$, independently of~$x$.
 More precisely, one has the following proposition.
\begin{prop}\label{prop:varalconv}
Assume that the $V$ is $\eta$-convex, for a positive constant $\eta$. Then, for any $\alpha
\ge 1$,
$$\sup_{t \ge 0} \E |T_t|^\alpha < \infty.$$
In particular, $T_t$  has a bounded variance uniformly in time.
\end{prop}
\begin{proof}
By \eqref{eq:R_T_Linfty} and the boundedness of $\dlam F_\lambda$, 
one has
\begin{align*}
\E[|T_t|^\alpha]
=\E\left[\left|\int_0^t R_{X^0}(s,t)\dlam F_\lambda(X_s^0)\rmd s\right|^\alpha\right]\leq&\E\left[\left(C\int_0^te^{-\alpha(t-s)}\rmd s\right)^\alpha\right]<\infty,
\end{align*}
so that~$T_t$ has a finite moment of order $\alpha$.
\end{proof}

The convexity assumption on the potential can be loosened, as shown
in the next Proposition. 
\begin{prop}\label{prop:variance}
  Let $\alpha\in[1,+\infty)$. Assume that~{\bf(\hyppoinc{$\eta$})} holds
  for some positive $\eta$, that 
    the initial condition~$X_0$ is distributed according to a
    measure~$\mu_0$ having a density with respect to the
    measure~$\eV$ which is in $\LL^p(\eV)$ for some $p \in (1,
    \infty]$, and Assumption~{\bf (\hypspecbeta{$\eta (p-1)
  / (\alpha p) $})} holds (with the convention $\eta ( \infty-1)
  / (\alpha \infty)= \eta/\alpha$).
 Then,
$$\sup_{t \ge 0} \E |T_t|^\alpha < \infty.$$
and, when $\alpha\geq 2$, $T_t$ has a bounded variance uniformly in time.
\end{prop}
\begin{proof}
By \eqref{eq:expression_Tt} and {\bf(\hypcadre)}-$(i)$,
$$\E^{1/\alpha}[|T_t|^\alpha]\leq \int_0^t\E^{1/\alpha}[|R_{X^0}(s,t)\dlam F_\lambda(X_s^0)|^\alpha]\rmd s\leq C\int_0^t\E^{1/\alpha}[|R_{X^0}(s,t)|^\alpha]\rmd s$$

Let $\mu_s$ denote the law of $X^0_s$ for $s\geq 0$ and $(Y_t)_{t\geq
  0}$ be a solution to \eqref{eq:EDS} with $Y_0$ distributed according
to $\pi_0$. We notice that the Markov property gives: for $0\leq s\leq t$,
 $$
  \E[|R_{X^0}(s,t)|^\alpha]
  =\E\left[|R_{Y}(s,t)|^\alpha\frac{\rmd \mu_s}{\eV}(Y_s)\right].
 $$
Using H\"older inequality with $q=p/(p-1)$ ($q=1$ if $p=\infty$),
Lemma~\ref{lem:densite_X_t}  and Proposition~\ref{cor:momr}, one
  deduces that for $t\geq s \ge 0$,
  \begin{align*}
    \E[|R_{X^0}(s,t)|^\alpha]
    &\leq \E \left[|R_Y(s,t)|^{\alpha q}\right]^{\frac1q}
   \left\|\frac{\rmd
          \mu_s}{\eV}\right\|_{\LL^p(\eV)}\leq C e^{-\frac{t-s}{C}}.
  \end{align*}
This concludes the proof.

\end{proof}

We are now in position to give sufficient conditions for the
finiteness of the variance of the two estimators~\eqref{eq:estim_1}
and~\eqref{eq:estim_2}.
\begin{cor}
Let $f: \Rd \to \R$ be a ${\mathcal C}^1$ function such that $\nabla
f$ is bounded.
Let us assume that either $V$ is $\eta$-convex (for a positive
constant $\eta$), or that there exists $\eta >0$ and $p\in
(1,\infty]$ such that~{\bf(\hyppoinc{$\eta$})} holds,~$X_0$ is distributed according to a
    measure~$\mu_0$ having a density in $\LL^p(\eV)$ with respect to~$\eV$ and Assumption~{\bf (\hypspecbeta{$\eta (p-1)
  / (2 p) $})} holds. Then, 
$$\sup_{t \ge 0} {\rm Var}(T_t \cdot \nabla f (X^0_t) ) < \infty \text{ and
} \sup_{t \ge 0} {\rm Var}\left(\frac 1 t \int_0^t T_s \cdot \nabla f (X^0_s) \,
\rmd s\right) < \infty.$$
\end{cor}
\begin{proof}
These results are simple consequences of the boundedness of $\nabla f$
and Proposition \ref{prop:variance} for $\alpha=2$.
\end{proof}
{From} the Central Limit Theorem for trajectorial averages (see for
example~\cite[Section 2.1.3, Theorem 6.3.20]{duflo-97}), it is expected that the variance of $\frac 1 t \int_0^t T_s \cdot \nabla f (X^0_s) \,
\rmd s$ actually scales like $1/t$ in the limit $t \to
\infty$. This requires for example to prove the existence of a
solution to the Poisson problem associated with the Markov process
$(X^0_s,T_s)_{s \ge 0}$, which does not seem to be ensured under our
set of assumptions. We leave the study of this issue to a future work.

\begin{rem}
Under the additional assumption~{\bf (\hypdelta)} given in Appendix~\ref{sec:hypdelta}, it is possible to extend the
previous results to more general initial conditions.  Assume that the initial condition~$X_0$ is distributed according to a
    measure~$\mu_0$ such that the measure~$e^{\frac12V} \rmd \mu_0$ can be written as
  \begin{equation}\label{eq:mu0V}
  e^{\frac12V(x)} \rmd \mu_0 =f(x)\rmd x+\rmd \nu,
  \end{equation}
  where~$f$ is some function in~$\LL^p(\rmd x)$ with $p \in [1,2]$ and~$\nu$ is some finite
  measure on~$\Rd$. From~\eqref{eq:majonorml2}, for any~$t>0$,~$\mu_t$ is absolutely continuous with respect
  to~$e^{- V(x)}\rmd x$ with $\frac{{\rmd}\mu_t}{\eV} \in \LL^2(\eV)$.
Now, by the semi-group property satisfied
 by~$R_{X^0}$, \eqref{eq:R_T_Linfty} and the fact  that $- \min{ \rm
   Spec}(\nabla^2 V(x)) \le  C < \infty$, one has for~$\varepsilon>0$,
  \[
  |R_{X^0}(s,t)|
  \leq|R_{X^0}(s\vee\varepsilon,t)R_{X^0}(s,s\vee\varepsilon)|
  \leq e^{C(\varepsilon-s)^+}|R_{X^0}(s\vee\varepsilon,t)|.
  \]
  For $\alpha>0$, using a similar change of measure as in the previous proof, the fact
  that $\frac{{\rmd}\mu_\epsilon}{\eV} \in \LL^2(\eV)$ and Proposition~\ref{cor:momr}, one deduces that under Assumptions~{\bf(\hyppoinc{$\eta$})} and~{\bf (\hypspecbeta{$\eta / (2\alpha)$})}, for $t\geq s\vee \varepsilon$,
  \begin{align*}
    \E[|R_{X^0}(s,t)|^\alpha]
    &\leq e^{C(\varepsilon-s)^+}\E[|R_Y(s\vee\varepsilon,t)|^{2\alpha}]^{\frac12}
    \E\left[\left(\frac{\rmd \mu_\varepsilon}{\eV}(Y_\varepsilon)\right)^2\right]^{\frac12}\\
    &\leq Ce^{C(\varepsilon-s)^+}e^{-\frac{t-s\vee\varepsilon}{C}}\leq Ce^{C\varepsilon}e^{-\frac{t-s}{C}}.
  \end{align*}
This estimation remains valid for $0\leq s\leq t\leq\varepsilon$ up to increasing $C$, since then, by \eqref{eq:R_T_Linfty} and the fact  that $- \min{ \rm
   Spec}(\nabla^2 V(x)) \le  C < \infty$, $|R_{X^0}(s,t)|\leq
 e^{C\varepsilon}$. In conclusion, for $\alpha\geq 1$,  under Assumptions~{\bf
   (\hypdelta)},~{\bf(\hyppoinc{$\eta$})} and~{\bf (\hypspecbeta{$\eta / (2\alpha) $})}, $\sup_{t\geq 0}\E[|T_t|^\alpha]<\infty$  if $\mu_0$ satisfies~\eqref{eq:mu0V}.
\end{rem}

\section{The Green-Kubo formulae}\label{sect:GK}

A first way to compute~the derivative~\eqref{eq:but} is to use the
Green-Kubo formula (see for example~\cite{hairer-majda-10} for a
mathematical approach and~\cite{chandler-87,evans-morriss-08} for physical motivations). This formula gives
an expression of~\eqref{eq:but} in terms of the
time autocorrelations of~$(X_t^0)_{t\geq0}$, where~$(X_t^0)_{t\geq0}$ satisfies~\eqref{eq:EDS}
  with an initial condition~$X_0$ being distributed according to the 
  equilibrium measure~$\pi_0$.

\subsection{Finite time Green-Kubo formula}
We start with the Green-Kubo formula in finite time, which will not be
used in the sequel of the paper, but motivates the infinite horizon Green-Kubo formula.
\begin{theo}\label{theo:GK_tps_fini}
Let~$f \in\LL^1(\eV)$ be a Lipschitz function and let $\nabla f$ be its
gradient in the sense of distributions which can be identified with its
almost everywhere gradient. Suppose that the initial condition $X_0$ is distributed according to the 
equilibrium measure~$\pi_0$ and that Assumption {\bf(\hypspec)}~is
satisfied. Then, for any $t\geq 0$, for any $\lambda\in[0,\lambda_0]$, $f(X_t^\lambda)
$ is integrable and $\lambda \mapsto  \E_{\pi_0}[f(X_t^\lambda)]$ is
differentiable at $0$ with derivative
\begin{equation}\label{eq:GK_tps_fini}
\dlam\E_{\pi_0}[f(X_t^\lambda)]
=\int_0^t\E_{\pi_0}\left[\nabla f(X_0)\cdot R_{X^0}^T(0,s)\dlam F_\lambda(X_s^0)\right]\rmd s.
\end{equation}
\end{theo}
\begin{proof}
Since $X^0_t$ is distributed according to $\eV$, Proposition \ref{proptt} and the chain rule ensure that $\lambda\mapsto f(X^\lambda_t)$ is a.s. differentiable at $\lambda=0$ with derivative $\nabla f (X^0_t).T_t$.\\
To justify the interchange between the derivation and the expectation,
we need some integrability property. For $\lambda\in (0,\lambda_0]$ and $t\geq 0$, one has, using  {\bf(\hypspec)} and {\bf(\hypcadre)}-$(i)$ for the inequality:
\begin{align*}
   |X^\lambda _t-X^0_t|^2&=2\int_0^t(F_\lambda(X^\lambda_s)+\nabla V(X^\lambda_s)).(X^\lambda_s-X^0_s)\rmd s+2\int_0^t(\nabla V(X^0_s)-\nabla V(X^\lambda_s)).(X^\lambda_s-X^0_s)\rmd s\\
&\leq C\lambda^2t+\left(1-2\alpha\right)\int_0^t|X^\lambda _s-X^0_s|^2\rmd s.
\end{align*}
As a consequence,
\begin{equation}
   \forall \lambda\in (0,\lambda_0],\;\frac{|X^\lambda _t-X^0_t|^2}{\lambda^2}\leq C\frac{e^{(1-2\alpha)t}-1}{1-2\alpha}\label{majodif}
\end{equation} with the convention that the last ratio is equal to $t$
if $1-2\alpha=0$. With the Lipschitz continuity of~$f$, one deduces that the random variable $\frac{f(X^\lambda_t)-f(X^0_t)}{\lambda}$ is bounded by a deterministic constant not depending on $\lambda$. The integrability of $f(X^\lambda_t)$ then follows from the integrability of $f(X^0_t)$ where $X^0_t$ is distributed according to $\eV$ and $f\in \LL^1(\eV)$. 

Moreover, by Lebesgue's theorem, $\lambda \mapsto  \E_{\pi_0}[f(X_t^\lambda)]$ is differentiable at $0$ with derivative
\begin{equation}\label{eq:green_kubo_preuve}
\dlam\E_{\pi_0}[f(X_t^\lambda)]
=\E_{\pi_0}[\nabla f(X_t^0)\cdot T_t]
=\int_0^t\E_{\pi_0}\left[\nabla f(X_t^0)\cdot R_{X^0}(s,t)\dlam F_\lambda(X_s^0)\right]\rmd s,
\end{equation}
where we used~\eqref{eq:expression_Tt} for the second equality.
All terms in Equation~\eqref{eq:green_kubo_preuve} are well defined
thanks to Lemma~\ref{lem:R_T_Linfty}.
Let us now rewrite the right-hand side of~\eqref{eq:green_kubo_preuve}. By introducing the process $(Y_s)_{0 \le s \le t}=(X^0_{t-s})_{0\leq s\leq t}$ (which
has the same law as $(X^0_s)_{0 \le s \le t}$), using a change of
variable $s\to t-s$ and Lemma~\ref{lem:retournement_resolvante}, we get
\begin{align*}
\int_0^t\E_{\pi_0}\left[\nabla f(X_t^0)\cdot R_{X^0}(s,t)\dlam F_\lambda(X_s^0)\right]\rmd s
&=\int_0^t\E_{\pi_0}\left[\nabla f(X_{t}^0)\cdot R_{X^0}(t-s,t)\dlam F_\lambda(X^0_{t-s})\right]\rmd s\\
&=\int_0^t\E_{\pi_0}\left[\nabla f(Y_{0})\cdot R_{Y}^T(0,s)\dlam F_\lambda(Y_{s})\right]\rmd s\\
&=\int_0^t\E_{\pi_0}\left[\nabla f(X_0)\cdot R_{X^0}^T(0,s)\dlam
  F_\lambda(X_s^0)\right] \rmd s.
\end{align*}
This completes the proof of~\eqref{eq:GK_tps_fini}.
\end{proof}
\begin{rem}The conclusion of Theorem \ref{theo:GK_tps_fini} still holds if $f \in\LL^1(\eV)$ is a ${\mathcal C}^1$ function such that $\nabla f$ is uniformly continuous on $\R^d$ and $\nabla f\in\LL^1(\eV)$.
\end{rem}

It is possible to give another expression of the right-hand side
in~\eqref{eq:GK_tps_fini}.
\begin{prop}\label{prop:GK_formule_magique}
Let~$f \in\LL^2(\eV)$ be a Lipschitz function. Assume {\bf(\hypspec)}and consider the process~$(X_t^0)_{t\geq0}$ satisfying~\eqref{eq:EDS}
with an initial condition~$X_0$ being distributed according to the 
equilibrium measure~$\pi_0$.
 For almost every~$s \ge 0$ 
\begin{equation}\label{eq:GK_formule_magique}
\E_{\pi_0}\left[\nabla f(X_0)\cdot R_{X^0}^T(0,s)\dlam F_\lambda(X_s^0)\right]
=\E_{\pi_0}\left[
f(X_0)\left(\nabla V\cdot\dlam F_\lambda-\nabla\cdot \dlam F_\lambda\right)(X_s^0)
\right].
\end{equation}
\end{prop}
\begin{proof}
Since $f\in\LL^2(\eV)$, by Proposition~\ref{prop:EDP_bien_posee}, the partial differential equation
\[
\left\{
\begin{aligned}
\partial_tu(t,x)
&=\Delta u(t,x)
-\nabla V(x)\cdot \nabla u(t,x),
\,t>0,~x\in\R,\\
u(0,x)
&=f(x),
\,x\in\R.
\end{aligned}
\right.
\]
admits a unique solution $u$ in the sense of Definition \ref{defi:solution_EDP}.
Moreover, according to
Proposition~\ref{prop:feynman-kac},
\[
\forall s\geq 0,\;\rmd x\mbox{ a.e. },\;u(s,x)=\E[f(Y_s^x)]
,
\]
where~$(Y_t^x)_{t\geq0}$ solves~\eqref{eq:Ytx}. When $s>0$, from Lemmas~\ref{lem:p}, \ref{lemderci} and \ref{lem:R_T_Linfty} and Assumption {\bf(\hypspec)}, one can apply the dominated convergence theorem
to differentiate $\E[f(Y_s^x)]$ with respect
to~$x$, obtaining $\nabla_x\E[f(Y_s^x)]=\E\left[R_{Y^x}^T(0,s)\nabla f(Y_s^x)\right]$.  Since $u\in\bigcap_{T>0}\LL^2\left([0,T],\HH^1(\eV)\right)$, $\rmd s$ a.e., $u(s,.)$ admits a distributional gradient denoted by $\nabla u(s,.)$ and 
$$\rmd s\mbox{ a.e.},\;\rmd x\mbox{ a.e.},\;\nabla u(s,x)=\E\left[R_{Y^x}^T(0,s)\nabla f(Y_s^x)\right].$$
When $X_0$ is distributed according to $\pi_0$, from reversibility of the dynamics~\eqref{eq:EDS} and Lemma~\ref{lem:retournement_resolvante}, the random vectors $(X_0,X^0_s,R_{X^0}^T(0,s))$ and $(X^0_s,X_0,R_{X^0}(0,s))$ have the same law. Hence
\begin{equation}\label{eq:feynman-kac2}
\rmd s\mbox{ a.e.},\;\mbox{a.s.},\;\E_{\pi_0}\left[f(X_0)|X_s^0\right]
=u(s,X_s^0)
~\mbox{ and }~
\E_{\pi_0}\left[R_{X^0}(0,s)\nabla f(X_0)|X_s^0\right]
=\nabla u(s,X_s^0).
\end{equation}
For $s$ such that Equation~\eqref{eq:feynman-kac2} holds,
one deduces that
\begin{align*}
\E_{\pi_0}\left[\nabla f(X_0)\cdot R_{X^0}^T(0,s)\dlam F_\lambda(X_s^0)\right]
&=\E_{\pi_0}\left[\E_{\pi_0}\left[R_{X^0}(0,s)\nabla f(X_0)|X_s^0\right]\cdot\dlam F_\lambda(X_s^0)\right]\\
&=\E_{\pi_0}\left[\nabla u(s,X_s^0)\cdot\dlam F_\lambda(X_s^0)\right]\\
&=\int_\Rd\nabla u(s,x)\cdot\dlam F_\lambda(x)\eV\\
&=\int_\Rd u(s,x)\left(\nabla V(x)\cdot\dlam F_\lambda(x)-\nabla\cdot\dlam F_\lambda(x)\right)\eV\\
&=\E_{\pi_0}\left[f(X_0)\left(\nabla V(X_s^0)\cdot\dlam F_\lambda(X_s^0)-\nabla\cdot\dlam F_\lambda(X_s^0)\right)\right],
\end{align*}
where we used Lemma~\ref{lem:IPP_eV} below with $v(.)=u(s,.)$ which is
in~$\HH^1(\eV)$ for the last but one equality.
\end{proof}

\begin{lem}\label{lem:IPP_eV}
Let~$v$ be a function in~$\HH^1(\eV)$. Then
\[
\int_\Rd\nabla v(x)\cdot\dlam F_\lambda(x)\eV
=\int_\Rd v(x)\left(\nabla V(x)\cdot\dlam F_\lambda(x)-\nabla\cdot\dlam F_\lambda(x)\right)\eV.
\]
\end{lem}
\begin{proof}
Let~$\chi_n(x)=\chi(x/n)$ where~$\chi$ is a smooth,~$[0,1]$-valued, cutoff function such that~$\chi(x)=1$
for~$|x|<1$ and~$\chi(x)=0$ for~$|x|>2$. By integration by parts, one gets
\begin{align*}
\int_\Rd\chi_n(x)\nabla v(x)\cdot\dlam F_\lambda(x)\eV
&=\int_\Rd\chi_n(x)v(x)\left(\nabla V(x)\cdot\dlam F_\lambda(x)-\nabla\cdot\dlam F_\lambda(x)\right)\eV\\
&\quad-\int_\Rd v(x)\nabla\chi_n(x)\cdot\dlam F_\lambda(x)\eV.
\end{align*}
The result then follows from Lebesgue's theorem by taking the
limit~$n\to\infty$, using the fact that $\nabla V \in \LL^2(\eV)$,
$\dlam F_\lambda \in \LL^\infty$ and
$\nabla\cdot\dlam F_\lambda \in \LL^2(\eV)$  from Assumptions~{\bf(\hypV)}-$(ii)$ and~{\bf(\hypcadre)}-$(i)$-$(ii)$.
\end{proof}

By combining~\eqref{eq:GK_tps_fini} and~\eqref{eq:GK_formule_magique},
one gets: for any~$t \ge 0$,
$$\dlam\E_{\pi_0}[f(X_t^\lambda)]
=\int_0^t \E_{\pi_0}\left[
f(X_0)\left(\nabla V\cdot\dlam F_\lambda-\nabla\cdot \dlam
  F_\lambda\right)(X_s^0) \right] \rmd s,$$
where, we recall, the process~$(X_t^0)_{t\geq0}$ satisfies~\eqref{eq:EDS}
with an initial condition~$X_0$ being distributed according to the 
equilibrium measure~$\pi_0$.
Taking formally the limit~$t\to\infty$, one
obtains the classical Green-Kubo formula discussed in the next section.

\subsection{Infinite time Green-Kubo formula}

\begin{theo}\label{theo:GK}
  Consider the process~$(X_t^0)_{t\geq0}$ satisfying~\eqref{eq:EDS}
  with an initial condition~$X_0$ being distributed according to the 
  equilibrium measure~$\pi_0$. 
  Assume that Assumption {\bf(\hyppoinc{$\eta$})}~holds  for some positive
$\eta$.
  Then, for any~$f\in\LL^2(\eV)$, $\lambda\mapsto \int_\Rd f\rmd\pi_\lambda$ is differentiable at $\lambda=0$ with derivative
  \begin{equation}\label{eq:GK}
    \dlam\int_\Rd f\rmd\pi_\lambda
    =\int_0^\infty\E_{\pi_0}\left[
      f(X_0)\left(\nabla V\cdot\dlam F_\lambda-\nabla\cdot \dlam F_\lambda\right)(X_s^0)
    \right]\rmd s.
  \end{equation}
\end{theo}

Let us recall some results and notation from Section~\ref{sec:inv_meas}. The generator of
the process~$(X_t^0)_{t\geq0}$ is $\mathcal L_0=-\nabla V\cdot\nabla+\Delta$. The generator of
the process~$(X_t^\lambda)_{t\geq0}$ is $\mathcal
L_\lambda=F_\lambda\cdot\nabla+\Delta=\mathcal L_0+\mathcal T_\lambda$ where $\mathcal T_\lambda
=(F_\lambda+\nabla V)\cdot\nabla$. The domain of the operator
$\mathcal L_0$ is
\[
\mathcal D(\mathcal L_0)
=\left\{ v \in \LL^2_0(\eV)\cap\HH^1(\eV),\, \mathcal L_0 v  \in \LL_0^2(\eV)\right\}.
\]
For any $f\in\LL^2(\eV)$, there
exists a unique function~$g=- \mathcal L_0^{-1}f$ in~$\LL^2_0(\eV)\cap\HH^1(\eV)$ such that
\[
\forall v\in\LL^2_0(\eV)\cap\HH^1(\eV),~
\int_\Rd \nabla g(x)\cdot\nabla v(x)\eV=\int_\Rd f(x)v(x)\eV.
\]

Let us start with a lemma which is a consequence of the results of
Section~\ref{sec:longtime} on the long-time behaviour of $\E(f(Y_t^x))$.
\begin{lem}\label{lem:transformee_laplace_L0}
Let us assume that Assumption {\bf(\hyppoinc{$\eta$})}~holds  for some positive
$\eta$. Let us introduce 
the semigroup~$P_t$ associated to the Markovian
evolution~\eqref{eq:EDS}: for any $f \in \LL^2_0(\eV)$,
$$P_t f(x) = \E(f(Y_t^x))$$
where $(Y_t^x)_{t \ge 0}$ satisfies~\eqref{eq:Ytx}. We then have the following Laplace
inversion formula for the operator~$\mathcal L_0^{-1}$: for any~$f$ in~$\LL^2_0(\eV)$,
\begin{equation}\label{eq:poisson}
-\mathcal L_0^{-1}f
=\int_0^\infty P_tf \, \rmd t.
\end{equation}
\end{lem}
\begin{proof}
{F}rom Proposition~\ref{prop:feynman-kac}, we know that $u(t,x)=P_t f (x) = \E [
f(Y_t^x) ]$ is well defined in~$\LL^2_0(\eV)$ and satisfies the
partial differential equation~\eqref{eq:EDP_phi} in the sense of Definition~\ref{defi:solution_EDP}. From
Proposition~\ref{prop:tps_long}, (since $\int_{\R^d} f(x) \eV =0$) 
\begin{equation}\label{eq:expo_decay}
\forall t \ge 0, \, \left\|P_t f  \right\|_{\LL^2(\eV)}
\leq e^{-\eta t} \left\|f  \right\|_{\LL^2(\eV)}.
\end{equation}
This shows that $\int_0^\infty P_t f \, \rmd t$  is well
defined in~$\LL^2_0(\eV)$. Moreover, an adaptation with $\lambda=\phi=0$ (and
thus $C_0=C=0$) of the first energy
estimate in the proof of Proposition~\ref{prop:EDP_bien_posee} shows
that
\begin{equation}\label{eq:PsfH1}
\int_0^\infty \int_\Rd |\nabla P_t f(x)|^2 \eV \, \rmd t< \infty.
\end{equation}
Therefore, $\int_0^\infty P_t f \, \rmd t \in \HH^1(\eV)$.

{From} Definition~\ref{defi:solution_EDP}, for any test function $v \in
\LL^2(\eV)\cap\HH^1(\eV)$,
$$\int_\Rd f(x) v(x) \eV= \int_\Rd P_tf(x) v(x) \eV + \int_0^t
\int_\Rd \nabla P_s f (x) \cdot \nabla v(x) \eV \, \rmd s.$$
{From} Equation~\eqref{eq:expo_decay}, $\lim_{t \to \infty}  \int_\Rd
P_t f(x) v(x) \eV=0$ and thus, from~\eqref{eq:PsfH1},  for any test function $v \in
\LL^2(\eV)\cap\HH^1(\eV)$,
$$\int_\Rd f(x) v(x) \eV=
\int_\Rd \nabla \left( \int_0^\infty P_t f (x) \rmd t \right) \cdot \nabla v(x) \eV.$$
This concludes the proof.
\end{proof}

We recall that  $\Pi_0$ the orthogonal projection
from~$\LL^2(\eV)$ onto~$\LL^2_0(\eV)$ (see~\eqref{eq:PI0}).
We can now give a different expression for the right-hand side
of~\eqref{eq:GK}.
{From} Lemma~\ref{lem:IPP_eV} applied to the constant $1$, one has
$\int_\Rd(\nabla V\cdot\dlam F_\lambda-\nabla\cdot\dlam F_\lambda)(x)\eV=0$.
Then, using successively this equality, the self-adjointness of~$P_t$
in~$\LL^2(\eV)$ (which is a direct consequence of~\eqref{eq:rev}),
Equation~\eqref{eq:poisson} and Lemma \ref{lem:IPP_eV}, one has,
\begin{align*}
&\int_0^\infty \E_{\pi_0}\left[f(X_0)(\nabla V\cdot\dlam
  F_\lambda-\nabla\cdot\dlam F_\lambda)(X^0_t)\right]\, \rmd t\\
&=\int_0^\infty\int_\Rd\E\left[f(x)(\nabla V\cdot\dlam
  F_\lambda-\nabla\cdot\dlam F_\lambda)(Y_t^x)\right]\eV \,\rmd t\\
&=\int_0^\infty\int_\Rd\E\left[\Pi_0f(x)(\nabla V\cdot\dlam
  F_\lambda-\nabla\cdot\dlam F_\lambda)(Y_t^x)\right]\eV \, \rmd t\\
&=\int_0^\infty\int_\Rd\Pi_0f(x)P_t(\nabla V\cdot\dlam
F_\lambda-\nabla\cdot\dlam F_\lambda)(x)\eV \, \rmd t\\
&=\int_0^\infty\int_\Rd(P_t\Pi_0f(x))(\nabla V\cdot\dlam
F_\lambda-\nabla\cdot\dlam F_\lambda)(x)\eV\, \rmd t\\
&=-\int_\Rd(\mathcal L_0^{-1}\Pi_0f)(x)(\nabla V\cdot\dlam F_\lambda-\nabla\cdot\dlam F_\lambda)(x)\eV\\
&=-\int_\Rd \Tl [(\mathcal L_0^{-1}\Pi_0f)](x)\eV,
\end{align*}
where~$\Tl$ stands for the
operator~$\dlam F_\lambda\cdot\nabla$ (consistently with the
definition~\eqref{eq:Tlambda} of $\mathcal T_\lambda$). As a consequence, proving
Equation~\eqref{eq:GK}
boils down to proving: for any~$f\in\LL^2(\eV)$,
\begin{equation}\label{eq:GK_revisite}
\dlam\int_\Rd\Pi_0f(x)\rmd\pi_\lambda(x)
=-\int_\Rd  \Tl [(\mathcal L_0^{-1}\Pi_0f)](x)\eV.
\end{equation}

We are now in position to complete the proof of  Theorem~\ref{theo:GK}.
\begin{proof}[Proof of Theorem~\ref{theo:GK}]

Recall that for~$\lambda$ small enough, the
operator~$I+(\mathcal T_\lambda\mathcal L_0^{-1}\Pi_0)^*$
is invertible from~$\LL^2(\eV)$ to itself with bounded inverse (see Lemma~\ref{lem:pil}). For such a small $\lambda$, one has the equality
\begin{equation}\label{eq:(I+T)^-1}
(I+(\mathcal T_\lambda\mathcal L_0^{-1}\Pi_0)^*)^{-1}
=I
-(\mathcal T_\lambda\mathcal L_0^{-1}\Pi_0)^*
+\mathcal R_\lambda,
\end{equation}
where (by~\eqref{eq:Ol}) the
remainder~$\mathcal R_\lambda
=(I+(\mathcal T_\lambda\mathcal L_0^{-1}\Pi_0)^*)^{-1}
((\mathcal T_\lambda\mathcal L_0^{-1}\Pi_0)^*)^2$
has a norm from ~$\LL^2(\eV)$ to itself of order~$\mathcal O(\lambda^2)$.
Thus, by the analytical formula for $\pi_\lambda$ obtained in Lemma~\ref{lem:pil},
\begin{align}
\int_\Rd f(x)\rmd\pi_\lambda(x)
-\int_\Rd f(x)\rmd\pi_0(x)
&=-\int_\Rd [ (\mathcal T_\lambda\mathcal L_0^{-1}\Pi_0)^*\mathbf1](x)f(x)\eV
+\int_\Rd [\mathcal R_\lambda \mathbf1](x)f(x)\rmd\pi_0\notag\\
&=-\int_\Rd [(\mathcal T_\lambda\mathcal L_0^{-1}\Pi_0)f](x)\eV
+\mathcal O(\lambda^2).\label{difespll0}
\end{align}
Since, according to {\bf(\hypcadre)}-$(i)$, $\frac{F_\lambda+\nabla V}{\lambda}$ is bounded by $C$ for $\lambda\in(0,\lambda_0]$, one has, by Lebesgue's theorem,
\[
\lim_{\lambda \to 0} \frac1\lambda\int_\Rd\mathcal T_\lambda\mathcal L_0^{-1}\Pi_0f(x)\eV
=\int_\Rd(\dlam\mathcal T_\lambda)\mathcal L_0^{-1}\Pi_0f(x)\eV.
\]
Dividing \eqref{difespll0} by $\lambda$ and taking the limit $\lambda\to 0$, one concludes that $\lambda\mapsto \int_\Rd f\rmd\pi_\lambda$ is differentiable at $\lambda=0$ and \eqref{eq:GK_revisite}
holds. 
\end{proof}

Combining the previous result with~\eqref{eq:GK_formule_magique}, we obtain the following corollary.
\begin{cor}\label{cor:GK}
Let~$f \in\LL^2(\eV)$ be a Lipschitz function.  Also assume that
the Assumptions~{\bf(\hyppoinc{$\eta$})} (for some positive
$\eta$)~and {\bf(\hypspec)}~are
satisfied. Then, one has
\[
\dlam\int_\Rd f\rmd\pi_\lambda
=\int_0^\infty\E_{\pi_0}\left[\nabla f(X_0)\cdot R_{X^0}^T(0,s)\dlam F_\lambda(X_s^0)\right]\rmd s.
\]
\end{cor}

\section{Long-time convergence of the estimators~\eqref{eq:estim_1}
  and~\eqref{eq:estim_2}}\label{sec:main_result}

\subsection{Statement of the main result}
Let us study the long-time behavior of the two estimators~\eqref{eq:estim_1} and~\eqref{eq:estim_2}.

\begin{theo}\label{theo:interversion}
 Let~$f:\R^d\to\R$ be a ${\mathcal C}^1$ function such that $\nabla f$ is bounded.\begin{itemize}
\item Assume the existence of $\eta>0$ such that either $V$ is $\eta$-convex or Assumptions {\bf(\hyppoinc{$\eta$})} and {\bf
    (\hypspecbeta{$\eta$})} hold. Then $\lambda\mapsto\frac{1}{t}\int_0^t f(X^\lambda_s)\rmd s$ is differentiable at $\lambda=0$ with derivative $\frac{1}{t}\int_0^t \nabla f(X^0_s).T_s \rmd s$ and \begin{equation}\label{eq:cv_estim2}
\lim_{t\to\infty}\dlam\left(\frac1t\int_0^tf(X_s^\lambda)\rmd s\right)
=\dlam\left(\int_\Rd f\rmd\pi_\lambda\right)
\mbox{ a.s.}.
\end{equation}
\item Assume either that $V$ is $\eta$-convex for a positive constant $\eta>0$ and $\E|X_0|<+\infty$, or that there exist $\eta>0$ and $p\in (1, \infty]$ such that ~{\bf(\hyppoinc{$\eta$})} holds, $X_0$ is distributed
  according to a measure having a density in $\LL^p(\eV)$ with respect to $\eV$ and  {\bf
    (\hypspecbeta{$\rho$})} holds for some $\rho<\eta(p-1)/p$ (with the convention $\eta(\infty-1)/\infty=\eta$). Then $\forall \lambda\in[0,\lambda_0]$, $\forall t\geq 0$, $\E|f(X^\lambda_t)|<+\infty$, $\lambda\mapsto\E[f(X^\lambda_t)]$ is differentiable at $\lambda=0$ with derivative $\dlam\E\left[f(X_t^\lambda)\right]=\E[\nabla f(X^0_t).T_t]$ and \begin{equation}\label{eq:cv_estim1}
\lim_{t\to\infty}\dlam\E\left[f(X_t^\lambda)\right]
=\dlam\left(\int_\Rd f\rmd\pi_\lambda\right).\end{equation}
\end{itemize}

\end{theo}
\begin{rem}\label{rem:mainres}
  When $\pi_0$ is assumed to satisfy a logarithmic Sobolev
  inequality with constant $\eta$, which is stronger than the
  Poincar\'e inequality {\bf(\hyppoinc{$\eta$})}, then the second statement still holds as soon as $X_0$ is distributed
  according to a measure having a density in $\LL^p(\eV)$ with respect to $\eV$ for some $p>1$ and  {\bf
    (\hypspecbeta{$\rho$})} holds for some $\rho < \eta$, because of the hypercontractivity property of the semi-group associated with~\eqref{eq:EDS} ensured by the Gross theorem.
\end{rem}

\subsection{Long-time behaviour of $(X^0_t,T_t)_{t\geq 0}$}

To prove Theorem~\ref{theo:interversion}, one first needs to know the
long-time limit of the trajectory and its tangent vector. We more generally consider $(X^0_t,T^0_t)_{t \ge0}$ solving
\begin{equation}\label{edsjointe}
\left\{
   \begin{aligned}
      \rmd X_t^0&=-\nabla V(X_t^0)\rmd t+\sqrt2\rmd W_t\\
  \rmd T^0_t
&=\left(\dlam F_\lambda(X_t^0)-\nabla^2V(X_t^0) T^0_t\right)\rmd t 
\end{aligned}
\right.
\end{equation}
with $(X^0_0,T^0_0)$
any initial condition independent from the Brownian motion $(W_t)_{t\geq 0}$. To write
conveniently the long-time limit of  $(X^0_t,T^0_t)$, we
will run time backward and use Lemma~\ref{lem:retournement_resolvante}
about the link between
the forward resolvent and its backward counterpart. 
\begin{lem}\label{lem:conv_loi}

Under
Assumptions {\bf(\hyppoinc{$\eta$})} (for some positive
$\eta$), {\bf(\hypspec)} and~{\bf(\hypconv)}, the
process~$(X_t^0,T^0_t)_{t\geq0}$ converges in
law as~$t$ goes to infinity to the couple
\[
\left(Y_0,\int_0^\infty R_Y^T(0,t)\dlam F_\lambda(Y_t)\rmd t\right),
\]
where~$(Y_t)_{t\geq0}$ follows the dynamics~\eqref{eq:EDS} with
with~$Y_0$ distributed according to~$\eV$. Moreover, the law ${\cal
  V}$ of this couple is invariant by the dynamics \eqref{edsjointe}
and ergodic for this dynamics: for any test function $\varphi: \Rd\times\Rd
 \to \R$ in $\LL^1({\cal V})$, for ${\cal V}$-a.e. deterministic
 initial condition $(X^0_0,T^0_0)$,
$$\lim_{t \to \infty} \frac 1 t \int_0^t \varphi(X^0_s,T^0_s) \, \rmd s = \int_{\Rd \times
  \Rd} \varphi(x,\tau) \, \rmd {\cal V}(x,\tau)\;a.s..$$
\end{lem}
\begin{proof}
The integral~$\int_0^\infty R_Y^T(0,t)\dlam F_\lambda(Y_t)\rmd t$ is
almost surely well defined, from Lemma~\ref{lem:convergence_R} and the
boundedness of~$\dlam F_\lambda$.
To prove Lemma~\ref{lem:conv_loi}, we are going to use a time reversal argument.

For $t_0 >0$, we construct a coupling of the trajectory~$(X^0_t)_{t\geq t_0}$ with
another process~$(\chi_t^{t_0})_{t\geq t_0}$
following the dynamics~\eqref{eq:EDS}, but being at
equilibrium.  Denote by
$q_{t_0}$ the density of the distribution of~$X_{t_0}^0$ (which exists
by Lemma~\ref{lem:p}), and define
$\rho_{t_0}=\frac{q_{t_0}\wedge e^{-V}}{q_{t_0}}$.
Let~$U$ and~$\zeta_{t_0}$ be mutually independent random variables which are
independent of~$X_0$ and of the Brownian motion~$(W_t)_{t\geq0}$
driving~$(X^0_t)_{t\geq0}$, such
that~$U$ is uniformly distributed over~$[0,1]$ and, when $q_{t_0}\neq
e^{-V}$, $\zeta_{t_0}$ is distributed according
to~$C(e^{-V(x)}-q_{t_0}(x))^+\rmd x$,~$C$ being a normalization constant
($\zeta_{t_0}$ does not need to be defined when $q_{t_0}=e^{-V}$). We define the position
of the process~$(\chi^{t_0}_t)_{t\geq{t_0}}$ at time~${t_0}$ by
$\chi_{t_0}^{t_0}=X_{t_0}^0\mathbf1_{U\leq\rho_{t_0}(X_{t_0}^0)}+\zeta_{t_0}\mathbf1_{U>\rho_{t_0}(X_{t_0}^0)}$,
which
is distributed according to~$\pi_0$. One has~$\PP(\chi_{t_0}^{t_0}\neq
X_{t_0}^0)=\frac12\|q_{t_0}(x)-e^{-V(x)}\|_{\LL^1(\rmd x)}$. For~$t>{t_0}$,
let~$(\chi_t^{t_0})_{t\geq{t_0}}$ evolve according to the dynamics
\eqref{eq:EDS} with Brownian motion~$(W_t)_{t\geq0}$. Notice that
 $(\chi_{t+{t_0}}^{t_0})_{t \ge 0}$ has the same law as the process at
 equilibrium $(Y_t)_{t \ge
   0}$ introduced in the statement of Lemma~\ref{lem:conv_loi}.  Moreover, $(\chi_t^{t_0})_{t \ge t_0}$ is
 such that $$\PP(\forall t\geq{t_0}, \chi_t^{t_0}=X_t^0)=1-\frac12\|q_{t_0}(x)-e^{-V(x)}\|_{\LL^1(\rmd x)}.$$

{From} an easy adaptation of Proposition~\ref{prop:tangent}, one has on the one hand
\[
(X_t^0,T^0_t)=\left(X_t^0,R_{X^0}(0,t)T_0^0+\int_0^tR_{X^0}(s,t)\dlam F_\lambda(X_s^0)\rmd s\right).
\]
On the other hand, for~$0 < {t_0}\leq t$, we have
the equalities,
by using successively the time translation~$s\to s-{t_0}$, the change of
variable~$u=t-{t_0}-s$,
Lemma~\ref{lem:retournement_resolvante} (using the notation, for $u
\in [0, t-{t_0}]$, $Z_u = Y_{t-{t_0}-u}$)  and the time
reversibility of the dynamics~\eqref{eq:EDS}:
\begin{align*}
\left(
\chi_t^{t_0},
\int_{t_0}^tR_{\chi^{t_0}}(s,t)\dlam F_\lambda(\chi_s^{t_0})\rmd s
\right)
&\stackrel{\mathcal D}=\left(
Y_{t-{t_0}},
\int_0^{t-{t_0}}R_Y(s,t-{t_0})\dlam F_\lambda(Y_s)\rmd s
\right)\nonumber\\
&=\left(
Y_{t-{t_0}},
\int_0^{t-{t_0}}R_Y(t-{t_0}-u,t-{t_0})\dlam F_\lambda(Y_{t-{t_0}-u})\rmd u
\right)\nonumber\\
&=
\left(Z_0,\int_0^{t-{t_0}}R_Z^T(0,u)\dlam F_\lambda(Z_u)\rmd u\right)\nonumber\\
&\stackrel{\mathcal D}=
\left(Y_0,\int_0^{t-{t_0}}R_Y^T(0,s)\dlam F_\lambda(Y_s)\rmd s\right),
\end{align*}
where $\stackrel{\mathcal D}=$ stands for the equality in distribution.
As a consequence, for any bounded Lipschitz 
function~$\phi:\R^d\times\R^d\to\R$ (with Lipschitz constant ${\rm Lip}(\phi)$), for $t \ge {t_0} > 0$
\begin{align}
&\left|\E\left[\phi\left(Y_0,\int_0^{\infty}
R_Y^T(0,s)\dlam F_\lambda(Y_s)\rmd s\right)-\phi(X_t^0,T^0_t)\right]\right|\nonumber\\
&\leq \left|\E\left[\phi\left(Y_0,\int_0^{\infty}
R_Y^T(0,s)\dlam F_\lambda(Y_s)\rmd s\right)-\phi\left(Y_0,\int_0^{t-{t_0}}
R_Y^T(0,s)\dlam F_\lambda(Y_s)\rmd s\right)\right]\right|\nonumber\\
&+\bigg|\E\bigg[\phi\left(\chi_t^{t_0},\int_{t_0}^tR_{\chi^{t_0}}(s,t)\dlam F_\lambda(\chi_s^{t_0})\rmd s\right)\nonumber\\
&\phantom{+\bigg|\E\bigg[}-\phi\left(\chi_t^{t_0},R_{X^0}(0,t)T_0^0+\int_0^{t_0} R_{X^0}(s,t)\dlam F_\lambda(X_s^0)\rmd s
+\int_{t_0}^tR_{\chi^{t_0}}(s,t)\dlam F_\lambda(\chi_s^{t_0})\rmd s\right)\bigg]\bigg|\nonumber\\&+\left|\E\left[\phi\left(\chi_t^{t_0},R_{X^0}(0,t)T_0^0+\int_0^{t_0} R_{X^0}(s,t)\dlam F_\lambda(X_s^0)\rmd s
+\int_{t_0}^tR_{\chi^{t_0}}(s,t)\dlam F_\lambda(\chi_s^{t_0})\rmd
                                                                                                   s\right)-\phi(X_t^0,T^0_t)\right]\right|\nonumber
  \\
&\leq \left|\E\left[\phi\left(Y_0,\int_0^{\infty}
R_Y^T(0,s)\dlam F_\lambda(Y_s)\rmd s\right)-\phi\left(Y_0,\int_0^{t-{t_0}}
R_Y^T(0,s)\dlam F_\lambda(Y_s)\rmd s\right)\right]\right|\nonumber\\
&+\E\left[
\left(2\|\phi\|_{\LL^\infty(\Omega)}\right)\wedge\left( {\rm Lip}(\phi)
\left|R_{X^0}({t_0},t)\left(R_{X^0}(0,{t_0})T_0^0+\int_0^{t_0} R_{X^0}(s,{t_0})\dlam F_\lambda(X_s^0)\rmd s\right)\right|\right)
\right]\nonumber\\
&+2\|\phi\|_{\infty}\PP(X_{t_0}^0\neq \chi_{t_0}^{t_0}).
 \label{eq:retourn_temps}
\end{align}
Notice that we used the 
semi-group property of $R_{X^0}$ to obtain the last but one inequality.
The first term in the right-hand side converges to $0$ as $t\to\infty$ by Lebesgue's theorem. A direct adaptation of Lemma~\ref{lem:convergence_R} shows that~$R_{X^0}({t_0},t)$
goes to $0$ as~$t$ goes to infinity, yielding from Lebesgue's theorem
that the second term in the right-hand side of~\eqref{eq:retourn_temps}
goes to $0$ as~$t$ goes to infinity. The third term in the right-hand
side of~\eqref{eq:retourn_temps} can be rewritten as
$2\|\phi\|_{\infty}\PP(X_{t_0}^0\neq \chi_{t_0}^{t_0})
=\|\phi\|_{\LL^\infty(\Omega)}\|e^{-V(x)}-p_{t_0}(x)\|_{\LL^1(\rmd x)}$
and thus goes to $0$ as~${t_0}$ goes to infinity, by Corollary~\ref{cor:CV_L1}.
Letting $t\to\infty$ and then
$t_0\to\infty$ in~\eqref{eq:retourn_temps}, we conclude that 
the couple~$(X_t^0,T^0_t)$
converges in law to $\left(Y_0,\int_0^{\infty}
R_Y^T(0,s)\dlam F_\lambda(Y_s)\rmd s\right)$ with law ${\cal V}$.

To check that ${\cal V}$ is invariant by the dynamics \eqref{edsjointe}, we denote by $({\cal P}_{s})_{s\geq 0}$ the Markov semi-group associated with this dynamics. One has $\E[{\cal P}_s\varphi(X^0_t,T^0_t)]=\E[\varphi(X^0_{t+s},T^0_{t+s})]$ where the right-hand side converges to $\int_{\R^d\times\R^d}\varphi(x,\tau)\rmd {\cal V}(x,\tau)$ as $t\to\infty$ and ${\cal P}_s\varphi(x,\tau)=\E[\varphi(Y^x_s,T^{x,\tau}_s)]$ with
\begin{equation*}
\left\{   \begin{aligned}
      Y^x_t&=x-\int_0^t\nabla V(Y^x_s)\rmd s+\sqrt2W_t \, ,\\
  T^{x,\tau}_t 
&=\tau+\int_0^t\left(\dlam F_\lambda(Y^x_s)-\nabla^2V(Y^x_s)
  T^{x,\tau}_s\right)\rmd s \, .\end{aligned}
\right.
\end{equation*}
The continuity of $x\mapsto (Y^x_t)_{t\geq 0}$ for the topology of
local uniform convergence on ${\mathcal C}({\mathbb R}_+,\R^d)$
together with the continuity of $\nabla^2V$ implies the continuity of $x\mapsto
(R_{Y^x}(s,t))_{s,t\geq 0}$ for the topology of local uniform
convergence on ${\mathcal C}({\mathbb R}_+\times{\mathbb R}_+,
\R^{d\times d})$. With the continuity of $\dlam F_\lambda$, one
deduces the continuity of $(x,\tau)\mapsto
T^{x,\tau}_s=R_{Y^x}(0,s)\tau+\int_0^sR_{Y^x}(r,s)\dlam
F_\lambda(Y^x_r)\rmd r$. Hence, by Lebesgue's theorem, ${\cal
  P}_s\varphi(x,\tau)$ is continuous and bounded and $\E[{\cal
  P}_s\varphi(X^0_t,T^0_t)]$ converges to $\int_{\R^d\times\R^d}{\cal
  P}_s\varphi(x,\tau)\rmd {\cal V}(x,\tau)$ as $t\to\infty$. Therefore
$\int_{\R^d\times\R^d}{\cal P}_s\varphi(x,\tau)\rmd {\cal
  V}(x,\tau)=\int_{\R^d\times\R^d}\varphi(x,\tau)\rmd {\cal
  V}(x,\tau)$ and the probability measure ${\cal V}$ is invariant.

Since ${\cal V}$ is the unique invariant probability measure for the
SDE \eqref{edsjointe}, this measure is ergodic (see for example~\cite[Theorem 3.8
and Equation (52)]{rey-bellet-06}).
\end{proof}
 Let us deduce from the previous results the limit of
 $\frac{1}{t}\int_0^t\varphi(X^0_s) \cdot T^0_s\rmd s$ where $\varphi:\R^d\to\R^d$ is measurable and bounded.
\begin{lem}\label{lem:ergconv}Assume the existence of $\eta>0$ such
  that either $V$ is $\eta$-convex or Assumptions {\bf(\hyppoinc{$\eta$})} and {\bf
    (\hypspecbeta{$\eta$})} hold. Then
  $\int_{\R^d\times\R^d}|\tau|\rmd {\cal V}(x,\tau)<\infty$ (where the
  probability distribution~${\cal V}$ has been introduced in Lemma~\ref{lem:conv_loi}) and for any function $\varphi:\R^d\to\R^d$ measurable and bounded, $\frac{1}{t}\int_0^t\varphi(X^0_s)\cdot T^0_s\rmd s$ converges a.s. to $\int_{\R^d\times\R^d}\varphi(x)\cdot\tau\rmd {\cal V}(x,\tau)$ as $t\to\infty$ whatever the choice of the initial condition $(X^0_0,T^0_0)$ independent of the Brownian motion $(W_t)_{t\geq 0}$.
\end{lem}
\begin{proof}Notice that if $V$ is $\eta$-convex,
  Assumptions~{\bf(\hypspec)},~{\bf(\hypconv)} and~{\bf(\hyppoinc{$\eta$})} hold.  In addition, Assumption {\bf
    (\hypspecbeta{$\eta$})} implies Assumptions~{\bf(\hypspec)}
  and~{\bf(\hypconv)}. Therefore, the conclusion of Lemma \ref{lem:conv_loi} holds under the two classes 
of hypotheses considered.  Since $\dlam F_\lambda$ is bounded, one has
$$\int_{\R^d\times\R^d}|\tau|\rmd {\cal V}(x,\tau)\leq \int_0^\infty \E\left[|R_Y^T(0,t)||\dlam F_\lambda(Y_t)|\right]\rmd t\leq C\int_0^\infty \E\left[|R_Y(0,t)|\right]\rmd t ,$$
where the right-hand side is finite by \eqref{eq:R_T_Linfty} when $V$
is $\eta$-convex and by Proposition \ref{cor:momr} otherwise.

In case the law of $(X_0^0,T^0_0)$ is absolutely continuous with
respect to ${\cal V}$, the result of Lemma~\ref{lem:ergconv} is then a
direct consequence of the ergodic property of the process
$(X^0_t,T^0_t)_{t \ge 0}$ stated in Lemma~\ref{lem:conv_loi}. 

To extend this result to more general initial conditions, we proceed as
follows. By Lemma~\ref{lem:densite_X_t}, the law of $X^0_1$ is absolutely
continuous with respect to $\pi_0$ which is the marginal law of the
$d$ first coordinates for the ergodic measure ${\cal V}$. Let
$d{\cal V}_{T|X=x}(\tau)$ denote a regular conditional probability
distribution of the $d$ last coordinates given the $d$ first ones
under ${\cal V}$ and $\tilde{T}^0_1$ be a random vector independent of
$(W_t-W_1)_{t\geq 1}$ with conditional law given $X^0_1$ equal to
$d{\cal V}_{T|X=X^0_1}(\tau)$. Let for $t\geq 1$,
$\tilde{T}^0_t=R_{X^0}(1,t)\tilde{T}^0_1+\int_1^tR_{X^0}(s,t)\dlam
F_\lambda(X_s^0)\rmd s$. Then $\rmd \tilde{T}^0_t=\left(\dlam
  F_\lambda(X_t^0)-\nabla^2V(X_t^0) \tilde{T}^0_t\right)\rmd t$ so
that $(X^0_t,\tilde{T}^0_t)_{t\geq 1}$ solves \eqref{edsjointe}
starting from $(X^0_1,\tilde{T}^0_1)$ the law of which is absolutely
continuous with respect to the measure ${\cal V}$ ergodic for this
stochastic differential equation (see Lemma~\ref{lem:conv_loi}). As a consequence $\frac{1}{t}\int_1^t\varphi(X^0_s)\cdot\tilde{T}^0_s\rmd s$ converges a.s. to $\int_{\R^d\times\R^d}\varphi(x)\cdot\tau\rmd {\cal V}(x,\tau)$ as $t\to\infty$.
Now, by an adaptation of Proposition
\ref{prop:tangent}, one can check that for $t \ge 1$,
$\tilde{T}^0_t-T^0_t=R_{X^0}(1,t)(\tilde{T}^0_1-T^0_1)=R_{X^0}(0,t)R_{X^0}(1,0)(\tilde{T}^0_1-T^0_1)$
so that
\begin{align*}
\frac{1}{t}\int_0^t\varphi(X^0_s)\cdot T^0_s\rmd
s-\frac{1}{t}\int_1^t\varphi(X^0_s)\cdot\tilde{T}^0_s\rmd
s&=\frac{1}{t}\int_0^1\varphi(X^0_s)\cdot T^0_s\rmd
s\\
&+\frac{1}{t}\int_1^t\varphi(X^0_s)\cdot R_{X^0}(0,s)\rmd
sR_{X^0}(1,0)(T^0_1-\tilde{T}^0_1).
\end{align*}
The proof is completed by noticing that this quantity converges
a.s. to $0$ (for the second term in the right-hand side, this is a
consequence of the boundedness of $\varphi$ and of the almost sure estimate~\eqref{eq:convergence_R}).
\end{proof}

\subsection{Proof of Theorem~\ref{theo:interversion}}

We are now in position to prove Theorem~\ref{theo:interversion}.
\begin{proof}[Proof of Theorem~\ref{theo:interversion}]

Let us start by a preliminary result concerning the integrability of
the random variable $\int_0^\infty \left|\nabla f(X_0) \right|
  \left|  R_{X^0}^T(0,t)\dlam F_\lambda(X_t^0) \right| \rmd t$. 
By the boundedness of $\nabla f$ and $\dlam F_\lambda$
\begin{align*}
   \E_{\pi_0}\left[ \int_0^\infty \left|\nabla f(X_0) \right|
  \left|  R_{X^0}^T(0,t)\dlam F_\lambda(X_t^0) \right| \rmd t\right]
   \leq C\int_0^{+\infty}\E_{\pi_0}\left[|R_{X^0}(0,t)|\right] \rmd t,
\end{align*}
where, we recall, the subscript $\pi_0$ in $\E_{\pi_0}$ indicates that $X_0$ is
distributed according to $\pi_0$. If $V$ is $\eta$-convex,
using~\eqref{eq:R_T_Linfty}, almost surely,
$$|R_{X^0}(0,t)| \le e^{-\eta t}.$$
If  Assumptions~{\bf(\hyppoinc{$\eta$})} and {\bf
  (\hypspecbeta{$\eta$})}  hold for some positive $\eta$ (notice that {\bf(\hypspecbeta{$\rho$})} for $\rho<\frac{\eta(p-1)}{p}$ with $p\in(1,+\infty]$ implies {\bf(\hypspecbeta{$\eta$})}), then, by Proposition~\ref{cor:momr} ,
\[
\E_{\pi_0}\left[|R_{X^0}(0,t)|\right]
\leq Ce^{-t/C}
\]
for some positive constant $C$.
Hence, in all cases,
\begin{equation}\label{eq:justif_fubini}
 \E_{\pi_0}\left[ \int_0^\infty \left|\nabla f(X_0) \right|
  \left|  R_{X^0}^T(0,t)\dlam F_\lambda(X_t^0) \right| \rmd t\right]<+\infty.
\end{equation}

Let us now prove the first statement of
Theorem~\ref{theo:interversion}. By \eqref{majodif}, the boundedness
of $\nabla f$ and Proposition \ref{proptt}, Lebesgue's theorem implies
that $\lambda\mapsto\frac{1}{t}\int_0^t f(X^\lambda_s)\rmd s$ is
differentiable at $\lambda=0$ with derivative $\frac{1}{t}\int_0^t
\nabla f(X^0_s)\cdot T_s \rmd s$. By Lemma \ref{lem:ergconv},
$\frac{1}{t}\int_0^t \nabla f(X^0_s)\cdot T_s \rmd s$ converges
a.s. to $\int_{\R^d\times\R^d}\nabla f(x)\cdot \tau\rmd {\cal V}(x,\tau)=\E_{\pi_0}\left[\nabla f(X_0)\cdot\int_0^\infty R_{X^0}^T(0,s)\dlam F_\lambda(X_s^0)\rmd s\right]$.
The proof of~\eqref{eq:cv_estim2} is then completed by the following
computations:
\begin{align}
\E_{\pi_0}\left[\nabla f(X_0)\cdot\int_0^\infty R_{X^0}^T(0,s)\dlam F_\lambda(X_s^0)\rmd s\right]
=&\int_0^\infty\E_{\pi_0}\left[\nabla f(X_0)\cdot R_{X^0}^T(0,s)\dlam F_\lambda(X_s^0)\right]\rmd s\nonumber\\
=&\dlam\int_\Rd f(x)\rmd\pi_\lambda(x)\label{eq:fub_GK}
\end{align}
where we used Fubini's theorem and~\eqref{eq:justif_fubini} for the
first equality and Corollary~\ref{cor:GK} for the second one.

Let us finally deal with the second statement  of Theorem~\ref{theo:interversion}. By~\eqref{majodif}, the boundedness of $\nabla f$ and Proposition~\ref{proptt},
it is enough to check that $\E|f(X^0_t)|<+\infty$ to deduce that
$\forall \lambda\in[0,\lambda_0]$, $\E|f(X^\lambda_t)|<+\infty$ and
$\lambda\mapsto\E[f(X^\lambda_t)]$ is differentiable at $\lambda=0$
with derivative $\dlam\E\left[f(X_t^\lambda)\right]=\E[\nabla
f(X^0_t)\cdot T_t]$. When $\eV$ satisfies a Poincar\'e inequality,
then according to \cite{bobkovledoux} and the references therein,
since $f$ is a
Lipschitz function, there exists a positive $\varepsilon$ such that $\int_{\R^d}e^{\varepsilon
  |f|(x)}\eV<+\infty$. Therefore, when the law
$\mu_0$ of $X_0$ has a density with respect to $\eV$ in $\LL^p(\eV)$, by Lemma \ref{lem:densite_X_t}, 
$$\sup_{t\geq 0}\E|f(X^0_t)|\leq \|f\|_{\LL^{\frac{p}{p-1}}(\eV)}\left\|\frac{\rmd
          \mu_0}{\eV}\right\|_{\LL^p(\eV)}<+\infty.$$

When $V$ is $\eta$-convex, then computing $|Y^x_t|^2$ by It\=o's
formula, remarking that $$-2\nabla V(Y^x_t)\cdot Y^x_t=-2(\nabla
V(Y^x_t)-\nabla V(0))\cdot Y^x_t-2\nabla V(0)\cdot Y^x_t\leq
-\eta|Y^x_t|^2+\frac{|\nabla V(0)|^2}{\eta},$$ applying a localization
procedure to get rid of the expectation of the stochastic integral,
one obtains $\forall t\geq 0$, $\E[|Y^x_t|^2]\leq e^{-\eta
  t}|x|^2+\frac{1-e^{-\eta t}}{\eta}\left( \frac{|\nabla
    V(0)|^2}{\eta} + 2d\right)$. Hence, when the initial random
variable $X_0$ with law $\mu_0$ is integrable, 
\begin{align*}
\E|X^0_t|&\leq \int_{\R^d}\E|Y^x_t|\rmd\mu_0(x)\leq
           \int_{\R^d}\sqrt{\E[|Y^x_t|^2]}\rmd\mu_0(x)\\
&\leq\E\left[\sqrt{e^{-\eta t}|X_0|^2+\frac{1-e^{-\eta t}}{\eta}\left( \frac{|\nabla
    V(0)|^2}{\eta} + 2d\right) }\right] <+\infty,
\end{align*}
and $\E|f(X^0_t)|<+\infty$ by the Lipschitz continuity of $f$.

Notice that if $V$ is $\eta$-convex,
  Assumptions~{\bf(\hypspec)},~{\bf(\hypconv)}
  and~{\bf(\hyppoinc{$\eta$})} hold. Moreover, Assumption {\bf
    (\hypspecbeta{$\rho$})} implies Assumptions~{\bf(\hypspec)}
  and~{\bf(\hypconv)}. Therefore, the conclusion of Lemma \ref{lem:conv_loi} holds under the two classes 
of hypotheses considered. The function $(x,\tau) \mapsto \nabla f(x)
\cdot \tau$ is continuous and the family $(\nabla f(X_t^0)\cdot
T_t)_{t \ge 0}$ is uniformly integrable by Proposition \ref{prop:varalconv} when $V$ is $\eta$-convex and since  $\sup_{t \ge 0} \E
\left( \left|\nabla f(X_t^0)\cdot
T_t\right|^{\frac{\eta(p-1)}{\rho p}} \right) < \infty $, by Proposition~\ref{prop:variance}, in the second framework. Therefore  the convergence in distribution in
Lemma~\ref{lem:conv_loi} yields
$$\lim_{t \to \infty} \E\left( \nabla f(X_t^0)\cdot T_t\right) = \E_{\pi_0}\left[\nabla f(X_0)\cdot\int_0^\infty R_{X^0}^T(0,s)\dlam F_\lambda(X_s^0)\rmd s\right],$$
where the right-hand side is equal to $\dlam\int_\Rd f(x)\rmd\pi_\lambda(x)$ according to~\eqref{eq:fub_GK}.
\end{proof}

\section{Numerical illustrations}\label{sect:numeric}

In this section, we illustrate through various numerical experiments
the theoretical results obtained above. In Section~\ref{sec:1d}, we
study numerically on a one-dimensional toy model the integrability of
the tangent vector $T_t$ and the sharpness of the integrability exponent
obtained in Proposition~\ref{prop:variance}. In Section~\ref{sec:WA},
we illustrate the interest of the estimator~\eqref{eq:estim_2} on a more realistic
test case proposed in~\cite{warren-allen-12}. Finally, we investigate
in Section~\ref{sec:roland} on the one-dimensional toy model a
variance reduction method for the estimator~\eqref{eq:estim_2}.

\subsection{A one-dimensional toy model}\label{sec:1d}

In this section, we would like to study on a simple test case the integrability of the
tangent vector~$T_t$, and to compare the theoretical bounds obtained
in Proposition~\ref{prop:variance}, with a numerical estimation of the
integrability exponent. Let us consider the potential
\[
\forall x \in \R, \, V_\lambda(x)
=x^4-\frac c2x^2+\lambda x,
\]
where~$c$ is some fixed constant, and $\lambda \in \R$ is the parameter.
For $\lambda=0$,~$V_0$ has curvature~$-c$ at the origin, and for~$c>0$,~$V_0$ is a double-well
potential, with wells located at~$\pm{\sqrt c}/2$ and separated by a
barrier with height~$c^2/16$. In particular, as~$c$ gets larger, the
dynamics~\eqref{eq:EDS} of $(X^0_t)_{t \ge 0}$ becomes more and more
metastable.

Let us start with some explicit computation on $T_t$. When~$\lambda\geq0$, the perturbative force pushes the
system to the left. Therefore, one expects the tangent vector~$T_t$
to be negative in the mean.
In fact, one can prove that~$T_t$ is in that case
almost surely negative for~$t>0$. Indeed,~$(T_t)_{t\geq0}$ satisfies the equation
\[
\left\{
\begin{aligned}
  \partial_t T_t
  &= -1
  +(c-12(X_t^0)^2)T_t,\\
  T_0&=0,
\end{aligned}
\right.
\]
which can be solved explicitly, since in dimension 1, the
resolvent~$R_{X^0}(s,t)$ is given by the
exponential~$R_{X^0}(s,t)=\exp\left(c(t-s)-12\int_s^t(X_u^0)^2\rmd u\right)$.
Equation~\eqref{eq:expression_Tt} then becomes
\[
T_t
=-\int_0^t\exp\left(c(t-s)-12\int_s^t(X_u^0)^2\rmd u\right)\rmd s
<0.
\]

Concerning the upper bound on the integrability exponent obtained in Proposition~\ref{prop:variance}, if the initial condition~$X_0$
has a bounded density, then the tangent vector is bounded in~$\LL^\alpha$
uniformly in time for all~$\alpha$ strictly smaller
than~$\eta/\rho$. Here, $\eta$ is
the Poincar\'e constant of the potential~$V_0$ and~$\rho$ is the
quantity
\begin{equation}\label{eq:definition_rho}
\rho
=-(\inf
\min\mathrm{Spec}(\nabla^2V(x)))\frac{\int_\Rd
  \left(\min\mathrm{Spec}(\nabla^2V(x)) \right)^2\eV}{\left(\int_\Rd
    \min\mathrm{Spec}(\nabla^2V(x))\eV\right)^2}
=c\frac
{\int_\R(12x^2-c)^2\eV}
{\left(\int_\R(12x^2-c)\eV\right)^2}
\end{equation}
appearing in Assumption {\bf (\hypspecbeta{$\eta/\alpha$})}. The real
number $\rho$ can easily be approximated by one-dimensional numerical
integration. Concerning the Poincar\'e constant $\eta$ of $V_0$, let us first
notice that the potential~$V_0$ can be written as the sum of a convex
potential and a bounded perturbation and thus satisfies a Poincar\'e
inequality thanks to the Holley-Stroock
perturbation lemma
(see~\cite[Theorem 3.4.1]{ABC-00}).
The corresponding Poincar\'e constant can be computed numerically,
since it is
the second
eigenvalue of the operator~$L=\partial_x^2-V_0'(x)\partial_x=e^{V_0}\partial_x(e^{-V_0}\partial_x)$, whose
first eigenvalue and eigenvector are $0$ and the constant function
$\mathbf 1$. The numerical method we use to approximate $\eta$ is the
following. First, notice that the spectrum of the operator~$L$ is
identical to the one of~$\tilde
L=e^{-V_0/2} L e^{V_0/2} = e^{V_0/2}  \partial_x(e^{-V_0}\partial_x(e^{V_0/2}\times\cdot))$ which
is self-adjoint in the space~$\LL^2(\rmd x)$. The
operator~$\tilde L$ is then discretized using a regular mesh with constant space step~$\delta x$ by the infinite tridiagonal
matrix~$(M_{i,j})_{i,j\in\mathbb Z}$ defined by
\[
M_{i,i}=
-\frac1{\delta x^2}
\left(
  e^{V(i\delta x)-V((i+1/2)\delta x)}
  +e^{V(i\delta x)-V((i-1/2)\delta x)}
\right)\]
and
\[
M_{i,i+1}
=M_{i+1,i}
=\frac1{\delta x^2}
e^{\frac12V(i\delta x)
+\frac12V((i+1)\delta x)
-V((i+1/2)\delta x)},
\]
(with~$M_{i,j}=0$ whenever~$|i-j|>1$). We consider the restriction to a
finite set of indices~$(M_{i,j})_{-N\leq i,j\leq N}$, which is
equivalent to imposing homogeneous Dirichlet boundary conditions at
$x=-N \delta x$ and $x=N \delta x$. These artificial boundary conditions
are justified (in the limit $N \to \infty$)
by the fact that the eigenvectors of~$\tilde L$ go to~$0$ at infinity.
Since the matrix $(M_{i,j})_{-N\leq i,j\leq N}$ is a nonpositive symmetric matrix, one can
successively compute its first eigenvalues by the inverse power method,
using at each step a projection on the
orthogonal of the eigenvector which have already been computed. We
checked that the numerical approximation obtained for the second
eigenvalue is converged when $\delta x \to 0$ and $N \to \infty$.
The graphs of numerical approximations of both~$\rho$ and the Poincar\'e
constant are plotted on Figure~\ref{fig:poincare_rho}.
In particular, for a curvature constant~$c$ located left to the
intersection of the two
curves (approximately~$c\leq0.86$), Proposition~\ref{prop:variance}
ensures that~$T_t$ is bounded in~$\LL^1$, uniformly in time. Also, for
curvature constants such that~$\rho$ is less than half the Poincar\'e constant
(corresponding approximately to~$c\leq0.50$), $T_t$ is bounded in~$\LL^2$,
and thus has a bounded variance uniformly in time.
On
Figure~\ref{fig:integrabilite_theorique}, we plot
the critical exponent~$\eta/\rho$ such that, according to
Proposition~\ref{prop:variance},~$T_t$ is
in~$\LL^\alpha$ for~$\alpha<\eta/\rho$.

\begin{figure}
\centerline{\epsfig{file=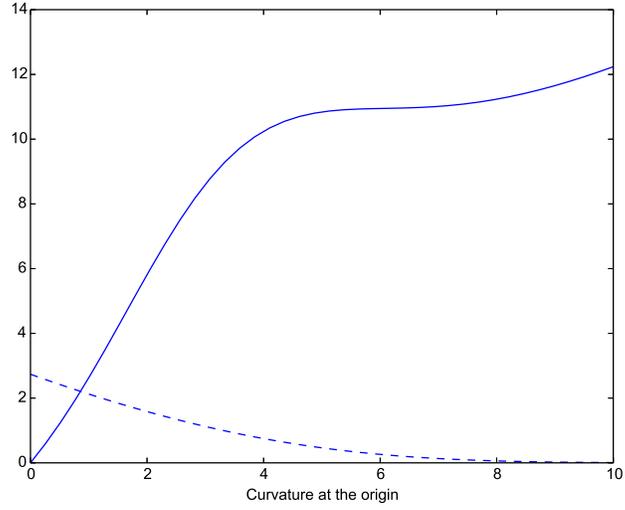,width=100mm,angle=0}}
\caption{\small Poincar\'e constant of the measure~$e^{-V_0}$ (dashed line), and
  parameter~$\rho$ defined in~\eqref{eq:definition_rho} (solid line), as a
  function of the curvature~$c$.}
\label{fig:poincare_rho}
\end{figure}

\begin{figure}
\centerline{\epsfig{file=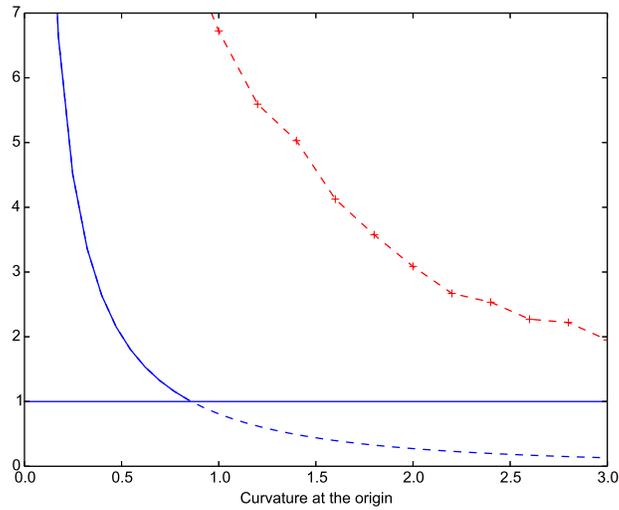,width=100mm,angle=0}}
\caption{\small Blue: theoretical lower bound of integrability for~$T_t$, according
  to Proposition~\ref{prop:variance} as a function of the curvature $c$ (In fact,
  Proposition~\ref{prop:variance} yields boundedness in~$\LL^\alpha$
  only for~$\alpha\geq1$, corresponding to~$c\lesssim0.86$) ; Red :
  Numerical estimation of the integrability exponent $\alpha$ (for $t=40$).}
\label{fig:integrabilite_theorique}
\end{figure}

Let us now explain how we estimate numerically the integrability exponent
$\alpha$ such that~$T_t$ actually is in~$\LL^\alpha$. This is done by
computing the tail of the empirical cumulative distribution function
of~$T_t$. We simulate~$10^6$ independent realizations of the
process~$(T_t,X_t)$, starting from~$\sqrt c/2$ (that is, at the bottom of
the right well), up to the time~$t=40$, at which
the systems seems to be at equilibrium.
On Figure~\ref{fig:integrabilite_empirique}, we plot in logarithmic scale
the tail of the empirical cumulative distribution function of
those~$N=10^6$ independent realizations~$(T_t^i)_{i=1,\hdots, N}$, namely
\[
[0,\infty)\ni x
\mapsto\frac1N\sum_{i=1}^N\mathbf1_{|T_t^i|\geq x}
=\frac1N\sum_{i=1}^N\mathbf1_{T_t^i\leq -x}
\]
with
curvature~$c$ being respectively 2, 3, 4 and 5, from
bottom to top. Linear regression in those four cases gives the
following slopes:
\[
\begin{array}{|c|c|c|c|c|}
  \hline
  c&2&3&4&5\\
  \hline
  {\rm slope}& -3.09 & -1.95  & -1.29 & -1.12 \\
  \hline
\end{array}
\]
We have checked that the results are the same for $t=40$ and for $t=80$.
Note that an integrable random variable corresponds roughly to a slope
less than~$-1$, and a square integrable variable corresponds to a slope
less than~$-2$. We also plot on Figure~\ref{fig:integrabilite_theorique} the
empirical integrability exponent for different curvatures between~$0$ and~$3$.
We observe that the results are in accordance with
Proposition~\ref{prop:variance}: the theoretical lower bound is indeed
smaller than the effective integrability exponent.

 For a curvature larger than 3, the tangent vector
$T_t$ at time $t=40$ does not seem to be of finite variance. This
raises the question of appropriate variance reduction technique to be
used in order to use the estimators~\eqref{eq:estim_1}
or~\eqref{eq:estim_2}. We will investigate in Section~\ref{sec:roland}
a first idea that could be used in this one-dimensional
situation. Further studies related to this problem will be the subject
of future works.

\begin{figure}
\centerline{\epsfig{file=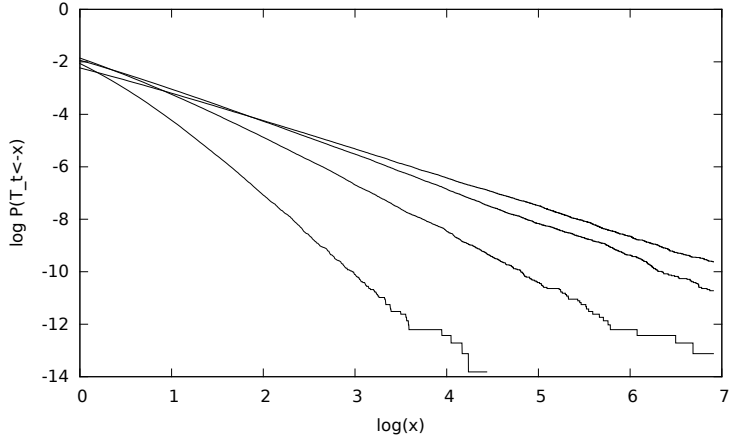,width=100mm,angle=0}}
\caption{\small Logarithmic plot of the empirical cumulative
  distribution function of $10^6$ independent realizations of~$T_t$ (for $t=40$), with parameter (from bottom to top)
  $c=2$, $c=3$, $c=4$, $c=5$.}
\label{fig:integrabilite_empirique}
\end{figure}

\subsection{A many particle system}\label{sec:WA}

In this section, we consider a more complex problem introduced
in~\cite{warren-allen-12}, and motivated by experimental studies of
colloidal particles in optical traps. Let us consider $X^\lambda_t=(Y^1_t,\cdots,Y^N_t)$ where $(Y^i_t)_{i=1\,\hdots,N}\in(\R^2)^N$ are the positions of~$N$ two-dimensional particles evolving according to
\begin{equation}\label{eq:warren-allen}
\rmd Y_t^i
= -\kappa Y_t^i\rmd t 
+ \sum_{j=1}^N\nabla U(Y_t^i-Y_t^j)\rmd t 
+ \lambda e_1 \rmd t
+ \rmd W_t^i,\;1\leq i\leq N
\end{equation}
with~$\kappa>0$,~$\lambda\in\R$,~$e_1$ the normed vector directed
along the first coordinate, and $U(x)=\Gamma e^{-|x|}/|x|$.
The particles undergo a quadratic confining potential near the
origin with strength~$\kappa$, a repulsive
interaction given by~$U$, a shear in the $x$-direction with
strength~$\lambda$ and a thermal noise.

We study the case of~$N=10$ particles with repulsion range~$\Gamma=25$
and attraction intensity~$\kappa=10$ corresponding to the
parameters studied in~\cite{warren-allen-12}. For those parameters, at
equilibrium, particles are gathered around the origin. At~$\lambda=0$, no
particular direction appears in the dynamics, and the equilibrium measure
is invariant with respect to rotations around the origin.

One wants to study the effect of shearing on the symmetry of the
invariant measure. This symmetry can be measured by the
empirical covariance~$\Phi$ of the particle system, defined by
\[
\Phi(X^\lambda_t)=\frac1N\sum_{i=1}^N(Y_t^{i,1}-\bar Y_t^1)(Y_t^{i,2}-\bar Y_t^2),
\]
where~$Y_t^i=(Y_t^{i,1},Y_t^{i,2})$ and
for~$k\in\{1,2\}$,~$\bar Y_t^k=\frac1N\sum_{i=1}^NY_t^{i,k}$.
One is interested in computing the derivative~$\dlam \int_{(\R^2)^N}\Phi\rmd\pi_\lambda$.

On Figure~\ref{fig:warren_allen}, we plot the confidence interval
obtained for the
expectation~$\E[\dlam\Phi(X_t^\lambda)]$, with~$N=10^5$ independent simulations, as a
function of the time~$t$. The dynamics~\eqref{eq:warren-allen} has been simulated
using an explicit Euler-Maruyama scheme with time step~$\delta t=10^{-5}$, and
the expectation has been calculated through the Monte Carlo approximation
\[
\E[\dlam\Phi(X_t^\lambda)]
=\E[T_t\cdot\nabla\Phi(X_t^0)]
\simeq\frac1N\sum_{i=1}^NT_t^i\cdot\nabla\Phi(X_t^{0,i})
\]
where the~$(X_t^{0,i},T_t^i)_{1\leq i\leq N}$ are independent
simulations of the Euler-Maruyama discretization of the
dynamics~\eqref{eq:EDS_extended} ruling the evolution of~$(X_t^0,T_t)$. As
in~\cite{warren-allen-12}, we observe that the correlation function
$\E[\dlam\Phi(X_t^\lambda)]$ increases as a function of time, before
reaching a plateau. We have checked that similar results are obtained
using a finite differenciation instead of the simulation of the couple
$(X^0_t,T_t)$.

\begin{figure}
\centerline{\epsfig{file=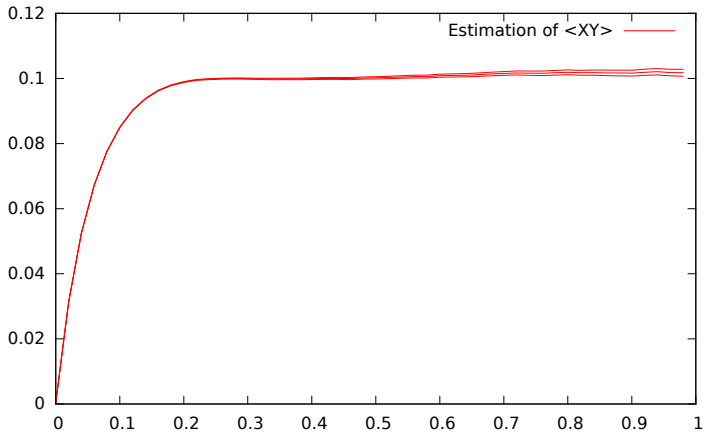,width=100mm,angle=0}}
\caption{\small Sensitivity of the covariance with respect to the shear.}
\label{fig:warren_allen}
\end{figure}

\subsection{Particle merging}\label{sec:roland}

As mentioned in Section~\ref{sec:1d}, in some situations, the variance
of the tangent vector may become very large (or even infinite) which
means that the estimators~\eqref{eq:estim_1} and~\eqref{eq:estim_2}
become ineffective. Therefore, it is desirable to introduce variance
reduction mechanisms. We explore in this section a first idea in the
simple one-dimensional test case of Section~\ref{sec:1d}. Extensions
and further variance reduction techniques will be the subject of
forthcoming works.

A first simple idea to reduce the variance is to replace the tangent
vector $T_s$  in the estimator of $\E[\nabla
f(X_s^0)\cdot T_s]$ by its conditional
expectation given $X^0_s$. This corresponds in practice to replacing the
tangent vector of particles which are at the same position at
a given time $s$ by the
average of their tangent vectors. Then, the particles evolve again
following the dynamics~\eqref{eq:EDS_extended}.
We refer to this procedure as ``particle
merging''. 
In practice, with this naive procedure the probability to observe two
particles at the same position is zero, in dimension larger than one. A
first  simple practical way to implement this technique is to introduce
small subsets of the configuration space, and to merge particles which
are in the same subset,  which of course reduces the variance but
introduce some bias.
The merging can be performed in a much efficient way and in larger
dimensions,  by correlating the particles, see e.g. \cite{kalos-pederiva-00}.
This will be the scope of future work.
Before studying
the interest of particle merging on the simple case of
Section~\ref{sec:1d}, let us first state the theoretical result which
justifies the use of this approach.

\begin{lem}
Assume {\bf(\hypspec)}. For $s\ge 0$, let~$(\tilde T_t)_{t\geq s}$ be the solution to
\[
\left\{
\begin{aligned}
\frac{\rmd \tilde T_t}{\rmd t}&=\dlam
F_\lambda(X_t^0)-\nabla^2(X_t^0)\tilde T_t \text{ , for all }~t\geq s,  \\
\tilde T_s&=\E[T_s|X_s^0].
\end{aligned}
\right.
\]
Then $\forall t\geq s,\;\tilde T_t=\E[T_t|X^0_s,(W_r-W_s)_{r\in[s,t]}]$. Assume moreover that~$f:\R^d\to\R^d$ is a Lipschitz function belonging to $\LL^p(\eV)$ for some $p\in[1,\infty]$ and that the initial condition $X_0$ to \eqref{eq:EDS_lam} admits a density with respect to $\pi_0$ belonging to $\LL^{\frac{p}{p-1}}(\eV)$ (where, by convention, $\frac{p}{p-1}=\infty$ if $p=1$). Then, for each $t\geq 0$, $f(X_t^\lambda)$ is integrable for $\lambda\in[0,\lambda_0]$, $\lambda\mapsto \E[f(X_t^\lambda)]$ is differentiable at $\lambda=0$ and $$\dlam\left(\E[f(X_t^\lambda)]\right)=\E[\nabla f(X_t^0)\cdot T_t]=\E[\nabla f(X_t^0)\cdot \tilde T_t],\mbox{ for each }t\ge s.$$
\end{lem}
This Lemma shows that if, at a given time $s$, the particles at
position $X^0_s$ replace their current tangent vectors by an average
of these tangent vectors, and then follow the
dynamics~\eqref{eq:EDS_extended} for $t \ge s$, the
estimator~\eqref{eq:estim_2} is still consistent.
\begin{proof}
By Lemma \ref{lem:R_T_Linfty}, Assumption {\bf(\hypspec)} ensures that $T_t$ is integrable for each $t\geq 0$. In view of the equality~\eqref{eq:expression_Tt} and using the
semigroup property \eqref{eq:semigroupe_R} of~$R_{X^0}$, one gets that for $t\geq s\geq 0$, 
$$T_t=R_{X^0}(s,t)T_s+\int_s^tR_{X^0}(r,t)\dlam
F_\lambda(X_r^0)dr.$$
Since $(X^0_r)_{r\in[s,t]}$ and therefore $(R_{X^0}(r,t))_{r\in[s,t]}$ are measurable with respect to the sigma-field generated by $X^0_s$ and $(W_r-W_s)_{r\in[s,t]}$, one deduces that
$$\E[T_t|X^0_s,(W_r-W_s)_{r\in[s,t]}]=R_{X^0}(s,t)\E[T_s|X^0_s,(W_r-W_s)_{r\in[s,t]}]+\int_s^tR_{X^0}(r,t)\dlam
F_\lambda(X_r^0)dr.$$
The independence of  $(X^0_s,T_s)$ and $(W_r-W_s)_{r\in[s,t]}$ implies that $\E[T_s|X^0_s,(W_r-W_s)_{r\in[s,t]}]=\E[T_s|X^0_s]$. Since, by an adaptation of Proposition \ref{prop:tangent}, $$\tilde{T}_t=R_{X^0}(s,t)\E[T_s|X^0_s]+\int_s^tR_{X^0}(r,t)\dlam
F_\lambda(X_r^0)dr,$$ one concludes that $\tilde{T}_t=\E[T_t|X^0_s,(W_r-W_s)_{r\in[s,t]}]$. 

If the initial condition $X_0$ to \eqref{eq:EDS_lam} admits a density with respect to $\pi_0$ belonging to $\LL^{\frac{p}{p-1}}(\eV)$, then so does $X^0_t$ for each $t\geq 0$ by Lemma \ref{lem:densite_X_t}. When~$f:\R^d\to\R^d$ is a Lipschitz function belonging to $\LL^p(\eV)$, the integrability of $\E[f(X_t^\lambda)]$, the differentiability of $\lambda\mapsto \E[f(X_t^\lambda)]$ at $\lambda=0$ and the equality $\dlam\left(\E[f(X_t^\lambda)]\right)=\E[\nabla f(X_t^0)\cdot T_t]$ are deduced from an adaptation of the beginning of the proof of Theorem \ref{theo:GK_tps_fini}.
Now, for $t \ge s$, \begin{align*}
   \E[\nabla f(X_t^0)\cdot T_t]&=\E[\E[\nabla f(X_t^0)\cdot T_t|X^0_s,(W_r-W_s)_{r\in[s,t]}]]\\&=\E[\nabla f(X_t^0)\cdot\E[T_t|X^0_s,(W_r-W_s)_{r\in[s,t]}]]=\E[\nabla f(X_t^0)\cdot\tilde{T}_t].
\end{align*}

\end{proof}

To test the interest of this approach, we consider again the setting
of Section~\ref{sec:1d} with $c=2.9$ (which corresponds to case where
the variance of tangent vector $T_t$, at $t=40$, is very large,
see Figure~\ref{fig:integrabilite_theorique}). The merging procedure is done
as follows: a uniform mesh with step size $0.04$ is introduced, and,
every ten timesteps, the tangent vectors of particles
which are in the same bin are replaced by an average of these tangent
vectors. On Figure~\ref{fig:int_conf}, we observe that this procedure
divides approximately the variance by 4, while introducing
a bias which is sufficiently small so that the confidence interval of
the simulation with merging is included in the confidence interval of
the simulation without merging. Figure~\ref{fig:variance_merging} then
gives more quantitative estimates of the variances of these two
simulations (with and without merging), as a function of time. We have
observed numerically that large values of $T_t$ become very unlikely
with the merging procedure: using $10^3$ independant realizations of
$10^3$ interacting particles over the time interval $(0,10)$, we did
not observe any realization of $T_t$ with absolute value larger than
$3$ (compare with what is reported on Figure~\ref{fig:integrabilite_empirique}).

\begin{figure}
\centerline{\epsfig{file=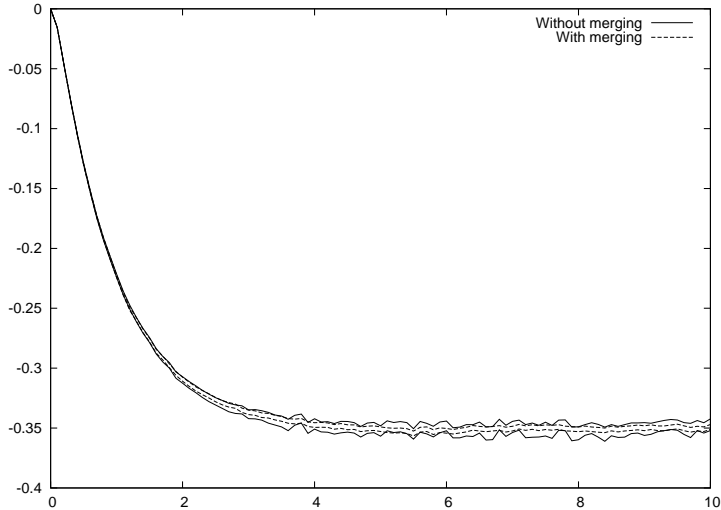,width=100mm,angle=0}}
\caption{\small Empirical average and 95\%-confidence interval for    the Monte Carlo estimator of $\partial_\lambda
\E[f(X_t^\lambda)]$ as a function of time, $f$ being a smooth
approximation of the
the indicator
function of $\R_+$: $f(x) = \frac12 + \frac1\pi \arctan(10x)$. The estimator is built with $10^6$
realizations. Green: with merging  ($10^3$ independent relizations of
$10^3$ interacting particles); Red:
without merging ($10^6$ independent relizations).}
\label{fig:int_conf}
\end{figure}

\begin{figure}
\centerline{\epsfig{file=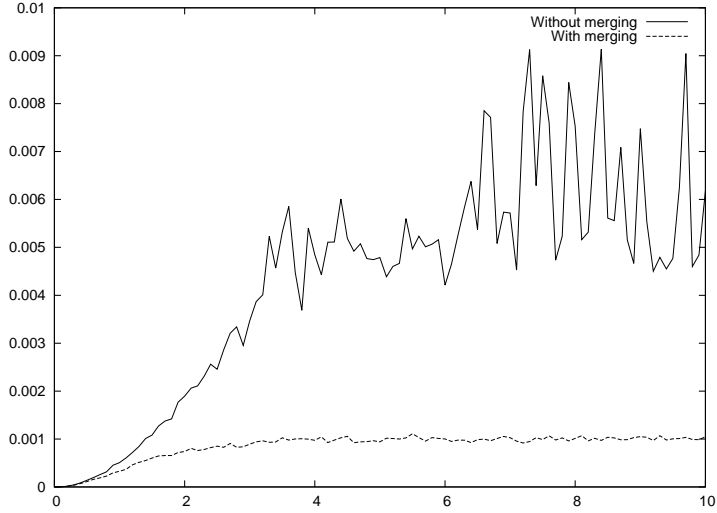,width=100mm,angle=0}}
\caption{\small Empirical variance for the Monte Carlo estimator of $\partial_\lambda
\E[f(X_t^\lambda)]$ as a function of time, $f$ being the indicator
function of $\R_+$. Green: with merging ; Red:
without merging. }
\label{fig:variance_merging}
\end{figure}

\appendix

\section{Alternative bounds on the density of $X^0_t$}\label{sec:hypdelta}

In this section, we would like to present a few results which can be
obtained under the assumption
\begin{hyp}[\hypdelta]
The function~$V$ is of
class~$\mathcal {\mathcal C}^2$ and satisfies
\[
C_V=\sup_{x\in\Rd}(2\Delta V(x)-|\nabla V(x)|^2)<+\infty.
\]
\end{hyp}
Note that simple assumptions on the quantity~$a_V(x)=2\Delta V(x)-|\nabla V(x)|^2$
can give strong results on the equilibrium measure~$\eV$. For
instance, if~$a_V(x)$ goes to~$-\infty$ at infinity, then the
equilibrium measures satisfies a Poincar\'e inequality (see for example
the appendix in~\cite{villani-09}).

\subsection{Bounds on the density of $X^0_t$}

\begin{prop}
 Consider the setting and the notation of Lemma~\ref{lem:densite_X_t} and let~Assumption~{\bf (\hypdelta)} hold. Assume that the measure~$e^{\frac12V} \rmd \mu_0$ can be written as
  \[
  e^{\frac12V(x)} \rmd \mu_0 =f(x)\rmd x+\rmd \nu,
  \]
  where~$f$ is some function in~$\LL^p(\rmd x)$ with $p \in [1,2]$ and~$\nu$ is some finite
  measure on~$\Rd$.
  Then, for any~$t>0$,~$\mu_t$ is absolutely continuous with respect
  to~$e^{-\frac12V(x)}\rmd x$ with
  \begin{equation}\label{eq:majonorml2}
    \left\|\frac{\rmd \mu_t}{e^{-\frac{1}{2}V(x)} \rmd x}\right\|_{\LL^2(\rmd x)}
    \leq Ce^{Ct}\left(\frac{C_p}{ t^{(1/p-1/2)d/2}}+\frac{\nu(\Rd)}{t^{d/4}}\right).
  \end{equation}
\end{prop}

\begin{proof}
Let $\psi:\R^d\to\R$ be a bounded measurable function and recall the formula
$$  \E[\psi(X_t^0)]
=\E\left[
    \psi(X_0+\sqrt2W_t)
    e^{
      -\frac12V(X_0+\sqrt2W_t)
    }
    e^{ \frac12V(X_0)}
    e^{
        \frac14\int_0^t\left(2\Delta V-|\nabla V|^2\right)(X_0+\sqrt2W_s)\rmd s
      }
  \right]
$$
obtained by the Girsanov theorem, see Equation~\eqref{eq:girsanov}.

If~$C$ is an upper bound for~$\frac14(2\Delta V-|\nabla V|^2)$, and
if one assumes~$\psi\geq0$, one obtains
\begin{align*}
\int_\Rd\psi\rmd\mu_t
  &\leq
  e^{Ct}\E\left[
    \psi(X_0+\sqrt2W_t)
    e^{
      -\frac12V(X_0+\sqrt2W_t)
    }
    e^{ \frac12V(X_0)}
  \right]\\
&=\frac{e^{Ct}}{(4\pi t)^{d/2}}\int_\Rd
\psi(y)
e^{-\frac12V(y)}
\left(\int_\Rd e^{\frac12V(x)}
e^{-\frac{|y-x|^2}{4t}}\rmd\mu_0(x)\right)
\rmd y.
\end{align*}
so that~$\mu_t\ll e^{-\frac12V(x)}\rmd x$ with a Radon-Nikodym
derivative~$\frac{\rmd \mu_t}{e^{-\frac12V(x)}\rmd x}$ satisfying
\begin{align*}
\frac{\rmd \mu_t}{e^{-\frac12V(x)}\rmd x}
&\leq e^{Ct}\left(e^{\frac12V}\mu_0\right)*\gamma_t=e^{Ct}\left(f*\gamma_t+\nu*\gamma_t\right),
\end{align*}
where~$*$ stands for the convolution product, and~$\gamma_t(x)=\frac{1}{(4 \pi
  t)^{d/2}} e^{-\frac{|x|^2}{4t}}$ denotes the centered Gaussian density with covariance matrix~$2tI_d$. One concludes
that~\eqref{eq:majonorml2} holds by:
\begin{itemize}
\item the Young inequality
$\|f * \gamma_t\|_{\LL^2(\rmd x)} \le \|f\|_{\LL^p(\rmd x)} \|\gamma_t\|_{\LL^q(\rmd x)}$
where $1/p+1/q =3/2$ ($p,q \in [1, \infty]$) and the heat kernel
estimate $\|\gamma_t\|_{\LL^q(\rmd x)} \le C_q
t^{-\left(1-\frac{1}{q}\right) \frac{d}{2}}$;
\item the estimate $\|\nu*\gamma_t\|^2_{\LL^2(\rmd x)}
\leq\|\nu*\gamma_t\|_{\LL^1(\rmd x)}\|\nu*\gamma_t\|_{\LL^\infty(\rmd x)}
\leq\frac{(\nu(\Rd))^2}{(4\pi t)^{d/2}}$.
\end{itemize}

\end{proof}

\subsection{An additional result}

Assumption~{\bf(\hypdelta)} can also be useful to prove the second point
in Assumption~{\bf(\hypV)}-$(ii)$ on the potential $V$.
\begin{lem}\label{lem:nablaV_L2}
Under Assumption~{\bf(\hypdelta)}, the function~$\nabla V$ is
in~$\LL^2(\eV)$:
\[
\int_\Rd|\nabla V|^2(x)\eV
\leq C_V.
\]
\end{lem}
\begin{proof}
Let~$\chi_n(x)=\chi(x/n)$ where~$\chi$ is a smooth,~$[0,1]$-valued, cutoff function such that~$\chi(x)=1$
for~$|x|<1$ and~$\chi(x)=0$ for~$|x|>2$.
\begin{align*}
\int_\Rd|\nabla V|^2(x)\chi_n(x)\eV
&=-\int_\Rd\left(\chi_n(x)\nabla V(x)\right)\cdot\nabla\eV\\
&=\int_\Rd\nabla\chi_n(x)\cdot\nabla V(x)\eV
+\int_\Rd\chi_n(x)\Delta V(x)\eV\\
&\leq\int_\Rd\Delta\chi_n(x)\eV
+\frac12\int_\Rd\chi_n(x)|\nabla V|^2(x)\eV+\frac12C_V.
\end{align*}
As a consequence,
\[
\int_\Rd|\nabla V|^2(x)\chi_n(x)\eV
\leq2\int_\Rd\Delta\chi_n(x)\eV+C_V
\]
and the result follows from taking~$n\to\infty$ by Fatou's lemma for the
left-hand side and Lebesgue's theorem for the right-hand side.
\end{proof}

\section{About the Assumption~{\bf(\hypconv)}}\label{sec:annex_conv}

In this section, we show that Assumption~{\bf(\hypconv)} is a natural
one, since it appears as a sufficient condition in another related problem.

We recall that~$(Y^x_t)_{t\geq 0}$ is defined in~\eqref{eq:Ytx} as
the solution to
\begin{equation}\label{eq:Ytx_prime}
\forall t\geq0,~
Y^x_t
=x-\int_0^t\nabla V(Y^x_s)\rmd s
+\sqrt2W_t.
\end{equation}
Since~$(R_{Y^x}(0,t))_{t\geq0}$ is the differential of the trajectory~$(Y^x_t)_{t\geq0}$
with respect to~$x$, we expect that a condition yielding long-time decay
for~$R_{Y^x}$ will imply that trajectories with same noise and close initial conditions will
eventually converge toward each other.
More precisely, we are interested in the joint long-time behavior
of the so-called duplicated dynamics~$(Y_t^x,Y_t^y)_{t\geq0}$, where~$x$
and~$y$ are two different initial
conditions. Note here that the two processes~$(Y^x_t)_{t\geq0}$
and~$(Y^y_t)_{t\geq0}$ are driven by the same Brownian motion.

In~\cite{lemaire-pages-panloup-13}, the same problem is considered for a
diffusion whose diffusion matrix may not be constant. In that case,
an example is provided, where the process~$Y_t^x-Y_t^y$ does not
converge to~$0$.

A similar problem is considered in~\cite{burdzy-chen-jones-06}: the
process is a Brownian motion reflected on the boundary of a
domain~$\Omega$. Such a dynamics can be formally seen as a singular case
of the problem we consider,
with~$V=\infty\times\mathbf1_{\Omega^c}$. Equation~\eqref{eq:Ytx_prime}
then has to be written with a local time on the boundary in place
of~$\nabla V$. In that case, the difference~$Y_t^x-Y_t^y$
will converge to~$0$ if the domain~$\Omega$ is smooth enough and has at
most one hole. However, it is conjectured that the same result holds for
much more general domains.

We will use the fact that $V$ is such that the dynamics~\eqref{eq:Ytx_prime} is ergodic with
respect to the invariant measure~$\eV$.

\subsection{The one-dimensional case}
In the one-dimensional case, this question is especially
simple, because of the order structure on the state space. In
particular (see~\cite{lemaire-pages-panloup-13}), it can be checked
that if for any~$x\in\R$, $Y_t^x$ converge weakly to~$\pi_0$
as~$t\to\infty$, then  the only invariant distribution of the duplicated
dynamics is the image of~$\eV$ by~$x\mapsto(x,x)$. Actually, under
additional assumption, one can show that $Y^x_t-Y^y_t$ converges in mean to 
$0$ in the long-time limit.

\begin{prop}
Assume that the dimension is~$d=1$. If for any~$x\in\R$ the time marginals of
the process~$(Y_t^x)_{t\geq0}$ converge weakly to~$\pi_0$ as~$t\to\infty$ and  the random
variables~$(Y_t^x)_{t\geq0}$ are uniformly integrable, then, for any~$x,y\in\R$, the
process~$(Y_t^x-Y_t^y)_{t\geq0}$ converges to~$0$ in~$\LL^1(\Omega)$.
\end{prop}
According to Corollary \ref{cor:CV_L1}, the long-time convergence of the marginals holds for
instance if the potential~$V$ satisfies a Poincar\'e
inequality (see Assumption~{\bf(\hyppoinc{$\eta$})}).
\begin{proof}
First, from the uniform integrability
of~$(X_t^x)_{t\geq0}$ and the weak convergence of the time marginals,
both~$\E[Y_t^x]$ and~$\E[Y_t^y]$ converge to~$\int_\R x\eV$
as~$t\to\infty$. 

Now assume, without loss of generality that~$x\leq y$. Then, from a
comparison theorem,~$Y_t^x\leq Y_t^y$ holds for
all positive times, and one obtains
\[
\E[|Y_t^x-Y_t^y|]
=\E[Y_t^x-Y_t^y]
=\E[Y_t^x]-\E[Y_t^y]
\to0.
\]
\end{proof}

\subsection{A general criterion}
\begin{prop}
  The following facts hold true:
\begin{enumerate}
\item
  Assume that
  \begin{equation}\label{eq:minoconv}
    \forall x,y\in\Rd,
    \;(x-y)\cdot(\nabla V(x)-\nabla V(y))\geq \frac{v(x)+v(y)}{2}|x-y|^2
  \end{equation}
  with~$v:\Rd\to\R$ such that~$\int_\Rd\max(0,-v(x))\eV<\infty$ and~$\int_\Rd v(x)e^{-V(x)}\rmd x>0$. Then for
  all~$x,y\in\Rd$,~$|Y^x_t-Y^y_t|$ converges {\it a.s.} to~$0$, exponentially
  fast at any rate between~$0$ and~$\int_\Rd v(x)\eV$ as~$t\to\infty$.
\item
  The exponential convergence to~$0$ still holds if~$V$ is convex and there
  exist~$x_0\in\Rd$ and~$\epsilon>0$ such that the
  inequality~$\inf_{x\in B(x_0,\epsilon)}\min{\rm Spec}(\nabla^2V(x))>0$
  holds.
\item
  If~$V$ is
  convex,
  then the only invariant measure of the
  duplicated dynamics is the image of~$\eV$
  by~$x\mapsto (x,x)$.
\end{enumerate}
\end{prop}

Let us start with a few remarks:
\begin{rem}
  \begin{itemize}
  \item
    The first point can be applied to the so-called Mexican hat
    potential~$V(x)=\beta(|x|^4-\gamma |x|^2)$, with~$\beta>0$
    and~$\gamma>0$, in dimension $d\geq 2$.
    For this potential, one has
    \begin{align*}
      (x-y)\cdot(\nabla V(x)-\nabla V(y))
      &=2\beta|x-y|^2(|x|^2+|y|^2-\gamma)+2\beta(|x|^2-|y^2|)^2\\
      &\geq\frac{v(x)+v(y)}{2}|x-y|^2,
    \end{align*}
    for~$v(x)=\beta(4|x|^2-2\gamma)$.
    In addition, one has~$\int_\Rd v(x)\eV>0$ since
    \begin{align*}
      \frac{\int_{\R^d}|x|^2e^{\beta(\gamma
          |x|^2-|x|^4)}\rmd x}{\int_{\R^d}e^{\beta(\gamma
          |x|^2-|x|^4)}\rmd x}
      &=\frac{\int_0^{+\infty}r^{\frac{d}{2}}e^{\beta(\gamma
          r-r^2)}dr}{\int_0^{+\infty}r^{\frac{d}{2}-1}e^{\beta(\gamma
          r-r^2)}dr}\\
      &=\frac{\gamma}{2}+\frac{\int_0^{+\infty}r^{\frac{d}{2}-1}(r-\frac{\gamma}{2})e^{\beta(\gamma r-r^2)}dr}{\int_0^{+\infty}r^{\frac{d}{2}-1}e^{\beta(\gamma r-r^2)}dr}\\
      &=\frac{\gamma}{2}+\frac{1_{\{d=2\}}}{2\beta
        \int_0^{+\infty}e^{\beta(\gamma
          r-r^2)}dr}+1_{\{d>2\}}\frac{(d-2)\int_0^{+\infty}r^{\frac{d}{2}-2}e^{\beta(\gamma
          r-r^2)}dr}{4\beta\int_0^{+\infty}r^{\frac{d}{2}-1}e^{\beta(\gamma
          r-r^2)}dr}\\
      &>\frac{\gamma}{2}.
    \end{align*}
  \item
    Letting~$y\to x$ in~\eqref{eq:minoconv}, one obtains
    that~$\forall x\in\Rd$,~$v(x)\leq \min{\rm Spec}(\nabla^2V(x))$.
    When~$x\mapsto\min{\rm Spec}(\nabla^2V(x))$ is concave,
    \begin{align*}
      (x-y) \cdot (\nabla V(x)-\nabla V(y))
      &=\int_0^1(x-y)\cdot\nabla^2V(\theta x+(1-\theta)y)(x-y)\rmd\theta\\
      &\geq |x-y|^2\int_0^1\min{\rm Spec}(\nabla^2V(\theta x+(1-\theta)y))\rmd\theta\\
      &\geq |x-y|^2\int_0^1\theta\min{\rm Spec}(\nabla^2V(x))+(1-\theta)\min{\rm Spec}(\nabla^2V(y))\rmd\theta\\
      &\geq \frac12\left(\min{\rm Spec}(\nabla^2V(x))+\min{\rm Spec}(\nabla^2V(y))\right)|x-y|^2
    \end{align*}
    and one may choose~$v(x)=\min{\rm Spec}(\nabla^2V(x))$ in~\eqref{eq:minoconv}.
  \item
    When~$V=\bar{V}+\hat{V}$ with~$\bar{V}$ such
    that~$x\mapsto\min{\rm Spec}(\nabla^2\bar{V}(x))$ is concave
    and~$\hat{V}$ such that~$x\mapsto \nabla\hat{V}(x)$ is Lipschitz
with constant~$\delta$ and constant outside some Borel
    subset~$A$ of~$\Rd$, then one may
    choose~$v(x)=\min{\rm Spec}(\nabla^2\bar{V}(x))-2\delta 1_A(x)$ in~\eqref{eq:minoconv}.
  \end{itemize}
\end{rem}
\begin{proof}
  \begin{enumerate}
  \item   
    One has
    \begin{align}\label{eq:decroissance}
      \rmd|Y^x_t-Y^y_t|^2
      &=-2(Y^x_t-Y^y_t)\cdot(\nabla V(Y^x_t)-\nabla V(Y^y_t))\rmd t\\
      &\leq -(v(Y^x_t)+v(Y^y_t))|Y^x_t-Y^y_t|^2\rmd t,\nonumber
    \end{align}
    under~\eqref{eq:minoconv}.
    Hence
    \[
    |Y^x_t-Y^y_t|^2
    \leq |x-y|^2e^{-\int_0^t(v(Y^x_s)+v(Y^y_s))\rmd s}.
    \]
    Since, by  \eqref{eq:erggen},~$\frac1t\int_0^t(v(Y^x_s)+v(Y^y_s))\rmd s$ converges
    {\it a.s.}
    to~$2\int_\Rd v(x)\eV>0$,
    one easily deduces the first assertion.
  \item
    When~$V$ is convex, then~$t\mapsto |Y^x_t-Y^y_t|$ is
    nonincreasing by~\eqref{eq:decroissance}. Now, for~$z\in B(x_0,\frac\epsilon2)$ and~$w\in\Rd$,
    one has
    \[
    (z-w)\cdot(\nabla V(z)-\nabla V(w))
    \geq|z-w|\inf_{B(x_0,\epsilon)}\min{\rm Spec}(\nabla^2V(\cdot))
    \left(
      \frac\epsilon21_{B(x_0,\epsilon)^c}(w)+|z-w|1_{B(x_0,\epsilon)}(w)
    \right).
    \]
    As a consequence,
    \[
    (Y^x_t-Y^y_t)\cdot(\nabla V(Y^x_t)-\nabla V(Y^y_t))
    \geq1_{B(x_0,\frac{\epsilon}{2})}(Y^x_t)
    \inf_{B(x_0,\epsilon)}\hspace{-3pt}\min{\rm Spec}(\nabla^2V(\cdot))
    \left(\frac\epsilon{2|x-y|}\wedge1\right)|Y^x_t-Y^y_t|^2.
    \]
    One concludes by arguments similar to the ones used for the first assertion.
  \item
    Let~$V$ be convex and differentiable and let~$x\neq y$ be such
    that~$(x-y)\cdot(\nabla V(x)-\nabla V(y))=0$.  Then~$V$ is affine on the
    segment~$[x,y]$
    and~$V(\frac{x+y}{2})=\frac{V(x)+V(y)}{2}$. For~$z\in\Rd\setminus\{0\}$
    and~$\epsilon\in \R$, 
    \begin{align*}
      \frac{V(x)+V(y)}{2}
      =V\left(\frac{x+y}2\right)
      &\leq \frac{V(x+\epsilon z)+V(y-\epsilon z)}{2}\\
      &=\frac{V(x)+V(y)}{2}+\frac{\epsilon z}{2}\cdot(\nabla V(x)-\nabla V(y))+o(\epsilon)
    \end{align*}
    as~$|\epsilon|\to 0$. As a
    consequence~$z\cdot(\nabla V(x)-\nabla V(y))=0$ and~$\nabla V(x)=\nabla V(y)$.
    
    Let~$(X_t)_{t\geq0}$ and~$(Y_t)_{t\geq0}$ be two solutions to the
    stochastic differential equation~\eqref{eq:EDS}, such
    that~$(X_0,Y_0)$ is distributed according to some invariant probability
    measure of the duplicated dynamics. Since~$|X_t-Y_t|^2$ is
    {\it a.s.} non-increasing with~$t$ and constant in distribution,
    {\it a.s.}~$t\mapsto |X_t-Y_t|^2$ is constant and therefore~$\rmd t$-{\it a.e.}~$(X_t-Y_t)\cdot(\nabla V(X_t)-\nabla V(Y_t))=0$ which
    implies~$\nabla V(X_t)=\nabla V(Y_t)$. One deduces that {\it
      a.s.},~$t\mapsto X_t-Y_t$ is constant.\\
    Now, since~$x\mapsto e^{-V(x)}$ is integrable, then~$V$ cannot be affine
    in some direction and for any~$z\in\Rd\setminus\{0\}$,~$x\mapsto
    z\cdot(\nabla V(x)-\nabla V(x-z))$ is not constant equal to zero. By
    continuity of~$\nabla V$, one deduces the existence of~$y\in\Rd$
    and~$\epsilon>0$ such that~$\forall x\in B(y,\epsilon)$,~$z\cdot(\nabla
    V(x)-\nabla V(x-z))>0$. With the ergodicity of~$(X_t)_{t\geq 0}$ and the fact
    that~$\rmd t$-{\it a.e.}~$(X_0-Y_0)\cdot(\nabla V(X_t)-\nabla
    V(X_t-X_0+Y_0))=0$, one concludes that {\it a.s.}~$X_0=Y_0$.
\end{enumerate}
\end{proof}

\section*{Acknowledgements}
This work is supported by the European Research
Council under the European Union's Seventh Framework Programme
(FP/2007-2013) / ERC Grant Agreement number 614492 and by the French National Research Agency under the grant ANR-12-BS01-0019 (STAB).
The authors would
like to thank fruitful discussions with G. Stoltz on nonequilibrium
methods and Green-Kubo formulae.

\bibliographystyle{plain}
\bibliography{./computation_sensitivities.bib}

\end{document}